\documentclass[11pt]{article}%
\usepackage{amsfonts}
\usepackage{amsmath}
\usepackage{graphicx}
\usepackage{amssymb}%
\providecommand{\U}[1]{\protect\rule{.1in}{.1in}}
\providecommand{\U}[1]{\protect\rule{.1in}{.1in}}
\newtheorem{theorem}{Theorem}

\newtheorem{definition}[theorem]{Definition}

\newtheorem{lemma}[theorem]{Lemma}

\newtheorem{proposition}[theorem]{Proposition}
\newtheorem{remark}[theorem]{Remark}

\newenvironment{proof}[1][Proof]{\noindent\textbf{#1.} }{\ \rule{0.5em}{0.5em}}

\begin{document}

\title{Thin instantons in $G_{2}$-manifolds and \ Seiberg-Witten invariants}
\author{Naichung Conan Leung, Xiaowei Wang and Ke Zhu}
\maketitle

\begin{abstract}
For two nearby disjoint coassociative submanifolds $C$ and $C^{\prime}$ in a
$G_{2}$-manifold, we construct thin instantons with boundaries lying on $C$
and $C^{\prime}$ from regular $J$-holomorphic curves in $C$. We explain their
relationship with the Seiberg-Witten invariants for $C$.

\end{abstract}
\tableofcontents

\section{\bigskip Introduction}

Intersection theory of Lagrangian submanifolds is an essential part of
symplectic geometry. By counting the number of holomorphic disks bounding
intersecting Lagrangian submanifolds, Floer and others defined the celebrated
Floer homology theory. It plays an important role in mirror symmetry for
Calabi-Yau manifolds and string theory in physics. In M-theory, Calabi-Yau
threefolds are replaced by seven dimensional $G_{2}$-manifolds $M$ (i.e.
oriented Octonion manifolds \cite{Leung RG over A}). The analogs of
holomorphic disks (resp. Lagrangian submanifolds) are instantons or
associative submanifolds (resp. coassociative submanifolds or branes) in $M$
\cite{Lee Leung Instanton Brane}. In \cite{GayetWitt}, the Fredholm theory for
instantons with coassociative boundary conditions has been set up. However,
existence of instantons is still a difficult problem. As a first step, we want
to give a construction modeled on the work of Fukaya and Oh \cite{Fukaya Oh}
in symplectic geometry. As\ it was shown in \cite{Fukaya Oh}, if we choose two
nearby Lagrangian submanifolds in such a way that one is the graph of a closed
one form on the other then the holomorphic disks bounding two is closely
related to gradient flow lines of the one form. Searching for the analog in
$G_{2}$-geometry leads us to study the following problem.

\bigskip

\textbf{Problem}: Given two nearby coassociative submanifolds $C$ and
$C^{\prime}$ in a (almost) $G_{2}$-manifold $M$. Relate the number of
instantons in $M$ bounding $C\cup C^{\prime}$ to the Seiberg-Witten invariants
of $C$.

\bigskip

The basic idea is as follows: When the coassociative submanifold $C^{\prime
}\ $is sufficiently close to $C$, then it is the graph of a self-dual two form
on $C$. This two form is essentially a symplectic form on $C$\ away from the
intersection $C\cap C^{\prime}$. Instantons bounding $C\cup C^{\prime}$ would
become \emph{holomorphic curves} on $C$ when $C^{\prime}$ collapses onto $C$
modulo the possible \emph{bubblings}. By the seminal work of Taubes \cite{Ta
Gr Sw} on the equivalence of Gromov-Witten and Seiberg-Witten invariants, we
expect that the number counted with algebraic weights of such instantons is
given by the Seiberg-Witten invariants of $C$.

Settling the above problem completely{\footnotesize \ }is very difficult at
the current stage. We treat the special case when $C$ and $C^{\prime}$ are
disjoint, i.e. $C$ is a symplectic four manifold in this paper. The basic tool
is the gluing technique. But even in this simpler case, the set up is quite
different from the Lagrangian Floer theory (c.f. Fukaya and Oh \cite{Fukaya
Oh}). Our domains are three dimensional instead of two dimensional and we have
to deform the submanifolds rather than deforming the maps as was done in Floer
theory. This is because we do not have the luxury of applying the conformal
geometry in dimension two to transform the problem of finding holomorphic
curves into the one of finding holomorphic maps. Furthermore, the linear
theory is more difficult in this case since we have to deal with a problem
which lacks uniform ellipticity, as we will explain in Section \ref{linear}.
Since the $G_{2}$ form is cubic, the needed quadratic estimate of the
$3$-dimensional instanton equation appears unavailable in the $L^{p}$ setting
(see end of Section \ref{quadratic} ). So we set up the problem in Schauder
setting, and in linear theory we need to go further from $L^{p}$ estimates to
Schauder estimates. This is different from \cite{Fukaya Oh}.

As $C^{\prime}$\ should be sufficiently close to $C$, we assume that they
arise in a one-parameter smooth family of coassociative submanifolds $C_{t}%
$\ with small $t$. Contracting with the normal vector field $n:=dC_{t}%
/dt|_{t=0}$\ for the infinitesimal deformation with the $G_{2}$-form $\Omega$,
we obtain a closed self-dual two form $\iota_{n}\Omega\in\Omega_{+}^{2}\left(
C_{0}\right)  $. Using the induced metric, from $n$ one can define an
\emph{almost complex structure} $J=J_{n}$\ on $C_{0}$\ away from the zero set
of $\iota_{n}\Omega$ (see $\left(  \ref{Jn1}\right)  $ for details). \

\begin{theorem}
\label{correspondence}Suppose that $\left(  M,\Omega\right)  $\ is a $G_{2}%
$-manifold and $\left\{  C_{t}\right\}  $\ is a one-parameter smooth family of
coassociative submanifolds in $M$. When $\iota_{n}\Omega\in\Omega_{+}%
^{2}\left(  C_{0}\right)  $\ is nonvanishing, then

\begin{enumerate}
\item (Proposition \ref{j-curve}) If $\left\{  \mathtt{A}_{t}\right\}  $\ is
any one-parameter family of associative submanifolds (i.e. instantons) in
$M$\ satisfying
\[
\partial\mathtt{A}_{t}\subset C_{t}\cup C_{0},\text{ }\lim_{t\rightarrow
0}\mathtt{A}_{t}\cap C_{0}=\Sigma_{0}\text{ in the }C^{1}\text{-topology,}%
\]
then $\Sigma_{0}$\ is a $J_{n}$-holomorphic curve in $C_{0}$. \

\item (Theorem \ref{main}) Conversely, every regular $J_{n}$-holomorphic curve
$\Sigma_{0}$ (namely those for which the linearization of $\overline{\partial
}_{J_{n}}$ on $\Sigma_{0}$ is surjective) in $C_{0}\ $is the limit of a family
of associative submanifolds $A_{t}$'s as described above.
\end{enumerate}
\end{theorem}

\bigskip

Notice that in the product situation, where
\[
\left(  M,\Omega\right)  :=\left(  X\times S^{1},\operatorname{Re}\Omega
_{X}+\omega_{X}\wedge d\theta\right)
\]
(c.f. Section \ref{counting}) with $\left(  X,\Omega_{X}\right)  $ being a
Calabi-Yau threefold, then our theorem would follow from the work of
\cite{Fukaya Oh}.

The paper is organized as follows: In Section \ref{instantons}, we first
recall some basics of Floer theory, next we describe their $G_{2}%
$-counterparts, then we explain the connection between instantons and
Seiberg-Witten invariants, and last we study the deformation of instantons
with the aim to generalize to almost instantons. In Section
\ref{Dirac-thin-mfd}, we first study the linear differential operator
$\mathcal{D}$ (defined in $\left(  \ref{odd-Dirac}\right)  $) on a type of
thin $3$-manifolds, which is a linear approximation of the instanton equation,
then we give the $L^{2}$ and Schauder estimates of its inverse $\mathcal{D}%
^{-1}$. In Section \ref{proof}, we first compare the linearized instanton
equation on almost instantons with the operator $\mathcal{D}$ on linear
models, then we use the implicit function theorem to perturb almost instantons
to true instantons, thus proving our main theorem.

\textbf{Acknowledgments: }The first author expresses his gratitude to J.H.
Lee, Y.G. Oh, C. Taubes, R. Thomas and A. Voronov for useful discussions. The
second author thanks S.L. Kong, G. Liu, Y.J. Lee, Y.G. Shi, L. Yin for useful
discussions. The third author thanks the nice research environment in math
departments of The Chinese University of Hong Kong and University of
Minnesota, and useful discussion with T.J. Li. We thank the anonymous referees
for pointing out many inaccuracies in our earlier versions and providing
suggestions for improvement.

The work of the first author described in this paper was partially supported
by grants from the Research Grants Council of the Hong Kong Special
Administrative Region, China (Project No. CUHK401908 and CUHK403709). The work
of the second author described in this paper was partially supported by grant
from the Research Grants Council of the Hong Kong Special Administrative
Region, China (Project No. CUHK403709). The third author was partially
supported by RGC grant from the Hong Kong Government.

\section{Instantons of Dimension $2$ and $3\label{instantons}$}

\subsection{Review of Symplectic Geometry\label{sympl-geo}}

Given any symplectic manifold $\left(  X,\omega\right)  $ of dimension $2n$,
there exists a compatible metric $g$ so that the equation
\[
\omega\left(  u,v\right)  =g\left(  Ju,v\right)
\]
defines a Hermitian almost complex structure
\[
J:T_{X}\rightarrow T_{X}\text{,}%
\]
that is $J^{2}=-id$ and $g\left(  Ju,Jv\right)  =g\left(  u,v\right)  $.

A \textbf{holomorphic curve}, or \textbf{instanton }of dimension\textbf{ }$2$,
is a two dimensional submanifold $\Sigma$ in $X$ whose tangent bundle is
preserved by $J$. Equivalently $\Sigma$ is calibrated by $\omega,$ i.e.
$\omega|_{\Sigma}=vol_{\Sigma}$. By \textbf{algebraic counting} the number of
instantons in $X$, one can define a highly nontrivial invariant for the
symplectic structure on $X$, called the Gromov-Witten invariant.

When the instanton $\Sigma$ has nontrivial boundary, then the corresponding
boundary value problem would require $\partial\Sigma$ to lie on a
\textbf{Lagrangian submanifold} $L$ in $X$, i.e. $\dim L=n$ and $\omega
|_{L}=0$. Floer studied the intersection theory of Lagrangian submanifolds and
defined the Floer homology group $HF\left(  L,L^{\prime}\right)  $ under
certain assumptions.

Suppose that $X$ is a compact Calabi-Yau manifold, i.e. the holonomy group of
the Levi-Civita connection is inside $SU\left(  n\right)  $, equivalently $J$
is an integrable complex structure on $X$ and there exists a holomorphic
volume form $\Omega_{X}\in\Omega^{n,0}\left(  X\right)  $ on $X$ satisfying
\[
\left(  -1\right)  ^{\frac{n\left(  n-1\right)  }{2}}\left(  i/2\right)
^{n}\Omega_{X}\wedge\bar{\Omega}_{X}=\omega^{n}/n!.
\]
Under the mirror symmetry transformation, $HF\left(  L,L^{\prime}\right)  $ is
expected to correspond to the Dolbeault cohomology group of coherent sheaves
in the mirror Calabi-Yau manifold.

A Lagrangian submanifold $L$ in $X$ is called a \textbf{special Lagrangian
submanifold} with phase zero (resp. $\pi/2$) if $\operatorname{Im}\Omega
_{X}|_{L}=0$ (resp. $\operatorname{Re}\Omega_{X}|_{L}$ $=0$). With suitable
choice of orientation of $L$, $L$ is calibrated by $\operatorname{Re}%
\Omega_{X}|_{L}$ (resp. $\operatorname{Im}\Omega_{X}$%
$\vert$%
$_{L}$), that is $\operatorname{Re}\Omega_{X}|_{L}$ is the volume form of $L$.
They play important roles in the Strominger-Yau-Zaslow mirror conjecture for
Calabi-Yau manifolds \cite{SYZ}.

When $X$ is a Calabi-Yau threefold, there are conjectures of Vafa and others
(e.g. \cite{Mina Vafa}\cite{GV2}) that relates the (partially defined) open
Gromov-Witten invariant of the number of instantons with Lagrangian boundary
condition to the large $N$ Chern-Simons invariants of knots in three manifolds.

\subsection{Counting Instantons in (almost) $G_{2}$-manifolds\label{counting}}

Notice that a real linear homomorphism $J:\mathbb{R}^{m}\rightarrow
\mathbb{R}^{m}$ being a Hermitian complex structure on $\mathbb{R}^{m}$ is
equivalent to the following conditions: for any vector $v\in\mathbb{R}^{m}$ we have

\begin{enumerate}
\item $Jv$ is perpendicular to $v$.

\item $\left\vert Jv\right\vert =\left\vert v\right\vert $.
\end{enumerate}

We can generalize $J$ to the case involving more than one vector. We call a
skew symmetric bilinear map
\[
\times:\mathbb{R}^{m}\otimes\mathbb{R}^{m}\rightarrow\mathbb{R}^{m}%
\]
a (2-fold) \textbf{vector cross product} if it satisfies
\[%
\begin{array}
[c]{cl}%
\text{(i) } & \left(  u\times v\right)  \,\text{is perpendicular to both
}u\text{ and }v.\\
\text{(ii) } & \left\vert u\times v\right\vert =\text{Area of parallelogram
spanned by }u\text{ and }v=\left\vert u\wedge v\right\vert .
\end{array}
\]
The obvious example of this is the standard vector product on $\mathbb{R}^{3}%
$. By identifying $\mathbb{R}^{3}$ with $\operatorname{Im}\mathbb{H}$, the
imaginary part of the quaternion numbers, we have
\[
u\times v=\operatorname{Im}\bar{v}u.
\]
The same formula defines a vector cross product on $\mathbb{R}^{7}%
=\operatorname{Im}\mathbb{O}$, the imaginary part of the octonion numbers.
Brown and Gray \cite{Gray VectorCrossProd} showed that these two are the only
possible vector cross product structures on $\mathbb{R}^{m}$ up to
automorphism of $\mathbb{R}^{m}$.

Suppose that $M$ is a seven dimensional Riemannian manifold with a vector
cross product $\times$ on each of its tangent spaces. The analog of the
symplectic form is a degree three differential form $\Omega$ on $M$ defined as
follows:
\[
\Omega\left(  u,v,w\right)  =g\left(  u\times v,w\right)  \text{.}%
\]

\begin{definition}
Suppose that $\left(  M,g\right)  $ is a Riemannian manifold of dimension
seven with a vector cross product $\times$ on its tangent bundle. Then

\begin{enumerate}
\item $M$ is called an \textbf{almost }$G_{2}$\textbf{-manifold} if
$d\Omega=0$.

\item $M$ is called a\textbf{\ }$G_{2}$\textbf{-manifold} if $\nabla\Omega=0 $
with $\nabla$ being the Levi-Civita connection.
\end{enumerate}
\end{definition}

\begin{remark}
$M$ is a $G_{2}$-manifold if and only if its holonomy group is inside the
exceptional Lie group $G_{2}=Aut\left(  \mathbb{O}\right)  $. The geometry of
$G_{2}$-manifolds can be interpreted as the symplectic geometry on its knot
space (see e.g. \cite{Lee Leung Instanton Brane}, \cite{Movshev}).
\end{remark}

A typical family of examples of $G_{2}$-manifolds can be obtained via the
product manifold $M:=X\times S^{1}$ with $\left(  X,\omega_{X}\right)  $ being
a Calabi-Yau threefold with a holomorphic volume form $\Omega_{X}$, and the
$G_{2}$-form is given by
\[
\Omega=\operatorname{Re}\Omega_{X}+\omega_{X}\wedge d\theta\text{.}%
\]

Next we define the analogs of holomorphic curves and Lagrangian submanifolds
in the $G_{2}$ setting.

\begin{definition}
Suppose that $A$ is a three dimensional submanifold of an almost $G_{2}%
$-manifold $M$. We call $A$ an \textbf{instanton }or \textbf{associative
submanifold}, if $TA$ is preserved by the vector cross product $\times$.
\end{definition}

Harvey and Lawson \cite{Harvey Lawson} showed that $A\subset M$ is an
instanton if and only if $A$ is calibrated by $\Omega$, i.e. $\Omega
|_{A}=vol_{A}$. This is in turn equivalent to $\tau|_{TA}=0$ in our Lemma
\ref{instanton-associative} for $\tau$ defined in $\left(  \ref{tau}\right)  $
(i.e. Corollary 1.7 in Section IV.1.A. of\ \cite{Harvey Lawson} and Corollary
14 in \cite{Lee Leung Instanton Brane}).

In M-theory, associative submanifolds are also called \emph{M2-branes}. In the
case when $M=X\times S^{1}$ with $X$ a Calabi-Yau threefold, $\Sigma\times
S^{1}$ (resp. $L\times\left\{  p\right\}  $) is an instanton in $M$ if and
only if $\Sigma$ (resp. $L$) is a holomorphic curve (resp. special Lagrangian
submanifold with zero phase) in $X$.

A natural interesting question is to count the number of instantons in $M$. In
the special case of $M=X\times S^{1}$, these numbers are related to the
conjectural invariants proposed by Joyce \cite{Joyce Count SLag} by counting
special Lagrangian submanifolds in Calabi-Yau threefolds, and the product of
holomorphic curves with $S^{1}$. This problem has been discussed by many
physicists. For example Harvey and Moore discussed in \cite{Harvey Moore} the
mirror symmetry aspects of these invariants; Aganagic and Vafa in \cite{Mina
Vafa} related these invariants to the open Gromov-Witten invariants for local
Calabi-Yau threefolds; Beasley and Witten argued in \cite{BeasleyWitten} that
when there is a moduli of instantons, then one should count them using the
Euler characteristic of the moduli space. The compactness issues of the moduli
of instantons is a very challenging problem because the bubbling-off phenomena
of ($3$-dimensional) instantons has not been well understood. This makes it
very difficult to define an honest invariant by counting instantons.

Analogous to the Floer homology for Lagrangian intersections, when an
instanton $A$ has a nontrivial boundary, $\partial A\neq\phi,$ one should
require it to lie inside a \textbf{brane}\emph{\ }or a \textbf{coassociative
submanifold} to make it a well-posed elliptic problem (see \cite{GayetWitt}
for Fredholmness and index computation), i.e. submanifolds in $M$ where the
restriction of $\Omega$ is zero and have the largest possible dimension.
Branes are the analog of Lagrangian submanifolds in symplectic geometry.

\begin{definition}
Suppose that $C$ is a submanifold of an almost $G_{2}$-manifold $M$. We call
$C$ a \textbf{coassociative submanifold} if
\[
\Omega|_{C}=0\text{ and }\dim C=4.
\]

\end{definition}

For example when $M=X\times S^{1}$ with $X$ a Calabi-Yau threefold, $H\times
S^{1}$ (resp. $C\times\left\{  p\right\}  $) is a coassociative submanifold in
$M$ if and only if $H$ (resp. $C$) is a special Lagrangian submanifold with
phase $\pi/2$ (resp. complex surface) in $X$. In \cite{Lee Leung Instanton
Brane} J.H. Lee and the first author showed that the isotropic knot space
$\mathcal{\hat{K}}_{S^{1}}X$ of $X$ admits a natural holomorphic symplectic
structure. Moreover $\mathcal{\hat{K}}_{S^{1}}H$ (resp. $\mathcal{\hat{K}%
}_{S^{1}}C$) is a complex Lagrangian submanifold in $\mathcal{\hat{K}}_{S^{1}%
}X$ with respect to the complex structure $J$ (resp. $K$).

Constructing special Lagrangian submanifolds with zero phase in $X$ with
boundaries lying on $H$ (resp. $C$) corresponds to the Dirichlet (resp.
Neumann) boundary value problem for minimizing volume among Lagrangian
submanifolds as studied by Schoen, Wolfson (\cite{S}, \cite{SW}) and
\cite{Bu}. For a general $G_{2}$-manifold $M$, the natural boundary value for
an instanton is a coassociative submanifold. Similar to the intersection
theory of Lagrangian submanifolds in symplectic manifolds, we propose to study
the following problem: Count the number of instantons in $G_{2}$-manifolds
bounding two coassociative submanifolds.

The product of a coassociative submanifold with a two dimensional plane inside
the eleven dimension spacetime $M\times\mathbb{R}^{3,1}$ is called a
\emph{D5-brane }in M-theory. Counting the number of M2-branes between two
D5-branes has also been studied in the physics literatures.

In general this is a very difficult problem. For instance, counting $S^{1}%
$-invariant instantons in $M=X\times S^{1}$ is the open Gromov-Witten
invariant. However when the two coassociative submanifolds $C$ and $C^{\prime
}$ are \emph{close} to each other, we can relate the number of thin instantons
between them to the number of $J$-holomorphic curves in $C$ (Theorem
\ref{correspondence}), hence by Taubes' work, to the Seiberg-Witten invariant
of $C$.

\subsection{Relationships to Seiberg-Witten Invariants\label{sw}}

To determine the number of instantons between nearby coassociative
submanifolds, we first recall the deformation theory of compact coassociative
submanifolds $C$ inside any $G_{2}$-manifold $M$, as developed by McLean
\cite{McLean}. Given any normal vector field $n\in\Gamma\left(  N_{C/M}%
\right)  $, the interior product $\iota_{n}\Omega$ is naturally a self-dual
two form on $C$ because of $\Omega|_{C}=0$. This gives a natural
identification,
\begin{align}
\Gamma &  \left(  N_{C/M}\right)  \overset{\simeq}{\rightarrow}\Lambda_{+}%
^{2}\left(  C\right) \nonumber\\
n  &  \rightarrow\eta_{0}=\iota_{n}\Omega\text{.} \label{n-eta}%
\end{align}
Furthermore infinitesimal deformations of coassociative submanifolds are
parameterized by \emph{self-dual harmonic two forms} $\eta_{0}\in H_{+}%
^{2}\left(  C\right)  $, and they are always unobstructed, i.e. any such forms
with sufficiently small norm give actual deformations to nearby coassociative
submanifolds (see section 4 of \cite{McLean}). Notice that the zero set of
$\eta_{0}$ is the intersection of $C$ with a infinitesimally nearby
coassociative submanifold $C_{t}$, that is
\[
\left\{  \eta_{0}=0\right\}  =\lim_{t\rightarrow0}\left(  C\cap C_{t}\right)
\text{,}%
\]
where $C=C_{0}$ and $\eta_{0}=dC_{t}/dt|_{t=0}$.

Since
\[
\eta_{0}\wedge\eta_{0}=\eta_{0}\wedge\ast\eta_{0}=\left\vert \eta
_{0}\right\vert ^{2}\ast1\text{,}%
\]
$\eta_{0}$ defines a natural symplectic structure on $C^{reg}:=C\backslash
\left\{  \eta_{0}=0\right\}  $.\ If we normalize $\eta_{0\text{ }}$to $\eta$,
\[
\eta=\eta_{0}/\left\vert \eta_{0}\right\vert \text{,}%
\]
then the equation
\[
\eta\left(  u,v\right)  =g\left(  Ju,v\right)
\]
defines a Hermitian almost complex structure $J$ on $C^{reg}$. The $J$ is
determined by $\eta$, which in turn is determined by $n$, so we denote it by
$J_{n}$. More explicitly, for $u\in TC^{reg}$,%
\begin{equation}
J_{n}\left(  u\right)  =\left\vert n\right\vert ^{-1}n\times u. \label{Jn1}%
\end{equation}

The next proposition says that when two coassociative submanifolds $C$ and
$C^{\prime}$ come together, not necessarily disjoint, then the limit of
instantons bounding them will be a $J_{n}$-holomorphic curve $\Sigma$ in
$C^{reg}$ with boundary $C\cap C^{\prime}$.

\begin{proposition}
\label{j-curve}Let $M$ be a $G_{2}$-manifold. Suppose that for some
$\varepsilon_{0}>0$, there is a smooth map
\[
\psi:C\times\left[  0,\varepsilon_{0}\right]  \longrightarrow M
\]
such that for each $t\in\left[  0,\varepsilon_{0}\right]  ,$ $\psi_{t}\left(
\cdot\right)  :=\psi\left(  \cdot,t\right)  $ is a smooth immersion of $C$
into $M$ as a coassociative submanifold $C_{t}:=\psi\left(  C\times\left\{
t\right\}  \right)  $. Suppose that $n=dC_{t}/dt|_{t=0}\in\Gamma\left(
N_{C/M}\right)  $ is nowhere vanishing, and
\[
\phi_{t}:\Sigma\times\left[  0,t\right]  \longrightarrow M
\]
is a smooth family of instantons in $M$ such that for each $t\in
(0,\varepsilon_{0}]$, $\ \phi_{t}$ is an associative immersion with boundary
condition
\[
\phi_{t}\left(  \Sigma\times\left\{  0\right\}  \right)  \subset C_{0}%
:=\psi\left(  C\times\left\{  0\right\}  \right)  \text{ , }\phi_{t}\left(
\Sigma\times\left\{  t\right\}  \right)  \subset C_{t}:=\psi\left(
C\times\left\{  t\right\}  \right)  .
\]
Suppose that the $C^{1}$-limit of $\phi_{t}\left(  \Sigma\times\left\{
0\right\}  \right)  $ exists as $t\rightarrow0$. Then $\Sigma_{0}%
:=\lim_{t\rightarrow0}\phi_{t}\left(  \Sigma\times\left\{  0\right\}  \right)
$ is a $J_{n}$-holomorphic curve in $C_{0}$ ,
\end{proposition}

\begin{proof}
Let us denote the boundary component of $A_{t}=$Image$\left(  \phi_{t}\right)
$ in $C_{0}$ by $\Sigma_{t},$ i.e. $\Sigma_{t}:=\phi_{t}\left(  \Sigma
\times\left\{  0\right\}  \right)  $. Let $w_{t}$ be a unit normal vector
field for $\Sigma_{t}$ in $A_{t}$. We claim that $w_{t}$ is
\emph{perpendicular to} $C_{0}$. To see this, note that $TA_{t}$ being
preserved by the vector cross product implies that
\[
w_{t}=u\times v
\]
for some tangent vectors $u$ and $v$ in $\Sigma_{t}$; In fact for any unit
vector $u\in T\Sigma_{t}$, $v:=w_{t}\times u\in T\Sigma_{t}$ by associativity
condition of $A_{t}$ and $v \perp w_{t}$, and%
\begin{align*}
u\times v &  =-u\times\left(  u\times w_{t}\right)  \\
&  =\tau\left(  u,u,w_{t}\right)  +g\left(  u,w_{t}\right)  u+g\left(
u,u\right)  w_{t}\\
&  =0+0+w_{t}=w_{t},
\end{align*}
where we have used in the second row the definition of $\tau$ in $\left(
\ref{tau}\right)$ and $\tau$ is a (vector-valued) form. Therefore given any
tangent vector $w$ along $C_{0}$, we
have
\[
g\left(  w_{t},w\right)  =g\left(  u\times v,w\right)  =\Omega\left(
u,v,w\right)  =0,
\]
where the last equality follows from $C_{0}$ being coassociative and
$\Sigma_{t}\subset C_{0}$. Reparameterize $\phi_{t}:\Sigma\times\left[
0,\varepsilon\right]  \rightarrow M$ when necessary, then for small $t$ we can
assume that $\left.  \frac{d}{ds}\phi_{t}\left(  z,s\right)  \right\vert
_{s=0}$ is parallel to $w_{t}\left(  z\right)  $ for any $z\in\Sigma$. Noting
that
\[
\phi_{t}\left(  z,0\right)  \subset C_{0}\text{ and }\phi_{t}\left(
z,t\right)  \subset C_{t},
\]
we see $\lim_{t\rightarrow0}w_{t}\left(  z\right)  $ is parallel to
\[
\left.  \left(  \left.  \frac{dC_{t}}{dt}\right\vert _{t=0}\right)
\right\vert _{\Sigma_{0}}=n\in\Gamma\left(  \Sigma_{0},N_{C_{0}/M}\right)
\text{.}%
\]
Therefore along $\Sigma_{0}$
\[
\lim_{t\rightarrow0}w_{t}\left(  z\right)  =\left\vert n\left(  z\right)
\right\vert ^{-1}n\left(  z\right)  .
\]
For any $u_{t}\in T\Sigma_{t}$, $w_{t}\times u_{t}\perp w_{t}$ in associative
$A_{t}$ so
\[
w_{t}\times u_{t}\in T\Sigma_{t}.
\]
Since $\Sigma_{t}\rightarrow\Sigma_{0}$ in $C^{1}$ topology, for any $u\in
T\Sigma_{0}$, $u$ can be realized as the limit of $u_{t}$. Therefore
\[
J_{n}\left(  u\right)  =\left\vert n\right\vert ^{-1}n\times u=\lim
_{t\rightarrow0}\left(  w_{t}\times u_{t}\right)  \in\lim_{t\rightarrow
0}T\Sigma_{t}=T\Sigma_{0},
\]
i.e. $\Sigma_{0}$ is a $J_{n}$-holomorphic curve in $C_{0}$ with respect to
the almost complex structure $J_{n}$ defined $\left(  \ref{Jn1}\right)  $.
\end{proof}

The reverse of the above proposition is also true (Theorem \ref{main}). The
Lagrangian analog of it was proven by Fukaya and Oh in \cite{Fukaya Oh}. On
the other hand, by the celebrated work of Taubes, we expect that the number
(counted with algebraic weights) of such holomorphic curves in $C_{0}$ equals
to the Seiberg-Witten invariant of $C_{0}$. We conjecture the following statement.

\bigskip

\textbf{Conjecture}: Suppose that $C$ and $C^{\prime}$ are nearby
coassociative submanifolds in a $G_{2}$-manifold $M$. Then the number of
instantons counted with algebraic weights in $M$ with small volume and with
boundary lying on $C\cup C^{\prime}$ is given by the Seiberg-Witten invariants
of $C$.

The main result of our paper is to solve a special case of the above
conjecture, namely, we will concentrate on the case that $C$ and $C^{\prime}$
are both \textbf{compact} and they do \textbf{NOT} intersect.\footnote{In the
remainder of the paper, we will always assume $C$ and $C^{\prime}$ are compact
and they do NOT intersect.}The basic ideas are (i) the limit of such
instantons is a $J$-holomorphic curve for almost complex structure $J$
compatible to the (degenerated) symplectic form $\eta$ on $C$ coming from its
deformations as coassociative submanifolds and this process can be reversed;
(ii) the number of $J$-holomorphic curves in the four manifold $C$ should be
related to the Seiberg-Witten invariant of $C$ by the work of Taubes (\cite{Ta
ICM1998}, \cite{Ta SW GW deg}). Note that one only gets one symplectic form
$\eta$ (and hence one almost complex structure $J$) from a given coassociative
deformation of $C$, though of course one can get more (from different
coassociative deformations).

Suppose that $\eta$ is a self-dual two form on $C$ with constant length
$\sqrt{2}$, in particular it is a (non-degenerate) symplectic form, and
$\Sigma$ is a smooth holomorphic curve in $C$, possibly disconnected. If
$\Sigma$ is \textbf{regular} in the sense that the linearized Cauchy-Riemann
operator $D_{\Sigma}\bar{\partial}_{J}$ \ has trivial cokernel \cite{Ta Gr
Sw}, then Taubes showed that the perturbed Seiberg-Witten equations,
\begin{align*}
F_{a}^{+}  &  =q\left(  \psi\right)  -r\sqrt{-1}\eta,\\
D_{A\left(  a\right)  }\psi &  =0,
\end{align*}
have solutions for all sufficiently large $r$. Here $a$ is a connection on the
complex line bundle $E$ over $C$ whose first Chern class equals the
Poincar\'{e} dual of $\Sigma$, $PD\left[  \Sigma\right]  $, $F_{a}$ is the
curvature $2$-form of $E$ and $F_{a}^{+}$ \ is the projection of $F_{a}$ to
$\wedge_{+}^{2}\left(  C\right)  $, $\psi$ is a section of the twisted spinor
bundle $S_{+}=E\oplus\left(  K^{-1}\otimes E\right)  $ and $D_{A\left(
a\right)  }$ is the twisted Dirac operator, and $q\left(  \cdot\right)  $ is a
certain canonical quadratic map from $S_{+}$ to $i\cdot\wedge_{+}^{2}\left(
C\right)  $. The number of such solutions (counted with algebraic weights) is
the Seiberg-Witten invariant $SW_{C}\left(  \Sigma\right)  $ of $C$.
Furthermore the converse is also true, namely the Seiberg-Witten invariant
$SW_{C}\left(  \Sigma\right)  $ is equal to the Gromov-Witten invariants
counting holomorphic curves $\Sigma$. Thus Taubes established an equivalence
between Seiberg-Witten theory and Gromov-Witten theory for symplectic four
manifolds. This result has far reaching applications in four dimensional
symplectic geometry.

For a general four manifold $C$ with nonzero $b^{+}\left(  C\right)  $, using
a generic metric, any self-dual two form $\eta$ on $C$ defines a degenerate
symplectic form on $C$, i.e. $\eta$ is a symplectic form on the complement of
$\left\{  \eta=0\right\}  $, which is a finite union of circles (see
\cite{Kirby}\cite{Honda}). Therefore, one might expect to have a relationship
between the Seiberg-Witten invariants of $C$ and the number of holomorphic
curves with boundaries $\left\{  \eta=0\right\}  $ in $C$. Part of this
Taubes' program has been verified in \cite{Ta ICM1998}, \cite{Ta SW GW deg}.

Suppose that $\eta$ is a nowhere vanishing self-dual harmonic two form on a
coassociative submanifold $C$ in a $G_{2}$-manifold $M$. For any holomorphic
curve $\Sigma$ in $C$, we want to construct an instanton in $M$ bounding $C$
and $C^{\prime}$, where $C^{\prime}$ is a small deformation of the
coassociative submanifold $C$ along the normal direction given by $\eta$.
Notice that $C$ and $C^{\prime}$ do not intersect. We will construct such an
instanton using a perturbation argument which requires a lower bound on the
first eigenvalue for the appropriate elliptic operator. Recall that the
deformation of an instanton is governed by a twisted Dirac operator. We will
reinterpret it as a complexified version of the Cauchy-Riemann operator in
Section \ref{linear-model}.\newline

\subsection{Deformation of Instantons}

To construct an instanton A in $M$ from a holomorphic curve $\Sigma$ in $C$,
we need to perturb an almost instanton A$^{\prime}$ to a honest one using a
quantitative version of the implicit function theorem. Let us first recall the
deformation theory of instantons A (\cite{Harvey Lawson} and \cite{Lee Leung
Instanton Brane}) in a Riemannian manifold $\left(  M,g\right)  $ with a
parallel (or closed) $r$-fold vector cross product
\[
\times:\Lambda^{r}TM\rightarrow TM\text{.}%
\]
In our situation, we have $r=2$. By taking the wedge product with $TM$ we
obtain a homomorphism $\tau$,
\[
\tau:\Lambda^{r+1}TM\rightarrow\Lambda^{2}TM\cong\Lambda^{2}T^{\ast}M\text{,}%
\]
where the last isomorphism is induced from the Riemannian metric. As a matter
of fact, the image of $\tau$ lies inside the subbundle $\mathfrak{g}_{M}%
^{\bot}$ which is the orthogonal complement of $\mathfrak{g}_{M}%
\subset\mathfrak{so}\left(  TM\right)  \cong\Lambda^{2}T^{\ast}M$, the bundle
of infinitesimal isometries of $TM$ preserving $\times$. That is,
\[
\tau\in\Omega^{r+1}\left(  M,\mathfrak{g}_{M}^{\bot}\right)  .
\]

\begin{lemma}
\label{instanton-associative}(\cite{Harvey Lawson}, \cite{Lee Leung Instanton
Brane}) An $r+1$ dimensional submanifold A$\subset M$ is an instanton, i.e.
$T$A is preserved by $\times$, if and only if
\[
\tau|_{\text{A}}=0\in\Omega^{r+1}\left(  \text{A},\mathfrak{g}_{M}^{\bot
}\right)  .
\]

\end{lemma}

This lemma is important in describing deformations of an instanton. Mclean
\cite{McLean} used this to show that the normal bundle to an instanton A is a
twisted spinor bundle over A and infinitesimal deformations of A are
parameterized by twisted harmonic spinors.

In our present situation, $\left(  M,g\right)  $ is a $G_{2}$-manifold. Using
the cross product, we can identify $\mathfrak{g}_{M}^{\bot}\subset\Lambda
^{2}T^{\ast}M\cong\Lambda^{2}TM$ with the tangent bundle $TM$, i.e. for
$u\wedge v\in\Lambda^{2}TM$, we identify it with $w\in TM$ that $w=u\times v$.
Then we can also characterize $\tau\in\Omega^{3}\left(  M,TM\right)  $ by the
following formula,
\[
\left(  \ast\Omega\right)  \left(  u,v,w,z\right)  =g\left(  \tau\left(
u,v,w\right)  ,z\right)  \text{,}%
\]
More explicitly,
\begin{equation}
\tau\left(  u,v,w\right)  =-u\times\left(  v\times w\right)  -g\left(
u,v\right)  w+g\left(  u,w\right)  v. \label{tau}%
\end{equation}
Therefore A$\subset M$ is an instanton if and only if $\ast_{\text{A}}\left(
\tau|_{\text{A}}\right)  =0\in T_{M}|_{\text{A}}$.

\textbf{Example.} The $G_{2}$ manifold $\mathbb{R}^{7}$:
\[
\mathbb{R}^{7}\simeq\operatorname{Im}\mathbb{O}\simeq\operatorname{Im}%
\mathbb{H\oplus H=}\left\{  \left(  x_{1}\mathbf{i}+x_{2}\mathbf{j}%
+x_{3}\mathbf{k},x_{4}+x_{5}\mathbf{i}+x_{6}\mathbf{j}+x_{7}\mathbf{k}\right)
\right\}  ,
\]
the standard basis consists of $e_{i}=\frac{\partial}{\partial x_{i}}\left(
i=1,2,\cdots7\right)  $, the multiplication $\times$ for $\left(  a,b\right)
,\left(  c,d\right)  \in\operatorname{Im}\mathbb{H\oplus H}\simeq
\operatorname{Im}\mathbb{O}$ is
\begin{equation}
\left(  a,b\right)  \times\left(  c,d\right)  =\left(  ac-d^{\ast
}b,da+bc^{\ast}\right)  \label{Cayley-Dickson}%
\end{equation}
(Cayley--Dickson construction), where $z^{\ast}$ denotes the conjugate of the
quaternion $z$. The $G_{2}$ form $\Omega$ \ is%
\[
\Omega=\omega^{123}-\omega^{167}-\omega^{527}-\omega^{563}-\omega^{154}%
-\omega^{264}-\omega^{374},
\]
the form $\tau$ is the following ($\left(  5.4\right)  $ in \cite{McLean})%
\begin{align}
\tau &  =\left(  \omega^{256}-\omega^{247}+\omega^{346}-\omega^{357}\right)
\frac{\partial}{\partial x_{1}}+\left(  \omega^{156}-\omega^{147}-\omega
^{345}+\omega^{367}\right)  \frac{\partial}{\partial x_{2}}\nonumber\\
&  +\left(  \omega^{245}-\omega^{267}-\omega^{146}-\omega^{157}\right)
\frac{\partial}{\partial x_{3}}+\left(  \omega^{567}-\omega^{127}+\omega
^{136}-\omega^{235}\right)  \frac{\partial}{\partial x_{4}}\nonumber\\
&  +\left(  \omega^{126}-\omega^{467}+\omega^{137}+\omega^{234}\right)
\frac{\partial}{\partial x_{5}}+\left(  \omega^{457}-\omega^{125}-\omega
^{134}+\omega^{237}\right)  \frac{\partial}{\partial x_{6}}\nonumber\\
&  +\left(  \omega^{124}-\omega^{456}-\omega^{135}-\omega^{236}\right)
\frac{\partial}{\partial x_{7}},\text{ \ \ \ \ } \label{tau-standard}%
\end{align}
where $\omega^{ijk}=dx_{i}\wedge dx_{j}\wedge dx_{k}$. $\operatorname{Im}%
\mathbb{H\oplus}\left\{  0\right\}  $ is associative (i.e. an instanton), and
$\left\{  0\right\}  \mathbb{\oplus H}$ is coassociative.

As a matter of fact, if A is already close to being an instanton, then we only
need the normal components of $\ast_{\mathtt{A}}\left(  \tau|_{\mathtt{A}%
}\right)  $ to vanish.

\begin{proposition}
\label{1Prop alm instanton}There is a positive constant $\delta$ such that for
any $3$-plane A in $\left(  \mathbb{R}^{7},\Omega\right)  $ with $\left\vert
\tau|_{\text{A}}\right\vert <\delta$, A is an instanton if and only if
$\ast_{\text{A}}\left(  \tau|_{\text{A}}\right)  \in T_{\text{A}}$.
\end{proposition}

\begin{proof}
McLean observed (From formula (5.6) in \cite{McLean}) that if A$_{t}$ is a
family of linear subspaces in $M\cong\mathbb{R}^{7}$ with A$_{0}$ an
instanton, then
\[
\ast_{\text{A}_{t}}\left.  \left(  \frac{d\tau|_{\text{A}_{t}}}{dt}\right)
\right\vert _{t=0}\in N_{\text{A}_{0}/M}\subset T_{M}|_{\text{A}_{0}}\text{.}%
\]
We may assume that A is spanned by $e_{1},e_{2}$ and $\tilde{e}_{3}=e_{3}%
+\sum_{a=4}^{7}t_{a}e_{a}$ for some small $t_{a}$'s where $e_{i}$'s are a
standard basis for $\mathbb{R}^{7},$ in particular $e_{1}\times e_{2}=e_{3}.$
This is because the natural action of $G_{2}$ on the Grassmannian $Gr\left(
2,7\right)  $ is transitive. An easy computation (c.f. equation (5.4) in
\cite{McLean}) shows that the normal component of $\ast\left(  \tau
|_{\text{A}}\right)  $ in $N_{\text{A}/M}$ is given by
\[
\ast\left(  \tau|_{\text{A}}\right)  ^{\perp}=-t_{5}\left(  e_{4}\right)
^{\perp}+t_{4}\left(  e_{5}\right)  ^{\perp}+t_{7}\left(  e_{6}\right)
^{\perp}-t_{6}\left(  e_{7}\right)  ^{\perp},
\]
where $\left(  \cdot\right)  ^{\bot}$ denote the orthogonal projection onto
$N_{\text{A}/M}.$When $t_{a}$'s are all zero, we have $\left(  e_{a}\right)
^{\perp}=e_{a}$ for $4\leq a\leq7$. In particular, they are linearly
independent when $t_{a}$'s are small. In that case, $\ast\left(
\tau|_{\text{A}}\right)  ^{\perp}=0$ will actually imply that $t_{a}=0$ for
all $a$, i.e. A is an instanton in $M$. Hence the proposition.
\end{proof}

This proposition will be needed later when we perturb an almost instanton to
an honest one. We also need to identify the normal bundle $N_{\text{A}/M}$ to
an instanton A with a twisted spinor bundle over A as following \cite{McLean}:
We denote $P$ to be the $SO\left(  4\right)  $-frame bundle of $N_{\text{A}%
/M}$. Using the identification
\[
a:SO\left(  4\right)  =\left(  Sp\left(  1\right)  \times Sp\left(  1\right)
\right)  /\pm\left(  1,1\right)  \rightarrow SO\left(  \mathbb{H}\right)  ,
\]%
\[
\left(  p,q\right)  \cdot y=py\bar{q}\text{, with }p,q\in Sp\left(  1\right)
\text{ and }y\in\mathbb{H}.
\]
Mclean \cite{McLean} showed that the normal bundle $N_{\mathtt{A}/M}$ can be
identified as an associated bundle to $P$ for the representation $SO\left(
4\right)  \rightarrow SO\left(  \mathbb{H}\right)  $ given by $\left(
p,q\right)  \cdot y=py\bar{q}$. The spinor bundle $\mathbb{S}$ of $\mathtt{A}$
is associated to $P$ for the representation $s:SO\left(  4\right)  \rightarrow
SO\left(  \mathbb{H}\right)  $ given by $\left(  p,q\right)  \cdot y=y\bar{q}%
$. Let $E$ be the associated bundle to $P$ for the representation $e:SO\left(
4\right)  \rightarrow SO\left(  \mathbb{H}\right)  $ given by $\left(
p,q\right)  \cdot y=py$. Then because the $3$ representations $a,s$ and $e$
have the relation%
\[
a=s\circ e,
\]
we obtain
\[
N_{\text{A}/M}\cong\mathbb{S}\otimes_{\mathbb{H}}E\text{.}%
\]
An alternative proof of $N_{\mathtt{A}/M}\cong\mathbb{S}\otimes_{\mathbb{H}}E$
using explicit frame identification is contained in the proof of next theorem.

We re-derive McLean's theorem on deformation of associative submanifolds $A$
in a $G_{2}$ manifold $M$. The original proof (Theorem 5.2 in \cite{McLean})
is not quite precise: the associative form $\tau\in\Omega^{3}\left(
M,TM\right)  $ is \emph{vector-valued} rather than a usual differential form,
so the pull back operation and Cartan-formula need to be clarified. The key is
to define suitable notion of pull back for vector-valued forms. Our
calculation is flexible and can be extended to almost associative submanifolds
in later sections. Other proofs were given in \cite{Akbulut Salur} and
\cite{Gayet}.

\begin{theorem}
[McLean]\label{McLean-thm}Under the correspondence of normal vector fields
with twisted spinors, the Zariski tangent space to associative submanifolds at
an associative sub-manifold $A$ is the space of harmonic twisted spinors on
$A$, that is the kernel of the twisted Dirac operator.
\end{theorem}

\begin{proof}
For any section $V$ of $N_{A/M}$ with $C^{0}$ norm smaller than the
injectivity radius $\delta_{0}$ of $M$, we define a nonlinear map%
\begin{align}
F &  :\Gamma\left(  N_{A/M}\right)  \rightarrow\Omega^{3}\left(  A,i^{\ast
}TM\right)  ,\nonumber\\
F\left(  V\right)   &  =T_{V}\circ\left(  \exp V\right)  ^{\ast}%
\tau,\label{pull-vec-form}%
\end{align}
where for the embedding $\exp V:A\rightarrow M$, $\left(  \exp V\right)
^{\ast}$ pulls back the differential form part of the tensor $\tau$, and
$T_{V}:T_{\exp_{p}\left(  tV\right)  }M\rightarrow T_{p}M$ pulls back the
vector part of the tensor $\tau$ by parallel transport along the geodesic
$\exp_{p}\left(  tV\right)  $. There is an ambiguity of the form part and
vector part of tensor $\tau$ up to a scalar function-valued matrix transform
$\Theta$ and $\Theta^{-1}$ respectively, but by the linearity of $T_{V}$ and
$\left(  \exp V\right)  ^{\ast}$ on scalar function factors one can easily
show the definition of $F$ is independent on such $\Theta,$ so $F$ is well-defined.
We make $F$ more explicit by using a good frame. At a point $p\in A$, we pick
two orthonormal vectors $\left\{  W_{1},W_{2}\right\}  $ in $T_{p}A$, \ then
\emph{with respect to the induced connection }$\nabla^{A}$ on $A$, we parallel
transport $\left\{  W_{1},W_{2},W_{3}=W_{1}\times W_{2}\right\}  $ from $p$ to
a neighborhood $B$ in $A$ along geodesic rays from $p$. From the construction
we see
\begin{equation}
\nabla_{W_{i}}^{A}W_{j}\left(  p\right)  =0\text{, for }1\leq i,j\leq3\text{
}\label{orthonormal-p}%
\end{equation}
Then at any $q\in B\subset A$ we have orthonormal basis$\ $%
\[
\left\{  W_{1}\left(  q\right)  ,W_{2}\left(  q\right)  ,W_{3}\left(
q\right)  =W_{1}\left(  q\right)  \times W_{2}\left(  q\right)  \right\}
\]
spanning $T_{q}A$. (This uses that $W_{1}\left(  q\right)  \times W_{2}\left(
q\right)  $ and the parallelly transported $W_{3}\left(  p\right)  $ both
orthogonal to $W_{1}\left(  q\right)  $ and $W_{2}\left(  q\right)  $ in $3$
dimensional $A$). We further choose a smooth normal unit vector field $W_{4}$
on $B$ in $M$ as following: we choose a $W_{4}\in N_{A/M}\left(  p\right)  $
then use the parallel transport of $N_{A/M}$ \emph{with respect to the induced
connection }$\nabla^{\bot}$ in $N_{A/M}$ defined as $\nabla^{\bot}=\bot\nabla$
from the metric on $M$, where $\bot:TM\rightarrow N_{A/M}$ is the natural
projection. Thus on $B$,%
\[
W_{4}\perp\left\{  W_{1},W_{2},W_{1}\times W_{2}\right\}  .
\]
Using Lemma A.15 in \cite{Harvey Lawson} (Cayley--Dickson construction) we can
uniquely extend $\left\{  W_{\alpha}\left(  q\right)  \right\}  _{1,2,3,4}$ to
basis $\left\{  W_{\alpha}\left(  q\right)  \right\}  _{\alpha=1,2,...,7}$ of
$T_{q}M$, such that
\[
W_{i+3}\left(  q\right)  =W_{i}\left(  q\right)  \times W_{4}\left(  q\right)
\]
for $i=1,2,3$, and the correspondence%
\begin{equation}
T_{q}M\ni W_{\alpha}\overset{i}{\leftrightarrow}e_{\alpha}\in\operatorname{Im}%
\mathbb{O},\text{ \ \ }\alpha=1,2,...,7\label{bdl-identification}%
\end{equation}
preserves the inner product $\cdot$ and cross product $\times$, where
$\left\{  e_{\alpha}\right\}  _{\alpha=1,2,...,7}$ is the standard basis of
$\mathbb{R}^{7}\simeq\operatorname{Im}\mathbb{O}$ defined in previous example.
So locally $N_{A/M}$ is trivialized as $B\times\mathbb{H}$, and at any $q\in
B$, by the above basis $W_{\alpha}$ we have algebra isomorphism
\[
T_{q}M=T_{q}A\oplus N_{A/M}\left(  q\right)  \overset{i}{\simeq}%
\operatorname{Im}\mathbb{H\oplus H}\simeq\operatorname{Im}\mathbb{O}%
\]
that smoothly depends on $q\in B$. Hence $N_{A/M}$ is a quaternion valued
bundle over $A$ isomorphic to $\mathbb{S}\otimes_{\mathbb{H}}E$.
From our construction we also have
\begin{equation}
\nabla_{W_{i}}^{\bot}W_{k}\left(  p\right)  =0\text{ for }i=1,2,3\text{ and
}k=4,5,6,7, \label{orthonormal-n}%
\end{equation}
because $\nabla_{W_{i}}^{\bot}W_{4}\left(  p\right)  =0$ by construction of
$W_{4}$, and for $k=5,6,7$ say $k=5$,
\begin{align}
\nabla_{W_{i}}^{\bot}W_{5}\left(  p\right)   &  =\bot\left(  \nabla_{W_{i}%
}\left(  W_{1}\times W_{4}\right)  \right)  \left(  p\right)  =\bot\left(
\nabla_{W_{i}}W_{1}\times W_{4}+W_{1}\times\nabla_{W_{i}}W_{4}\right)  \left(
p\right) \nonumber\\
&  =\nabla_{W_{i}}^{A}W_{1}\left(  p\right)  \times W_{4}+W_{1}\times
\nabla_{W_{i}}^{\bot}W_{4}\left(  p\right)  =0,
\label{orthonormal-compatibility}%
\end{align}
where the second row is because $N_{A/M}\left(  p\right)  \times W_{4}\subset
T_{p}A$ (for $\operatorname{Im}\mathbb{O}\simeq\operatorname{Im}%
\mathbb{H\oplus H}$, $\left(  0,\mathbb{H}\right)  \mathbb{\times}\left(
0,1\right)  \subset\left(  \operatorname{Im}\mathbb{H},0\right)  $ by $\left(
\ref{Cayley-Dickson}\right)  $) and $W_{1}\times T_{p}A\subset T_{p}A$ by
associative condition. We remark that the parallel transport in $N_{A/M}$
w.r.t $\nabla^{\bot}\,$\ is an \emph{isometry}, for if $\nabla_{T}^{\bot}W=0$
for section $W$ in $N_{A/M}$ then
\begin{equation}
\nabla_{T}\left\langle W,W\right\rangle =2\left\langle \nabla_{T}%
W,W\right\rangle =2\left\langle \nabla_{T}^{\bot}W,W\right\rangle =0.
\label{normal-conn-isometry}%
\end{equation}
Next for each $q\in B$, we \emph{parallel transport} the frame $\left\{
W_{\alpha}\left(  q\right)  \right\}  _{\alpha=1,2,...,7}$ along geodesical
rays emanating from $q$ in $M$ in $N_{A/M}\left(  q\right)  $ directions up to
length $\delta_{0}$. This extends the frame to a tubular neighborhood of $A$
in $M$. Then $\nabla_{V}W_{\alpha}\left(  q\right)  =0$ for any $V\in
\Gamma\left(  N_{A/M}\right)  $. If we write
\[
\tau=\omega^{\alpha}\otimes W_{\alpha}\text{ \ \ \ }\left(  \alpha
=1,2,...,7\right)
\]
following Einstein's summation convention, then
\[
F\left(  V\right)  \left(  q\right)  =\left(  \exp V\right)  ^{\ast}%
\omega^{\alpha}\left(  q\right)  \otimes T_{V}W_{\alpha}\left(  \exp
_{q}V\right)  .
\]
We have
\begin{align}
F^{\prime}\left(  0\right)  V &  =\left.  \frac{d}{dt}\right\vert
_{t=0}F\left(  tV\right)  \nonumber\\
&  =\left.  \frac{d}{dt}\right\vert _{t=0}\left[  \left(  \exp tV\right)
^{\ast}\omega^{\alpha}\otimes T_{tV}W_{\alpha}\right]  \nonumber\\
&  =L_{V}\omega^{\alpha}\otimes W_{\alpha}+\omega^{\alpha}\otimes\nabla
_{V}W_{\alpha}\nonumber\\
&  =d\left(  i_{V}\omega^{\alpha}\right)  \otimes W_{\alpha}+i_{V}%
d\omega^{\alpha}\otimes W_{\alpha}+\omega^{\alpha}\otimes\nabla_{V}W_{\alpha
}\label{DF}%
\end{align}
Since $\nabla\tau=0$ and $\nabla_{V}W_{\alpha}\left(  q\right)  =0$ by the
parallel property of $\tau$ and $W_{\alpha}$, we have
\[
0=\nabla_{V}\tau\left(  q\right)  =\nabla_{V}\omega^{\alpha}\otimes W_{\alpha
}\left(  q\right)  +\omega^{\alpha}\otimes\nabla_{V}W_{\alpha}\left(
q\right)  =\nabla_{V}\omega^{\alpha}\left(  q\right)  \otimes W_{\alpha
}\left(  q\right)  ,
\]
and so $\nabla_{V}\omega^{\alpha}\left(  q\right)  =0$. Since $\nabla=d+A$ and
in normal coordinates the connection $1$-form $A$ vanishes at $q$ along the
fiber direction of $N_{A/M}$, we have $i_{V}d\omega^{\alpha}\left(  q\right)
=0$. Therefore by $\left(  \ref{DF}\right)  $, at $q$ we have%
\begin{equation}
F^{\prime}\left(  0\right)  V\left(  q\right)  =d\left(  i_{V}\omega^{\alpha
}\right)  \otimes W_{\alpha}\left(  q\right)  .\label{DFVp}%
\end{equation}
By our choice of $W_{a}$, the $\tau=\omega^{a}\otimes W_{\alpha}$ at $q$ is
the standard form $\left(  \ref{tau-standard}\right)  $, and by the parallel
property of the cross product $\times$ and $\tau$, in the neighborhood of $q$
in $M$ the $\tau$ is also of standard form as $\left(  \ref{tau-standard}%
\right)  $, in the sense that the coordinate vector $\frac{\partial}%
{\partial\omega^{i}}$ is replaced by the frame $W_{i}$, and $\omega
^{ijk}=dx^{i}\wedge$ $dx^{j}\wedge dx^{k}$ is replaced by $W_{i}^{\ast}\wedge
W_{j}^{\ast}\wedge W_{k}^{\ast}$ where $W_{\alpha}^{\ast}$ is the dual vector
of $W_{\alpha}$. Namely%
\[
\tau=\left(  W_{2}^{\ast}\wedge W_{5}^{\ast}\wedge W_{6}^{\ast}-W_{2}^{\ast
}\wedge W_{4}^{\ast}\wedge W_{7}^{\ast}+\cdots\right)  \otimes W_{1}%
+\text{similar terms.}%
\]
This is because $\nabla_{V}\left(  \tau\left(  W_{i},W_{j},W_{k}\right)
\right)  =0$ for $V\in\Gamma\left(  N_{A/M}\right)  $, by $\nabla\tau=0$ and
that $\left\{  W_{\alpha}\right\}  _{\alpha=1,2,...,7}$ is parallel in
$N_{A/M}$ directions. Therefore from $\left(  \ref{DFVp}\right)  $, similar to
the way of deriving $\left(  5.6\right)  $ in \cite{McLean}, and noting
$dW_{i}^{\ast}\left(  p\right)  =0$ $\left(  i=1,2,\cdots7\right)  $ from
$\left(  \ref{orthonormal-p}\right)  $ and $\left(  \ref{orthonormal-n}%
\right)  $, we get%
\begin{align}
F^{\prime}\left(  0\right)  V\left(  p\right)   &  =d\left(  i_{V}%
\omega^{\alpha}\right)  \otimes W_{\alpha}\left(  p\right)
\label{symbol-Dirac}\\
&  =\mathcal{D}V\left(  p\right)  \otimes W_{1}^{\ast}\wedge W_{2}^{\ast
}\wedge W_{3}^{\ast}\left(  p\right)  \nonumber\\
&  =\mathcal{D}V\left(  p\right)  \otimes dvol_{A}\left(  p\right)  ,\nonumber
\end{align}
where for $V=V^{4}W_{4}+V^{5}W_{5}+V^{6}W_{6}+V^{7}W_{7}$,
\begin{align*}
\mathcal{D}V\left(  p\right)   &  =-\left(  V_{1}^{5}+V_{2}^{6}+V_{3}%
^{7}\right)  W_{4}+\left(  V_{1}^{4}+V_{3}^{6}-V_{2}^{7}\right)  W_{5}\\
&  +\left(  V_{2}^{4}-V_{3}^{5}+V_{1}^{7}\right)  W_{6}+\left(  V_{3}%
^{4}+V_{2}^{5}-V_{1}^{6}\right)  W_{7}%
\end{align*}
is the twisted Dirac operator $\left(  5.2\right)  $ in \cite{McLean} and also
$\left(  \ref{twisted-Dirac}\right)  $ below, with $V_{i}^{k}=dV^{k}\left(
W_{i}\right)  $, and $dvol_{A}=W_{1}^{\ast}\wedge W_{2}^{\ast}\wedge
W_{3}^{\ast}$ is the induced volume form on $A\subset M$ since $\left\{
W_{\alpha}\right\}  _{\alpha=1,2,3}$ is the orthonormal basis of $TA$. Since
$p\in A$ and the section $V$ are arbitrary, we have
\begin{equation}
F^{\prime}\left(  0\right)  =\mathcal{D}\otimes dvol_{A}%
.\label{DF-equal-Dirac}%
\end{equation}
Note that both sides of the above identity are independent on the choice of
the frame $\left\{  W_{\alpha}\right\}  _{\alpha=1,2,...,7}$. If the normal
vector field $V$ is induced from deformation of associative submanifolds
$\left\{  A_{t}\right\}  _{0\leq t\leq\varepsilon_{0}}$, i.e.
\[
A_{t}=\exp U\left(  t\right)  \cdot A\text{ \ for }U\left(  t\right)
\in\Gamma\left(  N_{A/M}\right)  \text{ and }U^{\prime}\left(  0\right)  =V,
\]
then differentiating $F\left(  U\left(  t\right)  \right)  =0$ at $t=0$ and
using $\left(  \ref{DF-equal-Dirac}\right)  $ we get $\mathcal{D}V=0$ on $A$,
namely $V$ is a harmonic twisted spinor.
\end{proof}

\begin{remark}
\begin{enumerate}
\item The vector field $V$ is only defined on $A\subset M$, but along the
geodesic rays from $A$ in the fiber directions of $N_{A/M}$, one can extend
$V$ to an open neighborhood of $A$ by parallel transport, therefore the Lie
derivative $L_{V}\omega$ for any $3$-form $\omega$ on $M$ makes sense in this
neighborhood. However, the Lie derivative of $\omega$ restricted on $A$,
namely $i_{A}^{\ast}\left(  L_{V}\omega\right)  $ for the inclusion
$i_{A}:A\hookrightarrow M$, is actually independent of the extension of $V$,
as mentioned in \cite{McLean}. One can see this from the Cartan formula
\[
L_{V}\omega=d\left(  i_{V}\omega\right)  +i_{V}d\omega
\]
as follows: On $A$, the second term $i_{V}d\omega$ is independent on the
extension of $V$. For the first term $d\left(  i_{V}\omega\right)  $, since we
restrict it on $A\subset M$, one can directly check that this term only
involves the derivatives of the section $V$ and $\omega$ in tangent directions
of $A$, for any derivative in normal direction will contribute a covector not
in $T^{\ast}A$ and make the corresponding summand in $i_{A}^{\ast}\left(
L_{V}\omega\right)  $ vanish. So $i_{A}^{\ast}\left(  L_{V}\omega\right)  $ is
independent on the extension of $V$ and $\omega\,$\ to $M$. \

\item In the preceding proof, the normal vector field $V$ is used only to
ensure that $\exp V:A\rightarrow M$ is an embedding and give $d\left(
i_{V}\omega^{\alpha}\right)  \otimes W_{\alpha}$ a twisted Dirac operator
interpretation on the normal bundle $N_{A/M}$. Actually, to write down the
linearization of $F$, one only needs the normal bundle of $A$ in $M$ in
differential topology sense (namely at any $p\in A$, $N_{A/M}\left(  p\right)
+T_{p}A=T_{p}M$, but not necessarily $N_{A/M}\left(  p\right)  \perp T_{p}A$).
If we denote the differential topological normal bundle by $N_{A/M}^{top}$,
then for section $V\in\Gamma\left(  N_{A/M}^{top}\right)  $ with small $C^{0}$
norm, $\exp V:A\rightarrow M$ is an embedding so we can define the
linearization of $F$ as before.

\item The linearization formula
\begin{equation}
F^{\prime}\left(  0\right)  V=d\left(  i_{V}\omega^{\alpha}\right)  \otimes
W_{\alpha}+i_{V}d\omega^{\alpha}\otimes W_{\alpha}+\omega^{\alpha}%
\otimes\nabla_{V}W_{\alpha}\label{DF-general}%
\end{equation}
holds for any section $V\in\Gamma\left(  N_{A/M}^{top}\right)  $, and any
tangent frame $\left\{  W_{a}\right\}  _{i=1,2,\cdots7}$ along any submanifold
$A$ $\subset M$ (not necessarily associative). The term $d\left(  i_{V}%
\omega^{\alpha}\right)  \otimes W_{\alpha}$ is the principal symbol term of
the first order linear differential operator $F^{\prime}\left(  0\right)  $
and behaves functorially under the diffeomorphism between two manifolds. To
get $F^{\prime}\left(  0\right)  V\left(  p\right)  =d\left(  i_{V}%
\omega^{\alpha}\right)  \otimes W_{\alpha}\left(  p\right)  $, one only needs
to choose the frame $\left\{  W_{\alpha}\right\}  _{\alpha=1,\cdots,7}$ that
is parallel along the curves $\exp_{p}\left(  tv\right)  $ for all $p\in A$
and $v\in N_{A/M}^{top}\left(  p\right)  $.

\item On an almost instanton $A$, its normal bundle is not closed under
$\times$ by $TA$ in general, so the linearized instanton equation $F^{\prime
}\left(  0\right)  V$ can NOT be interpreted as a twisted Dirac operator. For
this reason, we will seldom use the twisted Dirac operator, but mainly
$\left(  \ref{DF-general}\right)  $ to do computations and estimates of
$F^{\prime}\left(  0\right)  V$, with the aid of good local frames $\left\{
W_{a}\right\}  _{i=1,2,\cdots7}$.

\item At a point $p$ in a general $3$-manifold $A\subset M$ with volume form
$dvol_{A}$, we can write the linearization
\[
F^{\prime}\left(  0\right)  V\left(  p\right)  =G\circ V\left(  p\right)
\mathbb{\otimes}dvol_{A}\left(  p\right)  ,
\]
where $G$ is a first order linear differential operator, whose principal
symbol depends on the algebraic relation of the frame $\left\{  W_{a}\right\}
$ under the products $\times$ and $\cdot$ , and the volume form on $A$.
Besides the associative $\left\{  W_{a}\right\}  _{i=1,2,3}$+coassociative
$\left\{  W_{a}\right\}  _{i=4,5,6,7}$ type frame, there maybe some other type
of frames leading to a meaningful differential operator $G$.
\end{enumerate}
\end{remark}

By McLean's theorem, the normal bundle to any instanton $A$ is a twisted
spinor bundle $\mathbb{S}$ over $A$ corresponding to the representation
$SO\left(  4\right)  \rightarrow SO\left(  \mathbb{H}\right)  $ given by
$\left(  p,q\right)  \cdot y=py\bar{q}$. Let $\mathcal{D}$ be the twisted
Dirac operator on $A$. We want to write it down explicitly in local
coordinates near $p.$ We let $\left\{  W_{\alpha}\right\}  _{\alpha=1}^{7}$ be
a local orthonormal frame constructed as above. Suppose
\[
V=V^{4}W_{4}+V^{5}W_{5}+V^{6}W_{6}+V^{7}W_{7}%
\]
is a normal vector field to $A$ and we write the covariant differentiation of
$V$ as $\nabla\left(  V\right)  :=V_{i}^{\alpha}W_{\alpha}\otimes\omega^{i}$
with $\left\{  \omega^{i}\right\}  $ being the co-frame dual to $\left\{
W_{i}\right\}  ,$ then the twisted Dirac operator is
\begin{equation}
\mathcal{D}=W_{1}\times\nabla_{1}+W_{2}\times\nabla_{2}+W_{3}\times\nabla_{3},
\label{twisted-Dirac}%
\end{equation}
where $\nabla_{i}:=\nabla_{W_{i}}^{\bot}$ $\left(  i=1,2,3\right)  $. For
instance when the $G_{2}$ manifold is $\operatorname{Im}\mathbb{O}%
\simeq\operatorname{Im}\mathbb{H\oplus H}$ and the instanton is
$\operatorname{Im}\mathbb{H}$, then by viewing $V$ as a $\mathbb{H}$ valued
function, $V=V^{4}+\mathbf{i}V^{5}+\mathbf{j}V^{6}+\mathbf{k}V^{7}$, the
twisted Dirac operator is $\mathcal{D}=\mathbf{i}\nabla_{1}+\mathbf{j}%
\nabla_{2}+\mathbf{k}\nabla_{3}$. Expression $\left(  \ref{twisted-Dirac}%
\right)  $ can be easily shown to be independent on the choice of orthonormal
basis $\left\{  W_{a}\right\}  _{i=1,2,3}$ of $TA$, thus $\mathcal{D}$ is
globally defined on $A$. From $\left(  \ref{twisted-Dirac}\right)  $ and
$\left(  \ref{orthonormal-n}\right)  $ we have
\begin{align}
\mathcal{D}V\left(  p\right)   &  =\left(  W_{1}\times\nabla_{1}+W_{2}%
\times\nabla_{2}+W_{3}\times\nabla_{3}\right)  \left(  V^{4}W_{4}+V^{5}%
W_{5}+V^{6}W_{6}+V^{7}W_{7}\right) \nonumber\\
&  =-\left(  V_{1}^{5}+V_{2}^{6}+V_{3}^{7}\right)  W_{4}+\left(  V_{1}%
^{4}+V_{3}^{6}-V_{2}^{7}\right)  W_{5}\nonumber\\
&  +\left(  V_{2}^{4}-V_{3}^{5}+V_{1}^{7}\right)  W_{6}+\left(  V_{3}%
^{4}+V_{2}^{5}-V_{1}^{6}\right)  W_{7}. \label{McLean-Dirac}%
\end{align}
We remark that away from $p$ the expression of $\mathcal{D}V$ may have more
$0$-th order terms in general.

\section{Dirac Equation on Thin 3-manifolds\label{Dirac-thin-mfd}}

\subsection{A Simplified Model\label{linear-model}}

To motivate our analytical estimates for later sections, let us first
consider\ a simplified model. Suppose that $\mathtt{A}_{\varepsilon}=\left[
0,\varepsilon\right]  \times\Sigma$ is a three manifold and $\left(
x_{1},z\right)  $ are coordinates on $\mathtt{A}_{\varepsilon}\mathtt{.}$ On
$\mathtt{A}_{\varepsilon}$ we put a warped product metric $g_{\mathtt{A}%
_{\varepsilon},h}=h\left(  z\right)  dx_{1}^{2}+g_{\Sigma},$ where $\Sigma$ is
a Riemann surface with a background metric $g_{\Sigma}$ and $h\left(
z\right)  >0$ is a $C^{\infty}$ function on $\Sigma$. Let $e_{1}$ be the unit
tangent vector field on $\mathtt{A}_{\varepsilon}$ normal to $\Sigma$, namely
along the $x_{1}$-direction and $e_{1}=h^{-1/2}\left(  z\right)
\frac{\partial}{\partial x_{1}}$. We introduce a first order linear
differential operator $\mathcal{D}$ that
\begin{equation}
\mathcal{D}=e_{1}\cdot\nabla_{1}+\bar{\partial}=e_{1}\cdot h^{-1/2}\left(
z\right)  \frac{\partial}{\partial x_{1}}+\bar{\partial}, \label{odd-Dirac}%
\end{equation}
where $\nabla_{1}=\nabla_{e_{1}}=h^{-1/2}\left(  z\right)  \frac{\partial
}{\partial x_{1}}$, and $\bar{\partial}=$ $\left(  \overline{\partial
},\overline{\partial}^{\ast}\right)  $ is the Dirac operator on the Dolbeault
complex $\Omega_{\mathbb{C}}^{0,\ast}\left(  L\right)  $ of Hermitian line
bundle $L$ over the Riemann surface $\Sigma$ with a connection $\nabla$ (c.f.
Proposition 3.67 of \cite{BGV} or Proposition 1.4.25 of \cite{Nicolaescu}, by
taking the Kahler manifold to be $\Sigma$ and the Hermitian line bundle to be
$L$). Here $\bar{\partial}$ acts on the spinor bundle $\mathbb{S}_{\Sigma
}:=\mathbb{S}^{+}\oplus\mathbb{S}^{-}$ via the following identification
\begin{equation}%
\begin{tabular}
[c]{lll}%
$\Omega_{\mathbb{C}}^{0}\left(  L\right)  \oplus\Omega_{\mathbb{C}}%
^{0,1}\left(  L\right)  $ & $\overset{\left(  \overline{\partial}%
,\overline{\partial}^{\ast}\right)  }{\longrightarrow}$ & $\Omega_{\mathbb{C}%
}^{0,1}\left(  L\right)  \oplus\Omega_{\mathbb{C}}^{0}\left(  L\right)  $\\
$\ \ \ \ \ \ \ \ \ \ \ \ \ $\ $\ \ \downarrow$ &  &
$\ \ \ \ \ \ \ \ \ \ \ \ \ \ \downarrow$\\
$\ \ \ \ \ \ \ \ \ \ \ \mathbb{S}^{+}\oplus\mathbb{S}^{-}$ & $\overset
{\bar{\partial}}{\longrightarrow}$ & $\ \ \ \ \ \ \ \ \ \ \mathbb{S}^{-}%
\oplus\mathbb{S}^{+}$%
\end{tabular}
\ \ \ \ \ \ \ \ \ \ \ \ \ \ \ \ \ \label{spin-dbar}%
\end{equation}
where on the left $\mathbb{S}^{+}$ and $\mathbb{S}^{-}$ are identified to
complex line bundles $L$ and $L\otimes\Lambda_{\mathbb{C}}^{0,1}\left(
\Sigma\right)  $ respectively, $\overline{\partial}:\Omega_{\mathbb{C}}%
^{0}\left(  L\right)  \rightarrow\Omega_{\mathbb{C}}^{0,1}\left(  L\right)  $
is the Dolbeault operator and $\overline{\partial}^{\ast}:\Omega_{\mathbb{C}%
}^{0,1}\left(  L\right)  \rightarrow\Omega_{\mathbb{C}}^{0}\left(  L\right)  $
is its formal adjoint operator. The Dolbeault operator $\overline{\partial}$
depends on the complex structures $j$ on $\Sigma$, $J$ on $L$, and the
connection $\nabla$ of $L$.

When $h\left(  z\right)  $ is a constant, from the spinor bundle
$\mathbb{S}_{\Sigma}=\mathbb{S}^{+}\oplus\mathbb{S}^{-}$ over $\Sigma$ one can
construct a spinor bundle $\mathbb{S}$ over the odd dimensional manifold
A$_{\varepsilon}:=\left[  0,\varepsilon\right]  \times\Sigma$ by taking the
Cartesian product of $\mathbb{S}_{\Sigma}$ with $\left[  0,\varepsilon\right]
$ (see Chapter 22 in \cite{BW}). When there is no confusion, we also write
$\mathbb{S}=\mathbb{S}^{+}\oplus\mathbb{S}^{-}$ where the $\mathbb{S}^{\pm}$
are the Cartesian products of the $\mathbb{S}^{\pm}$ of $\mathbb{S}_{\Sigma}$
with $\left[  0,\varepsilon\right]  $. \cite{BW} also constructs a Dirac
operator on the spinor bundle $\mathbb{S\rightarrow}$A$_{\varepsilon}$ right
from the Dirac operator $\bar{\partial}$ on the spinor bundle $\mathbb{S}%
_{\Sigma}\mathbb{\rightarrow}\Sigma$, and it is $\mathcal{D=}e_{1}\cdot
\nabla_{1}+\bar{\partial}$ as in $\left(  \ref{odd-Dirac}\right)  $. For
general $h\left(  z\right)  $, our $\mathcal{D}$ is a Dirac type operator but
not necessarily a genuine Dirac operator.

Let's recall the Clifford multiplication of $T$A$_{\varepsilon}$ on
$\mathbb{S}$. Since $\mathbb{S}$ is the Cartesian product of $\mathbb{S}%
_{\Sigma}$ with $\left[  0,\varepsilon\right]  $, it enough to define the
Clifford multiplication $\cdot$ of $e_{1}$ on $\mathbb{S}_{\Sigma}$. Let
$\sigma$ be the volume element of the Clifford bundle $Cl\left(
\Sigma\right)  $, $\sigma^{2}=1$. Then $\mathbb{S}^{+}$ and $\mathbb{S}^{-}$
are the $\pm1$ eigenbundles of $\sigma$ and $\mathbb{S}_{\Sigma}%
=\mathbb{S}^{+}\oplus\mathbb{S}^{-}$. Since the Clifford multiplication
$\cdot$ of $e_{1}$ on $\mathbb{S}_{\Sigma}$ satisfies $e_{1}^{2}=-1$, the
natural choice of the action $e_{1}\cdot$ on $\mathbb{S}_{\Sigma}$ is to let
$\mathbb{S}^{+}$ and $\mathbb{S}^{-}$ be the $\pm i$ eigenbundles, namely for
$\left(  u,v\right)  \in\mathbb{S}^{+}\oplus\mathbb{S}^{-}$, $e_{1}%
\cdot\left(  u,v\right)  =\left(  iu,-iv\right)  $. The connection of
$\mathbb{S}$ along $x_{1}$ direction is trivial.

Given $z_{0}\in\Sigma$, in its neighborhood in $\Sigma$ we can choose a
complex coordinate $z=x_{2}+ix_{3}\in\mathbb{C}$ with $\frac{\partial
}{\partial x_{2}},\frac{\partial}{\partial x_{3}}$ orthonormal at $z_{0}$. We
locally trivialize the spinor bundle $\mathbb{S}_{\Sigma}\mathbb{=\mathbb{S}%
^{+}\oplus\mathbb{S}}^{-}\rightarrow\Sigma$ as
\[
\mathbb{H=C+C}\mathbf{j=\mathbb{C}^{2}\rightarrow}\mathbb{C}.
\]
We may choose the trivialization such that the Dirac operator%
\[
\left.  \left(  \overline{\partial},\overline{\partial}^{\ast}\right)
\right\vert _{z_{0}}=\left(  \overline{\partial}_{z},-\partial_{z}\right)  .
\]
We may write the section $V$ of $\mathbb{S\rightarrow}$ A$_{\varepsilon}$ as
$V=\left(  u,v\right)  :=u+v\mathbf{j\in}\mathbb{S}^{+}\oplus\mathbb{S}^{-}$
with
\begin{align*}
u  &  =V^{4}+\mathbf{i}V^{5}\in\mathbb{S}^{+}\simeq\mathbb{C},\\
v  &  =V^{6}+\mathbf{i}V^{7}\in\mathbb{S}^{+}\simeq\mathbb{C},\mathbf{\ }%
\text{and \ }v\mathbf{j\in}\text{ }\mathbb{S}^{+}\mathbf{j}=\mathbb{S}^{-}.
\end{align*}
With these understood, we may write $\left(  \ref{odd-Dirac}\right)  $ as%
\begin{align}
\mathcal{D}V  &  =\left[
\begin{array}
[c]{cc}%
\mathbf{i} & 0\\
0 & -\mathbf{i}%
\end{array}
\right]  \left(  h^{-1/2}\left(  z\right)  \frac{\partial}{\partial x_{1}%
}+\left[
\begin{array}
[c]{cc}%
0 & -\mathbf{i}\overline{\partial}^{\ast}\\
\mathbf{i}\overline{\partial} & 0
\end{array}
\right]  \right)  \left[
\begin{array}
[c]{c}%
u\\
v
\end{array}
\right] \label{3mfd-Dirac}\\
&  \overset{\text{at }z_{0}}{=}\left[
\begin{array}
[c]{cc}%
\mathbf{i} & 0\\
0 & -\mathbf{i}%
\end{array}
\right]  \left(  h^{-1/2}\left(  z\right)  \frac{\partial}{\partial x_{1}%
}+\left[
\begin{array}
[c]{cc}%
0 & \mathbf{i}\partial_{z}\\
\mathbf{i}\overline{\partial}_{z} & 0
\end{array}
\right]  \right)  \left[
\begin{array}
[c]{c}%
u\\
v
\end{array}
\right] \nonumber\\
&  =\left(  \left(  \nabla_{1}u+\mathbf{i}\partial_{z}v\right)  +\left(
\nabla_{1}v+\mathbf{i}\overline{\partial}_{z}u\right)  \cdot\mathbf{j}\right)
\cdot\mathbf{i}\nonumber\\
&  =\left(  \mathbf{i}\nabla_{1}u-\partial_{z}v\right)  +\left(
-\mathbf{i}\nabla_{1}v\mathbf{+}\overline{\partial}_{z}u\right)
\cdot\mathbf{j}, \label{3mfd-dirac1}%
\end{align}
where at $z_{0}$ we have the identification \
\[
\overline{\partial}=\overline{\partial}_{z}=\frac{1}{2}\left(  \nabla
_{2}+\mathbf{i}\nabla_{3}\right)  ,\text{ }-\overline{\partial}^{\ast
}=\partial_{z}=\frac{1}{2}\left(  \nabla_{2}-\mathbf{i}\nabla_{3}\right)
\]
with $\nabla_{i}=\nabla_{\frac{\partial}{\partial x_{i}}}$ for $i=2,3$ and
$\nabla_{1}=h^{-1/2}\left(  z\right)  \frac{\partial}{\partial x_{1}}$. We
will also denote
\[
-\mathbf{i}\overline{\partial}^{\ast}=\partial^{+}\text{ and }\mathbf{i}%
\overline{\partial}=\partial^{-}%
\]
respectively. They satisfy $\partial^{+}=\left(  \partial^{-}\right)  ^{\ast}$.

This implies that the equation $\mathcal{D}V=0$ \ is equivalent to the
following equations, an analog of Cauchy-Riemann equation,
\begin{align}
\partial^{-}u+h^{-1/2}\left(  z\right)  \frac{\partial v}{\partial x_{1}}  &
\mathbf{=}0,\nonumber\\
\partial^{+}v+h^{-1/2}\left(  z\right)  \frac{\partial u}{\partial x_{1}}  &
\mathbf{=}0. \label{CR}%
\end{align}
We put the boundary condition
\begin{equation}
v|_{\partial\text{A}_{\varepsilon}}=0. \label{bdry-cond-v}%
\end{equation}
(The more precise formulation will be given in $\left(  \ref{bdry-condition}%
\right)  $). It is similar to the totally real boundary condition for
$J$-holomorphic maps. When $h\left(  z\right)  $ is a constant, it is a
special case of the theory of boundary value problems for Dirac operators
developed in \cite{BW}. For general $h\left(  z\right)  $, that the boundary
value problem is Fredholm follows from our elliptic estimates in Subsection
\ref{linear}. Also see \cite{GayetWitt} for relevant discussion.

Later in Section \ref{proof}, we will apply the above $\mathcal{D}$ to the
case when $h\left(  z\right)  =\left\vert n\left(  z\right)  \right\vert ^{2}%
$, where $n\left(  z\right)  =\left.  \frac{d}{dt}C_{t}\right\vert _{t=0}$ is
the normal vector field on $C_{0}$ coming from the deformation of
coassociative submanifolds $C_{t}$, and the Hermitian line bundle $L$ is the
normal bundle $N_{\Sigma_{0}/C_{0}}$ of the $J_{n}$-holomorphic curve
$\Sigma_{0}$ in a coassociative manifold $C_{0}$, whose connection $\nabla$ is
the normal connection from the induced metric of $\Sigma_{0}\subset C_{0}$.

\subsection{\label{1Sec e v estimate}First Eigenvalue Estimates\label{eigen}}

In this subsection we will establish a quantitative estimate of the eigenvalue
of the linearized operator for the simplified model $\mathtt{A}_{\varepsilon
}=\left[  0,\varepsilon\right]  \times\Sigma$ with warped product metric
$g_{\mathtt{A}_{\varepsilon},h}=h\left(  z\right)  dx_{1}^{2}+g_{\Sigma},$
where $\Sigma$ is a \textbf{compact} Riemann surface, and $h\left(  z\right)
$ is a smooth function on $\Sigma$ with $\frac{1}{K}\leq h\leq K$ for some
constant $K>0$. The volume form of this metric on $\mathtt{A}_{\varepsilon}$
is
\[
dvol_{\mathtt{A}_{\varepsilon}}=h^{\frac{1}{2}}\left(  z\right)  dvol_{\Sigma
}dx_{1}.
\]

We introduce the following function spaces for spinors $V=\left(  u,v\right)
$ over $\mathtt{A}_{\varepsilon}$.

\begin{definition}
Let $\mathbb{S}$ be the spinor bundle over $\left(  \mathtt{A}_{\varepsilon
},g_{\mathtt{A}_{\varepsilon}}\right)  $ and $V$ be a smooth section of
$\mathbb{S},$

\begin{enumerate}
\item We define the norm
\[
\left\Vert V\right\Vert _{L^{m,p}\left(  \mathtt{A}_{\varepsilon}%
,\mathbb{S}\right)  }:=\left(  \sum_{\alpha+\beta\leq m}\int_{0}^{\varepsilon
}\int_{\Sigma}\left\vert \left(  \nabla_{x_{1}}\right)  ^{\alpha}\left(
\nabla_{\Sigma}\right)  ^{\beta}V\right\vert ^{p}h^{\frac{1}{2}}\left(
z\right)  dvol_{\Sigma}dx_{1}.\right)  ^{1/p}%
\]
and
\[
\left\Vert V\right\Vert _{C^{m}\left(  \mathtt{A}_{\varepsilon},\mathbb{S}%
\right)  }:=\sum_{\alpha+\beta\leq m}\sup\left\vert \left(  \nabla_{x_{1}%
}\right)  ^{\alpha}\left(  \nabla_{\Sigma}\right)  ^{\beta}V\right\vert
\]
where $\nabla_{x_{1}}$ and $\nabla_{\Sigma}$ are the covariant differentiation
along $x_{1}$-direction and $\Sigma$-directions (i.e. two tangent directions
on $\Sigma$) respectively with respect to the metric $g_{\mathtt{A}%
_{\varepsilon}}$, and the $L^{p}$-norm is with respect to $g_{\mathtt{A}%
_{\varepsilon}}$ too. By the standard Sobolev embedding theorem, we have an
$\varepsilon$ independent constant $C$ such that for any smooth section $V,$%
\begin{align*}
\left\Vert V\right\Vert _{L^{p}\left(  \mathtt{A}_{\varepsilon},\mathbb{S}%
\right)  }  &  \leq C\left\Vert V\right\Vert _{L^{m,q}\left(  \mathtt{A}%
_{\varepsilon},\mathbb{S}\right)  },\text{ for }p\leq\frac{3q}{3-mq}\\
\left\Vert V\right\Vert _{C^{m}\left(  \mathtt{A}_{\varepsilon},\mathbb{S}%
\right)  }  &  \leq C\left\Vert V\right\Vert _{L^{l,p}\left(  \mathtt{A}%
_{\varepsilon},\mathbb{S}\right)  },\text{ for }p\geq\frac{3}{l-m}%
\end{align*}
as long as $\varepsilon$ is bounded below and above by some universal
constant. For the later purpose let us fix $\varepsilon\in\left[
1/2,3/2\right]  .$

\item We define the function spaces
\[
L^{m,p}\left(  \mathtt{A}_{\varepsilon},\mathbb{S}\right)  :=\left\{
V=\left(  u,v\right)  \in\Gamma\left(  \mathtt{A}_{\varepsilon},\mathbb{S}%
\right)  |\left\Vert V\right\Vert _{L^{m,p}\left(  \mathtt{A}_{\varepsilon
},\mathbb{S}\right)  }<+\infty\right\}
\]
and $L_{-}^{m,p}\left(  \mathtt{A}_{\varepsilon},\mathbb{S}\right)  $
(resp.$L_{+}^{m,p}\left(  \mathtt{A}_{\varepsilon},\mathbb{S}\right)  $ ) be
the closure (with respect to the norm $\left\Vert \cdot\right\Vert
_{L^{m,p}\left(  \mathtt{A}_{\varepsilon}\right)  }$) of \ the subspace of
smooth sections $V=\left(  u,v\right)  \in\Gamma\left(  \mathtt{A}%
_{\varepsilon},\mathbb{S}\right)  $ such that $v\in C_{0}^{\infty}\left(
\mathtt{A}_{\varepsilon}\backslash\partial\mathtt{A}_{\varepsilon}\right)  $
(resp. $u\in C_{0}^{\infty}\left(  \mathtt{A}_{\varepsilon}\backslash
\partial\mathtt{A}_{\varepsilon}\right)  $), where $C_{0}^{\infty}\left(
\mathtt{A}_{\varepsilon}\backslash\partial\mathtt{A}_{\varepsilon}\right)  $
denotes the space of smooth functions with compact support inside
$\mathtt{A}_{\varepsilon}\backslash\partial\mathtt{A}_{\varepsilon}.$ Let us
also introduce the space
\begin{align*}
C^{m}\left(  \mathtt{A}_{\varepsilon},\mathbb{S}\right)   &  :=\left\{
V=\left(  u,v\right)  \in\Gamma\left(  \mathtt{A}_{\varepsilon},\mathbb{S}%
\right)  |\left\Vert V\right\Vert _{C^{m}\left(  \mathtt{A}_{\varepsilon
},\mathbb{S}\right)  }<+\infty\right\}  ,\\
C_{-}^{m}\left(  \mathtt{A}_{\varepsilon},\mathbb{S}\right)   &  :=\left\{
V=\left(  u,v\right)  \in\Gamma\left(  \mathtt{A}_{\varepsilon},\mathbb{S}%
\right)  |\left\Vert V\right\Vert _{C^{m}\left(  \mathtt{A}_{\varepsilon
},\mathbb{S}\right)  }<+\infty,v|_{\partial\mathtt{A}_{\varepsilon}%
}=0\right\}  .
\end{align*}

\end{enumerate}
\end{definition}

Let $\mathcal{D}$ be the operator $\left(  \ref{odd-Dirac}\right)  $ in our
linear model. We will impose boundary conditions for sections $V$ that
$\mathcal{D}$ acts on. It is known (c.f. \cite{BW} Theorem 21.5) that the
Dirac operators
\begin{equation}
\mathcal{D}_{\pm}:=\mathcal{D}|_{L_{\pm}^{1,2}}:L_{\pm}^{1,2}\left(
\mathtt{A}_{\varepsilon},\mathbb{S}\right)  \rightarrow L^{2}\left(
\mathtt{A}_{\varepsilon},\mathbb{S}\right)  \label{bdry-condition}%
\end{equation}
give well-posed\emph{ local elliptic boundary problems} and their formal
adjoint operators are $\mathcal{D}_{\pm}^{\ast}=\mathcal{D}_{\mp}$. This
boundary condition restricted on smooth sections $V=\left(  u,v\right)  $ in
$L_{-}^{1,2}\left(  \mathtt{A}_{\varepsilon},\mathbb{S}\right)  $ means
\begin{equation}
v|_{\partial\text{A}_{\varepsilon}}=0\text{, i.e. }v\left(  0,\Sigma\right)
=v\left(  \varepsilon,\Sigma\right)  =0. \label{bdry-condition-v}%
\end{equation}

The following theorem compares the first eigenvalue for Dirac operator
$\bar{\partial}$ on the Riemann surface $\Sigma$ with the operator
$\mathcal{D}$ on product three manifold $\mathtt{A}_{\varepsilon}$. The
theorem refers to the $L^{2}$ metric which is defined as follows: Let $U,V$ be
two sections in $L^{2}\left(  \mathtt{A}_{\varepsilon},\mathbb{S}\right)  $.
We let the inner product $\left\langle \cdot,\cdot\right\rangle _{L^{2}\left(
\mathtt{A}_{\varepsilon},\mathbb{S}\right)  }$ be
\[
\left\langle U,V\right\rangle _{L^{2}\left(  \mathtt{A}_{\varepsilon
},\mathbb{S}\right)  }:=\int_{\left[  0,\varepsilon\right]  }\int_{\Sigma
}\left\langle U,V\right\rangle \left(  x_{1},z\right)  h^{\frac{1}{2}}\left(
z\right)  dvol_{\Sigma}dx_{1},
\]
where in the integral the $\left\langle \cdot,\cdot\right\rangle $ is the
inner product in fibers of the spinor bundle $\mathbb{S}$. We let
\begin{align*}
\left\vert V\left(  x_{1},z\right)  \right\vert ^{2}  &  :=\left\langle
V\left(  x_{1},z\right)  ,V\left(  x_{1},z\right)  \right\rangle \\
\left\Vert V\right\Vert _{L^{2}\left(  \mathtt{A}_{\varepsilon},\mathbb{S}%
\right)  }^{2}  &  :=\left\langle V,V\right\rangle _{L^{2}\left(
\mathtt{A}_{\varepsilon},\mathbb{S}\right)  }.
\end{align*}

\begin{theorem}
\label{lambda}Suppose $\lambda_{\partial^{-}}$ is the first eigenvalue of the
Laplacian\ $\Delta_{\Sigma}=\partial^{+}\partial^{-}$ acting on the space
$L_{-}^{1,2}\left(  \Sigma,\mathbb{S}^{+}\right)  $ and $\lambda_{\partial
^{+}}$ is the first eigenvalue of the Laplacian\ $\Delta_{\Sigma}=\partial
^{-}\partial^{+}$ acting on the space $L_{+}^{1,2}\left(  \Sigma
,\mathbb{S}^{-}\right)  $. Let
\[
\lambda_{\mathcal{D}_{\pm}}:=\inf_{0\neq V\in L_{\pm}^{1,2}\left(
\mathtt{A}_{\varepsilon},\mathbb{S}\right)  }\frac{\left\Vert \mathcal{D}%
V\right\Vert _{L^{2}\left(  \mathtt{A}_{\varepsilon},\mathbb{S}\right)  }^{2}%
}{\left\Vert V\right\Vert _{L^{2}\left(  \mathtt{A}_{\varepsilon}%
,\mathbb{S}\right)  }^{2}}.
\]
(Notice the boundary conditions $\left(  \ref{bdry-condition}\right)  $ and
$\left(  \ref{bdry-condition-v}\right)  $ for $V\,$). Then for the operator
$\mathcal{D}_{\pm}:L_{\pm}^{1,2}\left(  \mathtt{A}_{\varepsilon}%
,\mathbb{S}\right)  \rightarrow L^{2}\left(  \mathtt{A}_{\varepsilon
},\mathbb{S}\right)  $, we have
\[
\lambda_{\mathcal{D}_{\pm}}\geq\min\frac{1}{K}\left\{  \lambda_{\partial^{\pm
}},\frac{2}{K\varepsilon^{2}}-K\left\Vert h^{-\frac{1}{2}}\right\Vert
_{C^{1}\left(  \Sigma\right)  }^{2}\right\}  ,
\]
where $K>0$ is some constant such that $\frac{1}{K}\leq h\left(  z\right)
\leq K$ for all $z\in\Sigma$.
\end{theorem}

\begin{proof}
In the following we will assume the volume form $dvol_{\mathtt{A}%
_{\varepsilon}}=dvol_{\Sigma}dx_{1}$ to simplify calculation. General cases
can be reduced to this case by observing
\begin{equation}
\frac{\left\Vert \mathcal{D}V\right\Vert _{L^{2}\left(  \mathtt{A}%
_{\varepsilon},\mathbb{S}\right)  }^{2}}{\left\Vert V\right\Vert
_{L^{2}\left(  \mathtt{A}_{\varepsilon},\mathbb{S}\right)  }^{2}}\geq
\frac{\int_{\mathtt{A}_{\varepsilon}}\left\vert DV\right\vert ^{2}%
K^{-1/2}dvol_{\Sigma}dx_{1}}{\int_{\mathtt{A}_{\varepsilon}}\left\vert
V\right\vert ^{2}K^{1/2}dvol_{\Sigma}dx_{1}}=\frac{1}{K}\frac{\int
_{\mathtt{A}_{\varepsilon}}\left\vert DV\right\vert ^{2}dvol_{\Sigma}dx_{1}%
}{\int_{\mathtt{A}_{\varepsilon}}\left\vert V\right\vert ^{2}dvol_{\Sigma
}dx_{1}} \label{L2-product-vol}%
\end{equation}
from the condition $\frac{1}{K}\leq h\left(  z\right)  \leq K$. It is enough
to consider the case $\mathcal{D}_{-}:L_{-}^{1,2}\left(  \mathtt{A}%
_{\varepsilon},\mathbb{S}\right)  \rightarrow L^{2}\left(  \mathtt{A}%
_{\varepsilon},\mathbb{S}\right)  $. The $\mathcal{D}_{+}$ case is similar.
For any $V=\left(  u,v\right)  \in L_{-}^{1,2}\left(  \mathtt{A}_{\varepsilon
}\right)  ,\mathcal{\ }$ we have%
\begin{align*}
\left\langle \mathcal{D}V,\mathcal{D}V\right\rangle _{L^{2}\left(
\mathtt{A}_{\varepsilon},\mathbb{S}\right)  } &  =\int_{\mathtt{A}%
_{\varepsilon}}\left(  h^{-1}\left(  z\right)  \left\vert \frac{\partial
V}{\partial x_{1}}\right\vert ^{2}+2\left\langle h^{-\frac{1}{2}}\left(
z\right)  \frac{\partial V}{\partial x_{1}},\left[
\begin{array}
[c]{cc}%
0 & \partial^{+}\\
\partial^{-} & 0
\end{array}
\right]  V\right\rangle \right.  \\
&  \left.  +\left\vert \partial^{+}v\right\vert ^{2}+\left\vert \partial
^{-}u\right\vert ^{2}\right)  dvol_{\mathtt{A}_{\varepsilon}}%
\end{align*}
Using the formula $\partial^{-}=\left(  \partial^{+}\right)  ^{\ast}$, we
have
\begin{align}
&  \int_{\mathtt{A}_{\varepsilon}}\left\langle h^{-\frac{1}{2}}\left(
z\right)  \frac{\partial V}{\partial x_{1}},\left[
\begin{array}
[c]{cc}%
0 & \partial^{+}\\
\partial^{-} & 0
\end{array}
\right]  V\right\rangle dvol_{\mathtt{A}_{\varepsilon}}\nonumber\\
&  =\int_{\mathtt{A}_{\varepsilon}}\left(  \left\langle h^{-\frac{1}{2}}%
\frac{\partial u}{\partial x_{1}},\partial^{+}v\right\rangle +\left\langle
h^{-\frac{1}{2}}\frac{\partial v}{\partial x_{1}},\partial^{-}u\right\rangle
\right)  dvol_{\mathtt{A}_{\varepsilon}}\nonumber\\
\text{(} &  \because\mathbb{S}=\mathbb{S}^{+}\mathbb{\oplus S}^{-}\text{
orthogonal decomposition)}\nonumber\\
&  =\int_{\mathtt{A}_{\varepsilon}}\left(  \left\langle \partial^{-}\left(
h^{-\frac{1}{2}}\frac{\partial u}{\partial x_{1}}\right)  ,v\right\rangle
-\left\langle h^{-\frac{1}{2}}v,\partial^{-}\left(  \frac{\partial u}{\partial
x_{1}}\right)  \right\rangle \right)  dvol_{\mathtt{A}_{\varepsilon}%
}\nonumber\\
&  +\int_{\left\{  \varepsilon\right\}  \times\Sigma}\left\langle h^{-\frac
{1}{2}}v,\partial^{-}u\right\rangle dvol_{\Sigma}-\int_{\left\{  0\right\}
\times\Sigma}\left\langle h^{-\frac{1}{2}}v,\partial^{-}u\right\rangle
dvol_{\Sigma}\nonumber\\
&  =\int_{\mathtt{A}_{\varepsilon}}\left(  \left\langle \partial^{-}\left(
h^{-\frac{1}{2}}\right)  \cdot\frac{\partial u}{\partial x_{1}},v\right\rangle
+\left\langle h^{-\frac{1}{2}}\partial^{-}\left(  \frac{\partial u}{\partial
x_{1}}\right)  ,v\right\rangle \right.  \nonumber\\
&  \left.  -\left\langle h^{-\frac{1}{2}}v,\partial^{-}\left(  \frac{\partial
u}{\partial x_{1}}\right)  \right\rangle \right)  dvol_{\mathtt{A}%
_{\varepsilon}}\text{ (since }v\text{ vanishes on }\partial\mathtt{A}%
_{\varepsilon}\text{)}\nonumber\\
&  =\int_{\mathtt{A}_{\varepsilon}}\left\langle \partial^{-}\left(
h^{-\frac{1}{2}}\right)  \cdot\frac{\partial u}{\partial x_{1}},v\right\rangle
\text{ }dvol_{\mathtt{A}_{\varepsilon}}\text{(}\because\text{the above 2nd and
3rd terms cancel)}\label{Mixing-term}%
\end{align}
Therefore
\begin{align}
&  2\left\vert \int_{\mathtt{A}_{\varepsilon}}\left\langle h^{-\frac{1}{2}%
}\left(  z\right)  \frac{\partial V}{\partial x_{1}},\left[
\begin{array}
[c]{cc}%
0 & \partial^{+}\\
\partial^{-} & 0
\end{array}
\right]  V\right\rangle dvol_{\mathtt{A}_{\varepsilon}}\right\vert \nonumber\\
&  \leq2\int_{\mathtt{A}_{\varepsilon}}\left\vert \left\langle \partial
^{-}\left(  h^{-\frac{1}{2}}\right)  \cdot\frac{\partial u}{\partial x_{1}%
},v\right\rangle dvol_{\mathtt{A}_{\varepsilon}}\right\vert \nonumber\\
&  \leq\left\Vert h^{-\frac{1}{2}}\right\Vert _{C^{1}\left(  \Sigma\right)
}\int_{\mathtt{A}_{\varepsilon}}\left\vert 2\left\langle \frac{\partial
u}{\partial x_{1}},v\right\rangle dvol_{\mathtt{A}_{\varepsilon}}\right\vert
\nonumber\\
&  \leq\int_{\mathtt{A}_{\varepsilon}}\left(  \frac{1}{K}\left\vert u_{x_{1}%
}\right\vert ^{2}+K\left\Vert h^{-\frac{1}{2}}\right\Vert _{C^{1}\left(
\Sigma\right)  }^{2}\left\vert v\right\vert ^{2}\right)  dvol_{\mathtt{A}%
_{\varepsilon}}.\label{mixterm-control}%
\end{align}
In order to estimate $\int_{\mathtt{A}_{\varepsilon}}h^{-1}\left(  z\right)
\left\vert V_{x_{1}}\right\vert ^{2}dvol_{\mathtt{A}_{\varepsilon}}$, we
notice that, for any fixed point $p\in\Sigma$, $v|_{\left[  0,\varepsilon
\right]  \times\left\{  p\right\}  }$ can be treated as a $\mathbb{C}$-valued
function over the interval $\left[  0,\varepsilon\right]  \ $with boundary
value $v\left(  0,z\right)  =0$. Hence
\begin{align}
\int_{0}^{\varepsilon}\left\vert v\right\vert ^{2}dx_{1}  &  =\int
_{0}^{\varepsilon}\left\vert \int_{0}^{x_{1}}\frac{\partial v}{\partial x_{1}%
}\left(  t\right)  dt\right\vert ^{2}dx_{1}\nonumber\\
&  \leq\int_{0}^{\varepsilon}\left(  \int_{0}^{x_{1}}ds\right)  \left(
\int_{0}^{x_{1}}\left\vert \frac{\partial v}{\partial x_{1}}\left(  t\right)
\right\vert ^{2}dt\right)  dx_{1}\nonumber\\
&  \leq\int_{0}^{\varepsilon}x_{1}dx_{1}\int_{0}^{\varepsilon}\left\vert
\frac{\partial v}{\partial x_{1}}\left(  t\right)  \right\vert ^{2}%
dt\nonumber\\
&  \leq K\frac{\varepsilon^{2}}{2}\int_{0}^{\varepsilon}h^{-1}\left(
z\right)  \left\vert \frac{\partial v}{\partial x_{1}}\left(  t\right)
\right\vert ^{2}dt, \label{v-L2}%
\end{align}
where the last inequality is by $\frac{1}{K}\leq h\left(  z\right)  \leq K$
for all $z\in\Sigma$. Putting $\left(  \ref{mixterm-control}\right)  ,\left(
\ref{v-L2}\right)  $ in $\left\langle \mathcal{D}V,\mathcal{D}V\right\rangle
_{L^{2}\left(  \mathtt{A}_{\varepsilon},\mathbb{S}\right)  }$, we have
\begin{align*}
&  \left\langle \mathcal{D}V,\mathcal{D}V\right\rangle _{L^{2}\left(
\mathtt{A}_{\varepsilon},\mathbb{S}\right)  }\\
&  \geq\int_{\mathtt{A}_{\varepsilon}}\left(  h^{-1}\left(  z\right)
\left\vert u_{x_{1}}\right\vert ^{2}+\left\vert \partial^{-}u\right\vert
^{2}+h^{-1}\left(  z\right)  \left\vert v_{x_{1}}\right\vert ^{2}+\left\vert
\partial^{+}v\right\vert ^{2}\right. \\
&  \left.  -2\left\vert \left\langle h^{-\frac{1}{2}}\left(  z\right)
\frac{\partial V}{\partial x_{1}},\left[
\begin{array}
[c]{cc}%
0 & \partial^{+}\\
\partial^{-} & 0
\end{array}
\right]  V\right\rangle \right\vert \right)  dvol_{\mathtt{A}_{\varepsilon}}\\
&  \geq\int_{\mathtt{A}_{\varepsilon}}\left(  \left(  h^{-1}\left(  z\right)
-\frac{1}{K}\right)  \left\vert u_{x_{1}}\right\vert ^{2}+\left\vert
\partial^{-}u\right\vert ^{2}+h^{-1}\left(  z\right)  \left\vert v_{x_{1}%
}\right\vert ^{2}\right. \\
&  \left.  -K\left\Vert h^{-\frac{1}{2}}\right\Vert _{C^{1}\left(
\Sigma\right)  }^{2}\left\vert v\right\vert ^{2}\text{ }\right)
dvol_{\mathtt{A}_{\varepsilon}}\text{(by }\left(  \text{\ref{mixterm-control}%
}\right)  \text{)}\\
&  \geq0+\lambda_{\partial^{-}}\int_{\mathtt{A}_{\varepsilon}}\left\vert
u\right\vert ^{2}dvol_{\mathtt{A}_{\varepsilon}}+\left(  \frac{2}%
{K\varepsilon^{2}}-K\left\Vert h^{-\frac{1}{2}}\right\Vert _{C^{1}\left(
\Sigma\right)  }^{2}\right)  \int_{\mathtt{A}_{\varepsilon}}\left\vert
v\right\vert ^{2}\text{ }dvol_{\mathtt{A}_{\varepsilon}}\\
&  \text{(by definition of }\lambda_{\partial^{-}}\text{ and }\left(
\text{\ref{v-L2}}\right)  \text{)}\\
&  \geq\min\left\{  \lambda_{\partial^{-}},\frac{2}{K\varepsilon^{2}%
}-K\left\Vert h^{-\frac{1}{2}}\right\Vert _{C^{1}\left(  \Sigma\right)  }%
^{2}\right\}  \left(  \int_{\mathtt{A}_{\varepsilon}}\left(  \left\vert
u\right\vert ^{2}+\left\vert v\right\vert ^{2}\right)  dvol_{\mathtt{A}%
_{\varepsilon}}\right) \\
&  =\min\left\{  \lambda_{\partial^{-}},\frac{2}{K\varepsilon^{2}}-K\left\Vert
h^{-\frac{1}{2}}\right\Vert _{C^{1}\left(  \Sigma\right)  }^{2}\right\}
\left\Vert V\right\Vert _{L^{2}\left(  \mathtt{A}_{\varepsilon},\mathbb{S}%
\right)  }^{2}.
\end{align*}
For general volume form $dvol_{\mathtt{A}_{\varepsilon}}=h^{\frac{1}{2}%
}\left(  z\right)  dvol_{\Sigma}dx_{1}$, by $\left(  \ref{L2-product-vol}%
\right)  $ there is an extra factor $\frac{1}{K}$ for the lower bound of the
$L^{2}$ eigenvalue $\lambda_{\mathcal{D}_{\pm}}$. Hence the result.
\end{proof}

Suppose $V=\left(  u,v\right)  \in C^{\infty}\left(  \mathtt{A}_{\varepsilon
},\mathbb{S}\right)  \cap L_{-}^{1,2}\left(  \mathtt{A}_{\varepsilon
},\mathbb{S}\right)  $ and $W=\left(  f,g\right)  \in C^{\infty}\left(
\mathtt{A}_{\varepsilon},\mathbb{S}\right)  \cap L^{2}\left(  \mathtt{A}%
_{\varepsilon},\mathbb{S}\right)  $ and $\mathcal{D}$ is the operator $\left(
\ref{odd-Dirac}\right)  $. Then with respect to the induced volume form
$dvol_{\mathtt{A}_{\varepsilon}}=h^{1/2}\left(  z\right)  d\Sigma dx_{1}$ on
$\left(  \mathtt{A}_{\varepsilon},g_{\mathtt{A}_{\varepsilon},h}\right)  $ we
have
\begin{align}
&  \int_{\mathtt{A}_{\varepsilon}}\left\langle \mathcal{D}V,W\right\rangle
dvol_{\mathtt{A}_{\varepsilon}}\nonumber\\
&  =\int_{0}^{\varepsilon}\int_{\Sigma}\left[  \left\langle \mathbf{i}\left(
h^{-1/2}u_{x_{1}}+\partial^{+}v\right)  ,f\right\rangle -\left\langle
\mathbf{i}\left(  h^{-1/2}v_{x_{1}}+\partial^{-}u\right)  ,g\right\rangle
\right]  h^{1/2}d\Sigma dx_{1}\nonumber\\
&  =\mathbf{i}\int_{0}^{\varepsilon}\int_{\Sigma}\left[  -\left\langle
u,f_{x_{1}}\right\rangle +\left\langle h^{1/2}v,\partial^{-}f\right\rangle
-\left\langle h^{1/2}u,\partial^{+}g\right\rangle +\left\langle v,g_{x_{1}%
}\right\rangle \right]  d\Sigma dx_{1}\nonumber\\
&  +\mathbf{i}\int_{\Sigma}\left(  \left\langle u,f\right\rangle
|_{0}^{\varepsilon}-\left\langle v,g\right\rangle |_{0}^{\varepsilon}\right)
d\Sigma\nonumber\\
&  =\mathbf{i}\int_{0}^{\varepsilon}\int_{\Sigma}\left[  \left\langle
u,-f_{x_{1}}-h^{1/2}\partial^{+}g\right\rangle +\left\langle v,g_{x_{1}%
}+h^{1/2}\partial^{-}f\right\rangle \right]  d\Sigma dx_{1}\nonumber\\
&  +\mathbf{i}\int_{\Sigma}\left(  \left\langle u,f\right\rangle
|_{0}^{\varepsilon}-\left\langle v,g\right\rangle |_{0}^{\varepsilon}\right)
d\Sigma\nonumber\\
&  =\int_{0}^{\varepsilon}\int_{\Sigma}\left\langle u,\mathbf{i}\left(
h^{-1/2}f_{x_{1}}+\partial^{+}g\right)  \right\rangle -\left\langle
v,\mathbf{i}\left(  h^{-1/2}g_{x_{1}}+\partial^{-}f\right)  \right\rangle
\left(  h^{1/2}d\Sigma dx_{1}\right) \nonumber\\
&  +\mathbf{i}\int_{\Sigma}\left(  \left\langle u,f\right\rangle
|_{0}^{\varepsilon}-\left\langle v,g\right\rangle |_{0}^{\varepsilon}\right)
d\Sigma\nonumber\\
&  =\int_{\mathtt{A}_{\varepsilon}}\left\langle V,\mathcal{D}W\right\rangle
dvol_{\mathtt{A}_{\varepsilon}}+\mathbf{i}\int_{\Sigma}\left(  \left\langle
u,f\right\rangle |_{0}^{\varepsilon}-\left\langle v,g\right\rangle
|_{0}^{\varepsilon}\right)  d\Sigma\label{GreenFormula}%
\end{align}
since $h$ is independent of $x_{1}$. This is the Green's formula.

When $V\in L_{-}^{1,2}\left(  \mathtt{A}_{\varepsilon},\mathbb{S}\right)  $
and $W\in L_{+}^{1,2}\left(  \mathtt{A}_{\varepsilon},\mathbb{S}\right)  ,$
the above boundary terms are zero so we have
\[
\int_{\mathtt{A}_{\varepsilon}}\left\langle \mathcal{D}V,W\right\rangle
dvol_{\mathtt{A}_{\varepsilon}}=\int_{\mathtt{A}_{\varepsilon}}\left\langle
V,\mathcal{D}W\right\rangle dvol_{\mathtt{A}_{\varepsilon}}.
\]
This implies that $\mathcal{D}$ is a \emph{self adjoint} operator from
$L_{\pm}^{1,2}\left(  \mathtt{A}_{\varepsilon},\mathbb{S}\right)  $ to
$L_{\mp}^{1,2}\left(  \mathtt{A}_{\varepsilon},\mathbb{S}\right)  $ in the
sense of \cite{BW}. Since $L_{+}^{1,2}\left(  \mathtt{A}_{\varepsilon
},\mathbb{S}\right)  $ is dense in $L^{2}\left(  \mathtt{A}_{\varepsilon
},\mathbb{S}\right)  ,$ this implies that $\mathcal{D}:L_{-}^{1,2}\left(
\mathtt{A}_{\varepsilon},\mathbb{S}\right)  \rightarrow L^{2}\left(
\mathtt{A}_{\varepsilon},\mathbb{S}\right)  $ is surjective if and only if
$\ker\mathcal{D}|_{L_{+}^{1,2}\left(  \mathtt{A}_{\varepsilon},\mathbb{S}%
\right)  }=0\subset L_{+}^{1,2}\left(  \mathtt{A}_{\varepsilon},\mathbb{S}%
\right)  $. By Theorem \ref{lambda}, for small enough $\varepsilon$, we have
\[
\ker\mathcal{D}|_{L_{+}^{1,2}\left(  \mathtt{A}_{\varepsilon},\mathbb{S}%
\right)  }=0\Leftarrow\ker\partial^{-}\partial^{+}=0\Leftrightarrow
\ker\partial^{+}=0,
\]
where the last \textquotedblleft$\Leftrightarrow$\textquotedblright\ is
because $\left(  \partial^{+}\right)  ^{\ast}=\partial^{-}$. Hence we have
obtained the following result.

\begin{theorem}
\label{onto} Let $\lambda_{\partial^{-}}$ and $\lambda_{\partial^{+}}$ be the
first eigenvalue for $\partial^{+}\partial^{-}$ and $\partial^{-}\partial^{+}$
respectively. For $\mathcal{D}:L_{-}^{1,2}(\mathtt{A}_{\varepsilon}%
,\mathbb{S}) \rightarrow L^{2}(\mathtt{A}_{\varepsilon},\mathbb{S})$, if
$\varepsilon$ is sufficiently small, then we have
\begin{align*}
\ker\partial^{-}  &  =\left\{  0\right\}  \Leftrightarrow\lambda_{\partial
^{-}}>0\Rightarrow\mathcal{D}\text{ injective, }\\
\ker\partial^{+}  &  =\left\{  0\right\}  \Leftrightarrow\lambda_{\partial
^{+}}>0\Rightarrow\mathcal{D}\text{ surjective.}%
\end{align*}
Especially if both $\lambda_{\partial^{-}},$ $\lambda_{\partial^{+}}>0$, then
$\mathcal{D}$ is one-to-one and onto.
\end{theorem}

\subsection{Schauder Estimates for Linear Model\label{linear}}

In this section we will develop the necessary linear theory for the equation
\[
\mathcal{D}V=W\text{ on }\mathtt{A}_{\varepsilon}=\Sigma\times\left[
0,\varepsilon\right]
\]
with a warped product metric $g_{\mathtt{A}_{\varepsilon},h}:=h\left(
z\right)  dx_{1}^{2}+g_{\Sigma}$, where $\mathcal{D}$ is the operator $\left(
\ref{odd-Dirac}\right)  $ in our linear model. The key issue is to estimate
the operator norm of the inverse operator of $\mathcal{D}$\ with explicit
dependence of $\varepsilon$, as $\varepsilon$ goes to zero. When $\varepsilon$
is away from zero, say $\varepsilon\in\left[  1/2,3/2\right]  $, we have
$\varepsilon$-free Schauder estimates. For $\varepsilon$ small, we overcome
the difficulty coming from $\varepsilon$ by choosing an appropriate integer
$k$ so that $k\varepsilon\in\left[  1/2,3/2\right]  $ and we extend any
solution $V=\left(  u,v\right)  $ on $\mathtt{A}_{\varepsilon}$ to
$\mathtt{A}_{k\varepsilon}$ in an $L^{p}$ sense by suitable reflection.
However much care will be needed to obtain the $C^{\alpha}$-estimate, because
after the reflection of $W$\ across the boundary of $\mathtt{A}_{\varepsilon
},$ it will no longer be continuous in general. This problem will be resolved
in the case (ii) part of the proof of the following theorem.

Estimating the operator norm of $\mathcal{D}^{-1}$ in the Schauder setting
rather than in the $L^{p}$ setting is crucial. Since our goal is to construct
instantons $A$ governed by the equation $\tau|_{A}=0$, which involves the
associative $3$\emph{-form} $\tau$, for deformations of an approximate
solution $A$ by normal vector fields $V$ in $M$ that are in $W^{1,p}$ class,
the nonlinear equation will have cubic terms of $\nabla V$ that are outside
the $L^{p}$ space. The cubic terms also cause difficulty for obtaining desired
quadratic estimates for the implicit function theorem in the $L^{p}$ setting
(see Remark \ref{Schauder-over-Lp}).

The Schauder estimates are harder to obtain than for the Cauchy-Riemann type
equations, partly because in our equation $\left(  \ref{CR}\right)  $, the
derivative of $v$ only controls the $\partial_{x_{1}}$ and $\bar{\partial}%
_{z}$ derivatives of $u$, not the full derivatives.

We recall the definition of H\"{o}lder\ norms for functions $f$ on a domain
$\Omega$ in $\mathbb{R}^{n}$:
\begin{align}
\left[  f\right]  _{\alpha;\Omega}  &  :=\sup_{\substack{x,y\in\Omega\\x\neq
y}}\frac{\left\vert f\left(  x\right)  -f\left(  y\right)  \right\vert
}{\left\vert x-y\right\vert ^{\alpha}},\label{Schauder-derivative}\\
\left[  f\right]  _{C^{\alpha}\left(  \Omega\right)  }  &  =\left\Vert
f\right\Vert _{C^{0}\left(  \Omega\right)  }+\left[  f\right]  _{\alpha
;\Omega},\nonumber\\
\left[  f\right]  _{C^{1,\alpha}\left(  \Omega\right)  }  &  =\left\Vert
f\right\Vert _{C^{1}\left(  \Omega\right)  }+\left[  \nabla f\right]
_{\alpha;\Omega}.\nonumber
\end{align}
Using trivialization of the bundle $\mathbb{S}$ over $\mathtt{A}_{\varepsilon
}$, the H\"{o}lder\ norms for sections $V$ of the bundle $\mathbb{S}$ are
defined by patching the (finitely many) $C^{1,\alpha}$-coordinate charts on
$\mathtt{A}_{\varepsilon}$.

\begin{theorem}
\label{e-inverse-bound}Let $\mathcal{D}:L_{-}^{1,2}\left(  \mathtt{A}%
_{\varepsilon},\mathbb{S}\right)  \rightarrow L^{2}\left(  \mathtt{A}%
_{\varepsilon},\mathbb{S}\right)  $ be the operator (\ref{CR}) defined on
$\mathtt{A}_{\varepsilon}=\Sigma\times\left[  0,\varepsilon\right]  $ with
warped product metric $g_{\mathtt{A}_{\varepsilon},h}:=h\left(  z\right)
dx_{1}^{2}+g_{\Sigma}$. Suppose that the first eigenvalues for $\partial
^{-}\partial^{+}$ and $\partial^{+}\partial^{-}$ are bounded below by
$\lambda>0.$ Then for any $0<\alpha<1$ and $p>3$ there is a positive constant
$C=C\left(  \alpha,p,\lambda,h\right)  $ independent of $\varepsilon$ such
that for any $V\in C_{-}^{1,\alpha}\left(  \mathtt{A}_{\varepsilon}%
,\mathbb{S}\right)  $ and $W\in C^{\alpha}\left(  \mathtt{A}_{\varepsilon
},\mathbb{S}\right)  $ satisfying
\[
\mathcal{D}V=W
\]
we have
\[
C\left\Vert V\right\Vert _{C_{-}^{1,\alpha}\left(  \mathtt{A}_{\varepsilon
},\mathbb{S}\right)  }\leq\varepsilon^{-\left(  \frac{3}{p}+2\alpha\right)
}\left\Vert W\right\Vert _{C^{\alpha}\left(  \mathtt{A}_{\varepsilon
},\mathbb{S}\right)  }.
\]
In other words, there exists a right inverse
\[
Q_{\varepsilon}:C^{\alpha}\left(  \mathtt{A}_{\varepsilon},\mathbb{S}\right)
\rightarrow C_{-}^{1,\alpha}\left(  \mathtt{A}_{\varepsilon},\mathbb{S}%
\right)
\]
of $\mathcal{D}:C_{-}^{1,\alpha}\left(  \mathtt{A}_{\varepsilon}%
,\mathbb{S}\right)  \rightarrow C^{\alpha}\left(  \mathtt{A}_{\varepsilon
},\mathbb{S}\right)  $ such that $\left\Vert Q_{\varepsilon}\right\Vert \leq
C\varepsilon^{-\left(  \frac{3}{p}+2\alpha\right)  }$.
\end{theorem}

\begin{proof}
For the operator (\ref{CR}), it is known (c.f. \cite{BW} Theorem 21.5) that
the Dirac operators
\[
\mathcal{D}_{\pm}:=\mathcal{D}|_{L_{\pm}^{1,2}}:L_{\pm}^{1,2}\left(
\mathtt{A}_{\varepsilon},\mathbb{S}\right)  \rightarrow L^{2}\left(
\mathtt{A}_{\varepsilon},\mathbb{S}\right)
\]
give well-posed local elliptic boundary problems. Using the orthogonal
decomposition $\mathbb{S=S}^{+}\mathbb{\oplus S}^{-}$ we write sections
$V=\left(  u,v\right)  \in C_{-}^{\infty}\left(  \mathtt{A}_{\varepsilon
},\mathbb{S}\right)  $ and $W=\left(  w_{1},w_{2}\right)  \in C^{\infty
}\left(  \mathtt{A}_{\varepsilon},\mathbb{S}\right)  $, where $u,w_{1}$ are
sections of $\mathbb{S}^{+}$ and $v,w_{2}$ are sections of $\mathbb{S}^{-}$.
So the equation $\mathcal{D}V=W$ may be explicitly written as (c.f. equation
(\ref{CR}))
\[
\left\{
\begin{array}
[c]{ccc}%
h^{-1/2}\left(  z\right)  u_{x_{1}}+\partial^{+}v & = & w_{1}\\
h^{-1/2}\left(  z\right)  v_{x_{1}}+\partial^{-}u & = & w_{2}%
\end{array}
\right.  \text{ \ with }v|_{\partial\mathtt{A}_{\varepsilon}}=0.
\]
To make the exposition more transparent we will assume that $h\equiv1$, and it
is clear from the proof below that the argument works equally well for any
$h\left(  z\right)  \in C^{\infty}\left(  \Sigma\right)  $ such that $\frac
{1}{K}\leq h\left(  z\right)  \leq K$ for some constant $K>0$. For any
$0<\varepsilon<3/2$ we take an integer $k=k\left(  \varepsilon\right)  >0$
(which \emph{depends} on $\varepsilon$) such that
\[
k\left(  \varepsilon\right)  \varepsilon\in\left[  1/2,3/2\right]  .
\]
For notation brevity, we will write $k$ for $k\left(  \varepsilon\right)  $ in
the remainder of our paper. Then A$_{k\varepsilon}=\Sigma\times\left[
0,k\varepsilon\right]  $ is a product region whose second component has length
\emph{uniformly bounded below and above} for any $\varepsilon$. We divide the
estimates into two cases:
\emph{Case (i)}: Suppose that $w_{1}=0,$ then we will have along the boundary
$\partial\mathtt{A}_{\varepsilon},$ $u_{x_{1}}=0$ since $v=0$. Since
$v|_{\partial\mathtt{A}_{\varepsilon}}=0$, we can extend $v$ from
$\mathtt{A}_{\varepsilon}$ to $\mathtt{A}_{k\varepsilon}$ by odd reflection
along the walls $\Sigma\times\left\{  j\varepsilon\right\}  $ with $0\leq
j\leq k-1,$ while still keeping $v$ in $C^{1,\alpha}\left(  A_{k\varepsilon
}\right)  $. Similarly we consider an even extension of $u$ to $\mathtt{A}%
_{k\varepsilon}$, then $u$ is still in $C^{1,\alpha}\left(  A_{k\varepsilon
}\right)  $ since $u_{x_{1}}|_{\partial\mathtt{A}_{\varepsilon}}=0$. The
extension formula is
\begin{align*}
v\left(  x,z\right)   &  =\left\{
\begin{array}
[c]{ccc}%
-v\left(  \left(  2j+2\right)  \varepsilon-x,z\right)  &  & \text{for }%
x\in\left[  \left(  2j+1\right)  \varepsilon,\left(  2j+2\right)
\varepsilon\right] \\
v\left(  x-2j\varepsilon,z\right)  &  & \text{for }x\in\left[  2j\varepsilon
,\left(  2j+1\right)  \varepsilon\right]
\end{array}
\right.  ,\\
u\left(  x,z\right)   &  =\left\{
\begin{array}
[c]{ccc}%
u\left(  \left(  2j+2\right)  \varepsilon-x,z\right)  &  & \text{for }%
x\in\left[  \left(  2j+1\right)  \varepsilon,\left(  2j+2\right)
\varepsilon\right] \\
u\left(  x-2j\varepsilon,z\right)  &  & \text{for }x\in\left[  2j\varepsilon
,\left(  2j+1\right)  \varepsilon\right]
\end{array}
\right.  .
\end{align*}
This will induce an even extension of $w_{2}$ so that the equation
$\mathcal{D}V=W$ is satisfied in the $C^{\alpha}$ sense on $\mathtt{A}%
_{k\varepsilon}.$ The motivation of even and odd extension is to have an
$\varepsilon$-independent $L^{p}$ and Schauder estimate.
By differentiating both sides of the equation $v_{x_{1}}+\partial^{-}u=w_{2}$
with respect to $x_{1},$ we obtain an equation which is equivalent to the
Dirichlet problem of the second order elliptic equation
\begin{align*}
v_{x_{1}x_{1}}-\partial^{-}\partial^{+}v &  =\frac{\partial w_{2}}{\partial
x_{1}}\text{ and }v|_{\partial\mathtt{A}_{k\varepsilon}}=0,\\
u_{x_{1}x_{1}}-\partial^{+}\partial^{-}u &  =-\partial^{+}w_{2},
\end{align*}
noting that $\partial^{-}\partial^{+}$ is a positive operator. (Since $u,v$
are only in $C^{1,\alpha}$, they should be understood as weak solutions of the
above equations). Since the $C^{\alpha}$-norm is preserved under the above
extension, Schauder and $L^{p}$ interior estimate for the second order
elliptic equation on the region
\[
\mathtt{A}_{k\varepsilon}\subset\mathtt{A}_{2k\varepsilon}\cup\mathtt{A}%
_{-2k\varepsilon}%
\]
would then imply that there are constants $C\left(  \alpha\right)  $ and
$\tilde{C}\left(  p\right)  $ independent of $\varepsilon$
(because$\ k\varepsilon\in\left[  1/2,3/2\right]  $) such that%
\begin{align}
\left\Vert w_{2}\right\Vert _{C^{\alpha}\left(  \mathtt{A}_{\varepsilon
},\mathbb{S}^{-}\right)  }+\left\Vert V\right\Vert _{C^{0}\left(
\mathtt{A}_{\varepsilon},\mathbb{S}\right)  } &  =\left\Vert w_{2}\right\Vert
_{C^{\alpha}\left(  \mathtt{A}_{2k\varepsilon}\cup\mathtt{A}_{-2k\varepsilon
},\mathbb{S}^{-}\right)  }+\left\Vert V\right\Vert _{C^{0}\left(
\mathtt{A}_{2k\varepsilon}\cup\mathtt{A}_{-2k\varepsilon},\mathbb{S}\right)
}\nonumber\\
&  \geq C\left(  \alpha\right)  \left\Vert V\right\Vert _{C^{1,\alpha}\left(
\mathtt{A}_{k\varepsilon},\mathbb{S}\right)  }\nonumber\\
&  =C\left(  \alpha\right)  \left\Vert V\right\Vert _{C_{-}^{1,\alpha}\left(
\mathtt{A}_{\varepsilon},\mathbb{S}\right)  },\label{Schauder-V}%
\end{align}
and%
\[
\left\Vert w_{2}\right\Vert _{L^{p}\left(  \mathtt{A}_{2k\varepsilon}%
\cup\mathtt{A}_{-2k\varepsilon},\mathbb{S}^{-}\right)  }+\left\Vert
V\right\Vert _{L^{p}\left(  \mathtt{A}_{2k\varepsilon}\cup\mathtt{A}%
_{-2k\varepsilon},\mathbb{S}\right)  }\geq4\tilde{C}\left(  p\right)
\left\Vert V\right\Vert _{L_{-}^{1,p}\left(  \mathtt{A}_{k\varepsilon
},\mathbb{S}\right)  }%
\]
\ (c.f. \cite{GT} section 8.11 Theorem 8.32 and \cite{MS} Theorem B.3.2
respectively), so%
\[
\left\Vert w_{2}\right\Vert _{L^{p}\left(  \mathtt{A}_{\varepsilon}%
,\mathbb{S}^{-}\right)  }+\left\Vert V\right\Vert _{L^{p}\left(
\mathtt{A}_{\varepsilon},\mathbb{S}\right)  }\geq\tilde{C}\left(  p\right)
\left\Vert V\right\Vert _{L_{-}^{1,p}\left(  \mathtt{A}_{\varepsilon
},\mathbb{S}\right)  }%
\]
by periodicity of reflection. We also have
\[
\left\Vert V\right\Vert _{C^{0}\left(  \mathtt{A}_{\varepsilon},\mathbb{S}%
\right)  }\leq\left\Vert V\right\Vert _{C_{-}^{1-3/p}\left(  \mathtt{A}%
_{\varepsilon},\mathbb{S}\right)  }\leq C\left(  p,\lambda\right)  \left\Vert
W\right\Vert _{C^{0}\left(  \mathtt{A}_{\varepsilon},\mathbb{S}\right)  },
\]
whose proof is identical to that of $\left(  \ref{C-Sch-p}\right)  $ in the
following case (ii). Plugging this in $\left(  \ref{Schauder-V}\right)  $ we
have%
\begin{equation}
C\left(  \alpha\right)  \left\Vert V\right\Vert _{C_{-}^{1,\alpha}\left(
\mathtt{A}_{\varepsilon},\mathbb{S}\right)  }\leq\left(  C\left(
p,\lambda\right)  +1\right)  \left\Vert w_{2}\right\Vert _{C^{\alpha}\left(
\mathtt{A}_{\varepsilon},\mathbb{S}^{-}\right)  }.\label{V-W-case1}%
\end{equation}
\emph{Case (ii)}: suppose that $w_{2}=0$ and $w_{1}\in C^{\alpha}\left(
\mathtt{A}_{\varepsilon},\mathbb{S}^{+}\right)  .$ Since $v=0$, $\ $this
implies that if we consider the odd extension of $v$ and the even extension of
$u$ to $\mathtt{A}_{2k\varepsilon}\cup$ $\mathtt{A}_{-2k\varepsilon}$ as in
the previous case, then they induce an odd extension of $w_{1}$ so that the
equation
\begin{equation}
\left\{
\begin{array}
[c]{ccc}%
u_{x_{1}}+\partial^{+}v & = & w_{1}\\
v_{x_{1}}+\partial^{-}u & = & 0
\end{array}
\right.  \label{CR1}%
\end{equation}
is satisfied in the weak sense on $\mathtt{A}_{2k\varepsilon}\cup
\mathtt{A}_{-2k\varepsilon}.$ Notice that $w_{1}$ does not vanish on
$\partial\mathtt{A}_{\varepsilon}$ in general, so after the odd extension,
$w_{1}$ is no longer continuous but still we have $w_{1}\in L^{p}\left(
\mathtt{A}_{2k\varepsilon}\cup\mathtt{A}_{-2k\varepsilon},\mathbb{S}\right)  $
for $\forall p$. Since A$_{k\varepsilon}$ is a proper subdomain in
$\mathtt{A}_{2k\varepsilon}\cup\mathtt{A}_{-2k\varepsilon}$ and $k\varepsilon
\in\left[  1/2,3/2\right]  ,$ the interior $L^{p}$-estimate then implies that
there is a constant $\tilde{C}\left(  p\right)  $ independent of $\varepsilon$
such that%
\[
\left\Vert w_{1}\right\Vert _{L^{p}\left(  \mathtt{A}_{2k\varepsilon}%
\cup\mathtt{A}_{-2k\varepsilon},\mathbb{S}^{-}\right)  }+\left\Vert
V\right\Vert _{L^{p}\left(  \mathtt{A}_{2k\varepsilon}\cup\mathtt{A}%
_{-2k\varepsilon},\mathbb{S}\right)  }\geq4\tilde{C}\left(  p\right)
\left\Vert V\right\Vert _{L_{-}^{1,p}\left(  \mathtt{A}_{k\varepsilon
},\mathbb{S}\right)  }.
\]
By the periodicity of $w_{1}$ and $V,$ that is
\begin{equation}
\left\Vert w_{1}\right\Vert _{L^{p}\left(  \mathtt{A}_{k\varepsilon
},\mathbb{S}^{-}\right)  }+\left\Vert V\right\Vert _{L^{p}\left(
\mathtt{A}_{k\varepsilon},\mathbb{S}\right)  }\geq\tilde{C}\left(  p\right)
\left\Vert V\right\Vert _{L_{-}^{1,p}\left(  \mathtt{A}_{k\varepsilon
},\mathbb{S}\right)  }.\label{Lp}%
\end{equation}
Then we proceed to the $L^{p\text{ }}$estimates of $V$ purely in terms of $W$.
It follows from (\ref{Lp}) and Theorem \ref{lambda} that for $\varepsilon<3/2$
we have
\[
\left\Vert W\right\Vert _{L^{2}\left(  \mathtt{A}_{\varepsilon},\mathbb{S}%
\right)  }\geq C\left(  \lambda\right)  \left\Vert V\right\Vert _{L^{2}\left(
\mathtt{A}_{\varepsilon},\mathbb{S}\right)  }\text{ for }V\in C_{-}^{\infty
}\left(  \mathtt{A}_{\varepsilon},\mathbb{S}\right)
\]
with the constant $C\left(  \lambda\right)  $ dependent on $\lambda
_{\partial^{+}}$ but independent of $\varepsilon.$ To go from $L^{2}$ to
$L^{p}$, \ note the following interpolation\ inequality
\begin{align*}
\left\Vert V\right\Vert _{L^{p}\left(  \mathtt{A}_{k\varepsilon}%
,\mathbb{S}\right)  } &  =\left(  \int_{\mathtt{A}_{k\varepsilon}}\left\vert
V\right\vert ^{p}\right)  ^{1/p}\leq\left\Vert V\right\Vert _{C^{0}\left(
\mathtt{A}_{k\varepsilon},\mathbb{S}\right)  }^{\frac{p-1}{p}}\cdot\left(
\int_{\mathtt{A}_{k\varepsilon}}\left\vert V\right\vert \right)  ^{1/p}\\
&  \leq\frac{p-1}{p}\delta^{\frac{p}{p-1}}\left\Vert V\right\Vert
_{C^{0}\left(  \mathtt{A}_{k\varepsilon},\mathbb{S}\right)  }^{\frac{p-1}%
{p}\cdot\frac{p}{p-1}}+\frac{1}{p}\frac{1}{\delta^{p}}\left(  \int
_{\mathtt{A}_{k\varepsilon}}\left\vert V\right\vert \right)  ^{\frac{p}{p}%
}\text{ }\\
&  \text{(By Young's Inequality)}\\
&  =\frac{p-1}{p}\delta^{\frac{p}{p-1}}\left\Vert V\right\Vert _{C^{0}\left(
\mathtt{A}_{k\varepsilon},\mathbb{S}\right)  }+\frac{1}{p}\frac{1}{\delta^{p}%
}\left\Vert V\right\Vert _{L^{1}\left(  \mathtt{A}_{k\varepsilon}%
,\mathbb{S}\right)  }\\
&  \leq C\left(  \delta^{\frac{p}{p-1}}\left\Vert V\right\Vert _{L^{1,p}%
\left(  \mathtt{A}_{k\varepsilon},\mathbb{S}\right)  }+\delta^{-p}\left\Vert
V\right\Vert _{L^{2}\left(  \mathtt{A}_{k\varepsilon},\mathbb{S}\right)
}\right)  \text{ }\\
&  \text{(By Sobolev embedding),}%
\end{align*}
where $C$ is independent of $\varepsilon$ in the last inequality because
$k\varepsilon\in\left[  1/2,3/2\right]  $. Putting this in $\left(
\ref{Lp}\right)  $ we have
\begin{align*}
\left(  \tilde{C}\left(  p\right)  -C\delta^{\frac{p}{p-1}}\right)  \left\Vert
V\right\Vert _{L^{1,p}\left(  \mathtt{A}_{k\varepsilon},\mathbb{S}\right)  }
&  \leq\left\Vert W\right\Vert _{L^{p}\left(  \mathtt{A}_{k\varepsilon
},\mathbb{S}\right)  }+C\delta^{-p}\left\Vert V\right\Vert _{L^{2}\left(
\mathtt{A}_{k\varepsilon},\mathbb{S}\right)  }\\
&  \leq\left\Vert W\right\Vert _{L^{p}\left(  \mathtt{A}_{k\varepsilon
},\mathbb{S}\right)  }+\frac{C\delta^{-p}}{C\left(  \lambda\right)
}\left\Vert W\right\Vert _{L^{2}\left(  \mathtt{A}_{k\varepsilon}%
,\mathbb{S}\right)  }\\
&  \leq\left\Vert W\right\Vert _{L^{p}\left(  \mathtt{A}_{k\varepsilon
},\mathbb{S}\right)  }\left(  1+\frac{C\delta^{-p}}{C\left(  \lambda\right)
}\right)  .
\end{align*}
So
\begin{equation}
\left\Vert V\right\Vert _{L_{-}^{1,p}\left(  \mathtt{A}_{k\varepsilon
},\mathbb{S}\right)  }\leq\left(  \tilde{C}\left(  p\right)  -C\delta
^{\frac{p}{p-1}}\right)  ^{-1}\left(  1+\frac{C\delta^{-p}}{C\left(
\lambda\right)  }\right)  \left\Vert W\right\Vert _{L^{p}\left(
\mathtt{A}_{k\varepsilon},\mathbb{S}\right)  }.\label{V-Lp}%
\end{equation}
Fixing a small $\delta$ such that $\tilde{C}\left(  p\right)  -C\delta
^{\frac{p}{p-1}}>0$ and let the constant
\[
\tilde{C}\left(p,\lambda\right)  :=\left(  \tilde{C}\left(  p\right)
-C\delta^{\frac{p}{p-1}}\right)  ^{-1}\left(  1+\frac{C\delta^{-p}}{C\left(
\lambda\right)  }\right)
\]
which is independent of $\varepsilon$. Then for $p>3,$ by Sobolev embedding
$C_{-}^{1-3/p}\left(  \mathtt{A}_{k\varepsilon},\mathbb{S}\right)
\hookrightarrow L_{-}^{1,p}\left(  \mathtt{A}_{k\varepsilon},\mathbb{S}%
\right)  $ and $\left(  \ref{V-Lp}\right)  $ we have
\begin{align}
\left\Vert V\right\Vert _{C_{-}^{1-3/p}\left(  \mathtt{A}_{\varepsilon
},\mathbb{S}\right)  } &  =\left\Vert V\right\Vert _{C_{-}^{1-3/p}\left(
\mathtt{A}_{k\varepsilon},\mathbb{S}\right)  }\nonumber\\
&  \leq C\left\Vert V\right\Vert _{L_{-}^{1,p}\left(  \mathtt{A}%
_{k\varepsilon},\mathbb{S}\right)  }\leq C\cdot\tilde{C}\left(  p,\lambda
\right)  \left\Vert W\right\Vert _{L^{p}\left(  \mathtt{A}_{k\varepsilon
},\mathbb{S}\right)  }\nonumber\\
&  \leq C\left(  p,\lambda\right)  \left\Vert W\right\Vert _{C^{0}\left(
\mathtt{A}_{k\varepsilon},\mathbb{S}\right)  }=C\left(  p,\lambda\right)
\left\Vert W\right\Vert _{C^{0}\left(  \mathtt{A}_{\varepsilon},\mathbb{S}%
\right)  },\label{C-Sch-p}%
\end{align}
where the constant $C\left(  p,\lambda\right)  =C\cdot\tilde{C}\left(
p,\lambda\right)  \left(  \frac{3}{2}vol\left(  \Sigma\right)  \right)
^{\frac{1}{p}}$ is independent of $\varepsilon$. We remark that the above
argument also works in case (i), so $\left(  \ref{C-Sch-p}\right)  $ holds in
all cases.
In the following we give the Schauder estimate of $V=\left(  u,v\right)  $.
Differentiating the second equation in (\ref{CR1}) with respect to $x_{1},$ we
get the Dirichlet problem of the second order elliptic equation for (weak
solutions) $v$ and $u$:
\begin{align}
v_{x_{1}x_{1}}-\partial^{-}\partial^{+}v  &  =-\partial^{-}w_{1}\text{ \ \ and
\ \ }v|_{\partial\mathtt{A}_{\varepsilon}}=0,\label{Dirichlet-uv}\\
u_{x_{1}x_{1}}-\partial^{+}\partial^{-}u  &  =\partial_{x_{1}}w_{1}.\nonumber
\end{align}
Recall that in \cite{GT} chapter 4, the \emph{non-dimensional }Schauder norm
$\left\Vert f\right\Vert _{C^{k,\alpha}\left(  \overline{\Omega}\right)
}^{\prime}$ is defined as%
\begin{equation}
\left\Vert f\right\Vert _{C^{k,\alpha}\left(  \overline{\Omega}\right)
}^{\prime}=\Sigma_{j=0}^{k}d^{j}\left\vert D^{j}u\right\vert _{0;\Omega
}+d^{k+\alpha}\left[  D^{k}u\right]  _{\alpha;\Omega},
\label{Schauder-norm-Prime}%
\end{equation}
where $d$ is the diameter of $\Omega$.
For component $v$, since $v|_{\partial\mathtt{A}_{\varepsilon}}=0,$ the
standard Schauder estimate for second order elliptic equation on radius
$\varepsilon$ half balls $B_{\varepsilon}$ with centers on $\partial
\mathtt{A}_{\varepsilon}$ (c.f. \cite{GT} Section 8.11 for regularity of weak
solutions and Section 4.4 equation $\left(  4.43\right)  $) gives
\[
\varepsilon\left\Vert w_{1}\right\Vert _{C^{\alpha}\left(  B_{\varepsilon
},\mathbb{S}^{+}\right)  }^{\prime}+\left\Vert v\right\Vert _{C_{-}^{0}\left(
B_{\varepsilon},\mathbb{S}\right)  }\geq C\left(  \alpha\right)  \left\Vert
v\right\Vert _{C_{-}^{1,\alpha}\left(  B_{\varepsilon/2},\mathbb{S}\right)
,}^{\prime}%
\]
or equivalently, by the definition of $\left\Vert f\right\Vert _{C^{k,\alpha
}\left(  \overline{\Omega}\right)  }^{\prime}$,
\begin{align}
&  \varepsilon\left\Vert w_{1}\right\Vert _{C^{\alpha}\left(  B_{\varepsilon
},\mathbb{S}^{+}\right)  ^{\prime}}+\left\Vert v\right\Vert _{C_{-}^{0}\left(
B_{\varepsilon},\mathbb{S}\right)  }\nonumber\\
&  \geq C\left(  \alpha\right)  \left[  \left\Vert v\right\Vert _{C_{-}%
^{0}\left(  B_{\varepsilon/2},\mathbb{S}\right)  }+\varepsilon\left\Vert
\nabla v\right\Vert _{C_{-}^{0}\left(  B_{\varepsilon/2},\mathbb{S}\right)
}+\varepsilon^{1+\alpha}\left[  \nabla v\right]  _{\alpha;\left(
B_{\varepsilon/2},\mathbb{S}\right)  }\right]  ,\label{e-Schauder-v}%
\end{align}
hence $\qquad$%
\begin{equation}
\varepsilon\left\Vert w_{1}\right\Vert _{C^{\alpha}\left(  B_{\varepsilon
},\mathbb{S}^{+}\right)  }+\left\Vert v\right\Vert _{C_{-}^{0}\left(
B_{\varepsilon},\mathbb{S}\right)  }\geq C\left(  \alpha\right)
\varepsilon^{1+\alpha}\left\Vert v\right\Vert _{C_{-}^{1,\alpha}\left(
B_{\varepsilon/2},\mathbb{S}\right)  },\label{v-ball}%
\end{equation}
where the constant $C\left(  \alpha\right)  $ is independent of $\varepsilon$.
Covering $\mathtt{A}_{\varepsilon}$ by such radius-$\varepsilon$ half balls
and radius-$\varepsilon/2$ full balls centered on meridian $\Sigma
\times\left\{  \varepsilon/2\right\}  $ and take supremum on $\mathtt{A}%
_{\varepsilon}$, we have
\begin{equation}
\varepsilon\left\Vert w_{1}\right\Vert _{C^{\alpha}\left(  \mathtt{A}%
_{\varepsilon},\mathbb{S}^{+}\right)  }+\left\Vert v\right\Vert _{C_{-}%
^{0}\left(  \mathtt{A}_{\varepsilon},\mathbb{S}\right)  }\geq C\left(
\alpha\right)  \varepsilon^{1+\alpha}\left\Vert v\right\Vert _{C_{-}%
^{1,\alpha}\left(  \mathtt{A}_{\varepsilon},\mathbb{S}\right)  }%
.\label{c-alpha}%
\end{equation}
\ By combining this with the fact that $v|_{\partial\mathtt{A}_{\varepsilon
}=0}$, and the definition of the Schauder norm, we have
\begin{equation}
\left\Vert v\right\Vert _{C_{-}^{0}\left(  \mathtt{A}_{\varepsilon}%
,\mathbb{S}\right)  }\leq\left\Vert v\right\Vert _{C_{-}^{1-3/p}\left(
\mathtt{A}_{\varepsilon},\mathbb{S}\right)  }\varepsilon^{1-3/p}\leq C\left(
p,\lambda\right)  \left\Vert w_{1}\right\Vert _{C^{0}\left(  \mathtt{A}%
_{\varepsilon},\mathbb{S}\right)  }\varepsilon^{1-\frac{3}{p}}.\label{v-C0}%
\end{equation}
Plug these into (\ref{c-alpha}) we obtain
\begin{equation}
C\left(  \alpha\right)  \left\Vert v\right\Vert _{C_{-}^{1,\alpha}\left(
\mathtt{A}_{\varepsilon},\mathbb{S}\right)  }\leq\varepsilon^{-\alpha
}\left\Vert w_{1}\right\Vert _{C^{\alpha}\left(  \mathtt{A}_{\varepsilon
},\mathbb{S}^{+}\right)  }+C\left(  p,\lambda\right)  \left\Vert
w_{1}\right\Vert _{C^{0}\left(  \mathtt{A}_{\varepsilon},\mathbb{S}\right)
}\varepsilon^{-\left(  \frac{3}{p}+\alpha\right)  }.\label{v-sch}%
\end{equation}
For the component $u$, its boundary value is nonzero so we can not directly
apply the above inequalities as $v$. First we derive the $C^{0}$ estimate of
$u$, using the assumption that $\partial^{+}\partial^{-}$ has \emph{trivial
kernel} on $\Sigma$. Consider the section $\overline{u}$ of the bundle
$\mathbb{S}^{+}\rightarrow\Sigma$ defined as
\[
\overline{u}\left(  z\right)  =\int_{0}^{\varepsilon}u\left(  x_{1},z\right)
dx_{1}.
\]
From equation $\left(  \ref{CR1}\right)  $ we have
\[
\partial^{-}\overline{u}\left(  z\right)  =-\int_{0}^{\varepsilon}%
h^{-1/2}\left(  z\right)  \partial_{x_{1}}v\left(  x_{1},z\right)
dx_{1}=h^{-1/2}\left(  z\right)  \left(  v\left(  0,z\right)  -v\left(
\varepsilon,z\right)  \right)  =0.
\]
But from the assumption that $\partial^{+}\partial^{-}$ has trivial kernel on
$\Sigma$ we get that
\[
\overline{u}\left(  z\right)  \equiv0.
\]
Let $\operatorname{Re}\overline{u}$ and $\operatorname{Im}\overline{u}$ be the
real and imaginary parts of the section $\overline{u}$ (notice that
$\mathbb{S}^{+}$ is a complex line bundle). Then for any fixed $z\in\Sigma$,
\[
\int_{0}^{\varepsilon}\operatorname{Re}u\left(  x_{1,}z\right)  dx_{1}%
=0=\int_{0}^{\varepsilon}\operatorname{Im}u\left(  x_{1,}z\right)  dx_{1}.
\]
By mean value theorem for the $\mathbb{R}$-valued function $\operatorname{Re}%
u\left(  x_{1,}z\right)  $, there exists $s\in\left[  0,\varepsilon\right]  $
depending on $z$, such $\operatorname{Re}u\left(  s,z\right)  =0$. Therefore
for $x_{1}\neq s$, by the definition of the Schauder norm we have
\begin{align*}
\left\vert \operatorname{Re}u\left(  x_{1},z\right)  \right\vert  &
=\frac{\left\vert \operatorname{Re}u\left(  x_{1},z\right)  -\operatorname{Re}%
u\left(  s,z\right)  \right\vert }{\left\vert x_{1}-s\right\vert ^{1-\frac
{3}{p}}}\cdot\left\vert x_{1}-s\right\vert ^{1-\frac{3}{p}}\\
& \leq\left\Vert \operatorname{Re}u\right\Vert _{C^{1-3/p}\left(
\mathtt{A}_{\varepsilon},\mathbb{S}\right)  }\varepsilon^{1-\frac{3}{p}}.
\end{align*}
Then we have
\begin{align*}
\left\Vert \operatorname{Re}u\right\Vert _{C^{0}\left(  \mathtt{A}%
_{\varepsilon},\mathbb{S}\right)  } &  \leq\left\Vert \operatorname{Re}%
u\right\Vert _{C^{1-3/p}\left(  \mathtt{A}_{\varepsilon},\mathbb{S}\right)
}\varepsilon^{1-\frac{3}{p}}\\
&  \leq\left\Vert V\right\Vert _{C_{-}^{1-3/p}\left(  \mathtt{A}_{\varepsilon
},\mathbb{S}\right)  }\varepsilon^{1-\frac{3}{p}}\\
&  \leq C\left(  p,\lambda\right)  \left\Vert w_{1}\right\Vert _{C^{0}\left(
\mathtt{A}_{\varepsilon},\mathbb{S}\right)  }\varepsilon^{1-\frac{3}{p}}%
\end{align*}
by $\left(  \ref{C-Sch-p}\right)  $. Similarly\
\[
\left\Vert \operatorname{Im}u\right\Vert _{C_{-}^{0}\left(  \mathtt{A}%
_{\varepsilon},\mathbb{S}^{+}\right)  }\leq C\left(  p,\lambda\right)
\left\Vert w_{1}\right\Vert _{C^{0}\left(  \mathtt{A}_{\varepsilon}%
,\mathbb{S}\right)  }\varepsilon^{1-\frac{3}{p}}.
\]
Combining these we get
\begin{equation}
\left\Vert u\right\Vert _{C_{-}^{0}\left(  \mathtt{A}_{\varepsilon}%
,\mathbb{S}^{+}\right)  }\leq C\left(  p,\lambda\right)  \left\Vert
w_{1}\right\Vert _{C^{0}\left(  \mathtt{A}_{\varepsilon},\mathbb{S}\right)
}\varepsilon^{1-\frac{3}{p}}.\label{u-C0}%
\end{equation}
Now we are ready to derive the Schauder estimate of $u$. Using the equation
\[
\left\{
\begin{array}
[c]{ccc}%
u_{x_{1}}+\partial^{+}v & = & w_{1}\\
v_{x_{1}}+\partial^{-}u & = & 0
\end{array}
\right.  \text{ \ with }v|_{\partial\mathtt{A}_{\varepsilon}}=0,
\]
the Schauder estimate of $\ u_{x_{1}}$ and $\partial^{-}u$ can be reduced to
that of $v$. For the full covariant derivative $\nabla_{\Sigma}u$ on $\Sigma$,
we observe that
\[
\left[  \nabla_{\Sigma}u\right]  _{\alpha;\left(  \mathtt{A}_{\varepsilon
},\mathbb{S}\right)  }\leq\left[  \nabla_{\Sigma}u\right]  _{\alpha;\left(
\mathtt{A}_{\varepsilon},\mathbb{S}\right)  }^{z}+\left[  \nabla_{\Sigma
}u\right]  _{\alpha;\left(  \mathtt{A}_{\varepsilon},\mathbb{S}\right)
}^{x_{1}}%
\]
where $\left[  \cdot\right]  _{\alpha;\left(  \mathtt{A}_{\varepsilon
},\mathbb{S}\right)  }^{z}$ and $\left[  \cdot\right]  _{\alpha;\left(
\mathtt{A}_{\varepsilon},\mathbb{S}\right)  }^{x_{1}}$ as defined in $\left(
\ref{Schauder-derivative}\right)  $ are the Schauder $\alpha$-components for
the $\Sigma$ and $\left[  0,\varepsilon\right]  $ directions respectively. For
$\left[  \nabla_{\Sigma}u\right]  _{\alpha;\left(  \mathtt{A}_{\varepsilon
},\mathbb{S}\right)  }^{z}$, on each slice $\Sigma_{s}:=\Sigma\times\left\{
s\right\}  $, we use the elliptic estimate on compact closed surface
$\Sigma_{s}$ to control $\nabla_{\Sigma}u$ by $\partial^{-}u$,%
\[
\left[  \nabla_{\Sigma}u\right]  _{\alpha;\left(  \Sigma_{s},\mathbb{S}%
\right)  }^{z}\leq\left\vert u\right\vert _{C^{1,\alpha}\left(  \Sigma
_{s},\mathbb{S}\right)  }\leq C\left(  \left\vert \partial^{-}u\right\vert
_{C^{\alpha}\left(  \Sigma_{s},\mathbb{S}\right)  }+\left\vert u\right\vert
_{C^{0}\left(  \Sigma_{s},\mathbb{S}\right)  }\right)  ,
\]
and then take the sup for $0\leq s\leq\varepsilon$ to get
\begin{align}
\left[  \nabla_{\Sigma}u\right]  _{\alpha;\left(  \mathtt{A}_{\varepsilon
},\mathbb{S}\right)  }^{z} &  \leq C\left(  \left\vert \partial^{-}%
u\right\vert _{C^{\alpha}\left(  \text{A}_{\varepsilon},\mathbb{S}\right)
}+\left\vert u\right\vert _{C^{0}\left(  \text{A}_{\varepsilon},\mathbb{S}%
\right)  }\right)  \nonumber\\
&  =C\left(  \left\vert v_{x_{1}}\right\vert _{C^{\alpha}\left(
\text{A}_{\varepsilon},\mathbb{S}\right)  }+\left\vert u\right\vert
_{C^{0}\left(  \text{A}_{\varepsilon},\mathbb{S}\right)  }\right)
.\label{gradient-u-Sigma-z}%
\end{align}
Since the estimate for $v_{x_{1}}$ is known in $\left(  \ref{v-sch}\right)  $,
the only term left to estimate is $\left[  \nabla_{\Sigma}u\right]
_{\alpha;\left(  \mathtt{A}_{\varepsilon},\mathbb{S}\right)  }^{x_{1}}$. For
this we reduce $u$ to the zero boundary value case by introducing
\[
\widetilde{u}=u-\rho\left(  \frac{x_{1}}{\varepsilon}\right)  u\left(
\varepsilon,z\right)  -\left(  1-\rho\left(  \frac{x_{1}}{\varepsilon}\right)
\right)  u\left(  0,z\right)  ,
\]
where $\rho:\left[  0,1\right]  \rightarrow\left[  0,1\right]  ,\rho\left(
0\right)  =0,\rho\left(  1\right)  =1$ is a smooth cut-off function such that
$\left\Vert \rho\right\Vert _{C^{2,\alpha}\left[  0,1\right]  }\leq C\,$. Then
$\widetilde{u}$ satisfies
\[
\widetilde{u}_{x_{1}x_{1}}-\partial^{+}\partial^{-}\widetilde{u}%
=\partial_{x_{1}}\widetilde{w_{1}}+g\ \ \ and\ \ \widetilde{u}|_{\partial
\mathtt{A}_{\varepsilon}}=0,
\]
where%
\[
\widetilde{w_{1}}=w_{1}+\rho\left(  \frac{x_{1}}{\varepsilon}\right)  \left[
\left(  u_{x_{1}}\left(  \varepsilon,z\right)  -u_{x_{1}}\left(  0,z\right)
\right)  +\left(  w_{1}\left(  0,z\right)  -w_{1}\left(  \varepsilon,z\right)
\right)  \right]
\]
and
\begin{align*}
g &  =\frac{1}{\varepsilon^{2}}\rho^{\prime\prime}\left(  \frac{x_{1}%
}{\varepsilon}\right)  \left[  u\left(  0,z\right)  -u\left(  \varepsilon
,z\right)  \right]  \\
&  +\frac{1}{\varepsilon}\rho^{\prime}\left(  \frac{x_{1}}{\varepsilon
}\right)  \left[  \left(  u_{x_{1}}\left(  0,z\right)  -u_{x_{1}}\left(
\varepsilon,z\right)  \right)  +\left(  w_{1}\left(  \varepsilon,z\right)
-w_{1}\left(  0,z\right)  \right)  \right]  .
\end{align*}
In the above derivation of $\widetilde{w_{1}}$ and $g$ we have used the
equation
\[
\partial^{+}\partial^{-}u=u_{x_{1}x_{1}}-\partial_{x_{1}}w_{1}%
\]
from $\left(  \ref{Dirichlet-uv}\right)  $. By \cite{GT} section 4.4 equation
$\left(  4.46\right)  $ we have on radius-$\varepsilon$ half ball
$B_{\varepsilon},$
\[
\left\Vert \widetilde{u}\right\Vert _{C_{-}^{0}\left(  B_{\varepsilon
},\mathbb{S}\right)  }+\varepsilon\left\Vert \widetilde{w_{1}}\right\Vert
_{C^{\alpha}\left(  B_{\varepsilon},\mathbb{S}^{+}\right)  }^{\prime
}+\varepsilon^{2}\left\vert g\right\vert _{C_{-}^{0}\left(  B_{\varepsilon
},\mathbb{S}\right)  }\geq C\left(  \alpha\right)  \left\Vert \widetilde
{u}\right\Vert _{C_{-}^{1,\alpha}\left(  B_{\varepsilon/2},\mathbb{S}\right)
}^{\prime}.
\]
By the definition of the norm $\left\Vert f\right\Vert _{C^{k,\alpha}\left(
\overline{\Omega}\right)  }^{\prime}$ in $\left(  \ref{Schauder-norm-Prime}%
\right)  $, this is equivalent to
\begin{align}
&  \left\Vert \widetilde{u}\right\Vert _{C_{-}^{0}\left(  B_{\varepsilon
},\mathbb{S}\right)  }+\varepsilon\left\Vert \widetilde{w_{1}}\right\Vert
_{C^{0}\left(  B_{\varepsilon},\mathbb{S}^{+}\right)  }+\varepsilon^{1+\alpha
}\left[  \widetilde{w_{1}}\right]  _{\alpha;\left(  B_{\varepsilon}%
,\mathbb{S}^{+}\right)  }+\varepsilon^{2}\left\vert g\right\vert _{C_{-}%
^{0}\left(  B_{\varepsilon},\mathbb{S}\right)  }\nonumber\\
&  \geq C\left(  \alpha\right)  \left[  \left\Vert \widetilde{u}\right\Vert
_{C_{-}^{0}\left(  B_{\varepsilon/2},\mathbb{S}\right)  }+\varepsilon
\left\Vert \nabla\widetilde{u}\right\Vert _{C_{-}^{0}\left(  B_{\varepsilon
/2},\mathbb{S}\right)  }+\varepsilon^{1+\alpha}\left[  \nabla\widetilde
{u}\right]  _{\alpha;\left(  B_{\varepsilon/2},\mathbb{S}\right)  }\right]
.\label{schauder_u_tilde}%
\end{align}
Therefore
\begin{align}
&  \left\Vert \widetilde{u}\right\Vert _{C_{-}^{0}\left(  B_{\varepsilon
},\mathbb{S}\right)  }+\varepsilon\left\Vert \widetilde{w_{1}}\right\Vert
_{C^{0}\left(  B_{\varepsilon},\mathbb{S}^{+}\right)  }+\varepsilon^{1+\alpha
}\left[  \nabla\widetilde{w_{1}}\right]  _{\alpha;\left(  B_{\varepsilon
},\mathbb{S}^{+}\right)  }+\varepsilon^{2}\left\vert g\right\vert _{C_{-}%
^{0}\left(  B_{\varepsilon},\mathbb{S}\right)  }\nonumber\\
&  \geq C\left(  \alpha\right)  \varepsilon^{1+\alpha}\left[  \nabla_{\Sigma
}\widetilde{u}\right]  _{\alpha;\left(  B_{\varepsilon/2},\mathbb{S}\right)
}.\label{schauder_u_tilde1}%
\end{align}
We are going to get rid of the tilde terms in the above inequality by the
following rule: The terms with tilde differ from the original terms by $w_{1}%
$, $u$ and their derivatives. If the derivative is with respect to $x_{1}$ or
$\partial^{-}$, then we reduce it to the previous estimates of $v$ using
$\left(  \ref{CR1}\right)  $; If it is with respect to $\nabla_{\Sigma}$, the
full derivative on $\Sigma$, then we use elliptic estimate on each slice
Riemann surface $\left\{  t\right\}  \times\Sigma$. Notice that all
$\varepsilon$ powers are coupled with the order of derivatives so all terms
are dimensionless, and after the enlargements the estimate of $u$ is of
correct $\varepsilon$ power. The following is the precise calculation:
From the definition of $g,\widetilde{w_{1}}$ and $\widetilde{u}$ we can easily
check that
\begin{align}
\left\Vert \varepsilon^{2}g\right\Vert _{C^{0}\left(  \text{$B$}_{\varepsilon
},\mathbb{S}\right)  } &  \leq C\left(  \left\Vert u\right\Vert _{C^{0}\left(
\text{$B$}_{\varepsilon},\mathbb{S}\right)  }+\varepsilon\left\Vert u_{x_{1}%
}\right\Vert _{C^{0}\left(  \text{$B$}_{\varepsilon},\mathbb{S}\right)
}+\varepsilon\left\Vert w_{1}\right\Vert _{C^{0}\left(  B_{\varepsilon
},\mathbb{S}\right)  }\right)  \nonumber\\
\left\vert \widetilde{w_{1}}\right\vert _{C^{0}\left(  \text{$B$}%
_{\varepsilon},\mathbb{S}\right)  } &  \leq3\left\Vert w_{1}\right\Vert
_{C^{0}\left(  \text{$B$}_{\varepsilon},\mathbb{S}\right)  }+\left\Vert
u_{x_{1}}\right\Vert _{C^{0}\left(  \text{$B$}_{\varepsilon},\mathbb{S}%
\right)  }\nonumber\\
\left[  \widetilde{w_{1}}\right]  _{\alpha;\left(  \text{$B$}_{\varepsilon
},\mathbb{S}\right)  }^{x_{1}} &  \leq C\left(  \left[  w_{1}\right]
_{\alpha;\left(  \text{$B$}_{\varepsilon},\mathbb{S}\right)  }^{x_{1}}%
+\frac{1}{\varepsilon^{\alpha}}\left\Vert u_{x_{1}}\right\Vert _{C^{0}\left(
\text{$B$}_{\varepsilon},\mathbb{S}\right)  }+\frac{1}{\varepsilon^{\alpha}%
}\left\Vert w_{1}\right\Vert _{C^{0}\left(  \text{$B$}_{\varepsilon
},\mathbb{S}\right)  }\right)  \nonumber\\
\left[  \widetilde{w_{1}}\right]  _{\alpha;\left(  \text{$B$}_{\varepsilon
},\mathbb{S}\right)  }^{z} &  \leq C\left(  \left[  w_{1}\right]
_{\alpha;\left(  \text{$B$}_{\varepsilon},\mathbb{S}\right)  }^{z}+\left[
u_{x_{1}}\right]  _{\alpha;\left(  \text{$B$}_{\varepsilon},\mathbb{S}\right)
}^{z}\right)  \nonumber\\
\left\Vert \widetilde{u}\right\Vert _{C^{0}\left(  \text{$B$}_{\varepsilon
},\mathbb{S}\right)  } &  \leq2\left\Vert u\right\Vert _{C^{0}\left(
\text{$B$}_{\varepsilon},\mathbb{S}\right)  }\nonumber\\
\left\Vert \nabla_{\Sigma}\widetilde{u}\right\Vert _{C^{0}\left(  \text{$B$%
}_{\varepsilon},\mathbb{S}\right)  } &  \leq2\left\Vert \nabla_{\Sigma
}u\right\Vert _{C^{0}\left(  \text{$B$}_{\varepsilon},\mathbb{S}\right)
},\text{ \ }\nonumber\\
\left\Vert \nabla_{x_{1}}\widetilde{u}\right\Vert _{C^{0}\left(  \text{$B$%
}_{\varepsilon},\mathbb{S}\right)  } &  \leq C\left(  \left\Vert \nabla
_{x_{1}}u\right\Vert _{C^{0}\left(  \text{$B$}_{\varepsilon},\mathbb{S}%
\right)  }+\frac{1}{\varepsilon}\left\Vert u\right\Vert _{C^{0}\left(
B_{\varepsilon},\mathbb{S}\right)  }\right)  \nonumber\\
\left[  \nabla_{\Sigma}u\right]  _{\alpha;\left(  \text{$B$}_{\varepsilon
/2},\mathbb{S}\right)  }^{x_{1}} &  \leq C\left(  \left[  \nabla_{\Sigma
}\widetilde{u}\right]  _{\alpha;\left(  \text{$B$}_{\varepsilon/2}%
,\mathbb{S}\right)  }^{x_{1}}+\frac{1}{\varepsilon^{\alpha}}\left\Vert
\nabla_{\Sigma}u\right\Vert _{C^{0}\left(  \text{$B$}_{\varepsilon
/2},\mathbb{S}\right)  }\right)  .\label{gradient-u-compare}%
\end{align}
hence putting these back to $\left(  \ref{schauder_u_tilde1}\right)  $ and
$\left(  \ref{schauder_u_tilde}\right)  $, and noticing that from $\left(
\ref{e-Schauder-v}\right)  $ we can control $\left\Vert u_{x_{1}}\right\Vert
_{C^{0}\left(  B_{\varepsilon},\mathbb{S}\right)  }$ and $\left\Vert
\nabla_{\Sigma}u\right\Vert _{C^{0}\left(  B_{\varepsilon},\mathbb{S}\right)
}$ that appeared in the above inequalities by the following:
\begin{align}
\varepsilon\left\Vert u_{x_{1}}\right\Vert _{C^{0}\left(  B_{\varepsilon
},\mathbb{S}\right)  } &  \leq C\left(  \varepsilon\left\Vert \partial
^{+}v\right\Vert _{C^{0}\left(  B_{\varepsilon},\mathbb{S}\right)
}+\varepsilon\left\Vert w_{1}\right\Vert _{C^{0}\left(  B_{\varepsilon
},\mathbb{S}\right)  }\right)  \text{\ (by }\left(  \text{\ref{CR1}}\right)
\text{)}\nonumber\\
&  \leq C\left[  \left(  \left\Vert v\right\Vert _{C^{0}\left(  \text{A}%
_{\varepsilon},\mathbb{S}\right)  }+\varepsilon\left\Vert w_{1}\right\Vert
_{C^{\alpha}\left(  \text{A}_{\varepsilon},\mathbb{S}\right)  }\right)
+\varepsilon\left\Vert w_{1}\right\Vert _{C^{0}\left(  B_{\varepsilon
},\mathbb{S}\right)  }\right]  \text{ }\nonumber\\
&  \text{(by }\left(  \ref{e-Schauder-v}\right)  \text{)}\nonumber\\
&  \leq C\left(  \left\Vert v\right\Vert _{C^{0}\left(  \text{A}_{\varepsilon
},\mathbb{S}\right)  }+\varepsilon\left\Vert w_{1}\right\Vert _{C^{\alpha
}\left(  \text{A}_{\varepsilon},\mathbb{S}\right)  }\right)  ,\label{e-u-x1}%
\end{align}
and%
\begin{align}
\varepsilon\left\Vert \nabla_{\Sigma}u\right\Vert _{C^{0}\left(  \Sigma
_{s},\mathbb{S}\right)  } &  \leq C\varepsilon\left\Vert \nabla_{\Sigma
}u\right\Vert _{C^{\alpha}\left(  \Sigma_{s},\mathbb{S}\right)  }\nonumber\\
&  \leq C\varepsilon\left(  \left\Vert \partial^{-}u\right\Vert _{C^{\alpha
}\left(  \Sigma_{s},\mathbb{S}\right)  }+\left\Vert u\right\Vert
_{C^{0}\left(  \Sigma_{s},\mathbb{S}\right)  }\right)  \text{ (by ellipticity
on }\Sigma_{s}\text{)}\nonumber\\
&  \overset{\text{(by }\left(  \text{\ref{CR1}}\right)  \text{)}}{=}C\left(
\varepsilon\left\Vert v_{x_{1}}\right\Vert _{C^{\alpha}\left(  \Sigma
_{s},\mathbb{S}\right)  }+\varepsilon\left\Vert u\right\Vert _{C^{0}\left(
\Sigma_{s},\mathbb{S}\right)  }\right)  \text{ \ }\nonumber\\
&  \overset{\text{(by }\left(  \ref{v-ball}\right)  \text{) }}{\leq}C\left(
\varepsilon^{-\alpha}\left\Vert v\right\Vert _{C^{0}\left(  \text{A}%
_{\varepsilon},\mathbb{S}\right)  }+\varepsilon^{1-\alpha}\left\Vert
w_{1}\right\Vert _{C^{\alpha}\left(  \text{A}_{\varepsilon},\mathbb{S}\right)
}+\varepsilon\left\Vert u\right\Vert _{C^{0}\left(  \Sigma_{s},\mathbb{S}%
\right)  }\right)  \label{e-u-z}%
\end{align}
then from $\left(  \ref{schauder_u_tilde1}\right)  ,\left(
\ref{gradient-u-compare}\right)  $ we get the estimate for $\left[
\nabla_{\Sigma}u\right]  _{\alpha;\left(  B_{\varepsilon/2},\mathbb{S}\right)
}^{x_{1}}$:%
\begin{align*}
&  \varepsilon^{1-\alpha}\left\Vert w_{1}\right\Vert _{C^{\alpha}\left(
\text{A}_{\varepsilon},\mathbb{S}^{+}\right)  }+\varepsilon\left\Vert
u\right\Vert _{C_{-}^{0}\left(  \text{A}_{\varepsilon},\mathbb{S}\right)
}+\varepsilon^{-\alpha}\left\Vert v\right\Vert _{C_{-}^{0}\left(
\text{A}_{\varepsilon},\mathbb{S}\right)  }\\
&  \geq CC\left(  \alpha\right)  \varepsilon^{1+\alpha}\left[  \nabla_{\Sigma
}u\right]  _{\alpha;\left(  B_{\varepsilon/2},\mathbb{S}\right)  }^{x_{1}}.
\end{align*}
Combining $\left(  \ref{gradient-u-Sigma-z}\right)  $ about $\left[
\nabla_{\Sigma}u\right]  _{\alpha;\left(  B_{\varepsilon/2},\mathbb{S}\right)
}^{z}$, we have the full control of $\left[  \nabla_{\Sigma}u\right]
_{\alpha;\left(  B_{\varepsilon/2},\mathbb{S}\right)  }$:%
\begin{align*}
&  \varepsilon^{1-\alpha}\left\Vert w_{1}\right\Vert _{C^{\alpha}\left(
\text{A}_{\varepsilon},\mathbb{S}^{+}\right)  }+\varepsilon\left\Vert
u\right\Vert _{C_{-}^{0}\left(  \text{A}_{\varepsilon},\mathbb{S}\right)
}+\varepsilon^{-\alpha}\left\Vert v\right\Vert _{C_{-}^{0}\left(
\text{A}_{\varepsilon},\mathbb{S}\right)  }\\
&  \geq CC\left(  \alpha\right)  \varepsilon^{1+\alpha}\left[  \nabla_{\Sigma
}u\right]  _{\alpha;\left(  B_{\varepsilon/2},\mathbb{S}\right)  }.
\end{align*}
Combining with $\left(  \ref{e-u-x1}\right)  $ and $\left(  \ref{e-u-z}%
\right)  $ about $\left\Vert \nabla u\right\Vert _{C^{0}}$, we get%
\begin{align}
&  \varepsilon\left\Vert w_{1}\right\Vert _{C^{\alpha}\left(  \text{A}%
_{\varepsilon},\mathbb{S}^{+}\right)  }+\varepsilon^{1+\alpha}\left\Vert
u\right\Vert _{C_{-}^{0}\left(  \text{A}_{\varepsilon},\mathbb{S}\right)
}+\left\Vert v\right\Vert _{C_{-}^{0}\left(  \text{A}_{\varepsilon}%
,\mathbb{S}\right)  }\nonumber\\
&  \geq CC\left(  \alpha\right)  \varepsilon^{1+2\alpha}\left\Vert
u\right\Vert _{C_{-}^{1,\alpha}\left(  B_{\varepsilon/2},\mathbb{S}\right)
}.\label{(u-ball)}%
\end{align}
Covering $\mathtt{A}_{\varepsilon}$ by radius-$\varepsilon$ half balls
$B_{\varepsilon}$ and radius-$\varepsilon/2$ full balls centered on meridian
$\Sigma\times\left\{  \varepsilon/2\right\}  $ and then taking supremum on
$\mathtt{A}_{\varepsilon}$, we derive
\begin{align}
&  \varepsilon\left\Vert w_{1}\right\Vert _{C^{\alpha}\left(  \mathtt{A}%
_{\varepsilon},\mathbb{S}^{+}\right)  }+\varepsilon^{1+\alpha}\left\Vert
u\right\Vert _{C_{-}^{0}\left(  \mathtt{A}_{\varepsilon},\mathbb{S}\right)
}+\left\Vert v\right\Vert _{C_{-}^{0}\left(  \mathtt{A}_{\varepsilon
},\mathbb{S}\right)  }\nonumber\\
&  \geq CC\left(  \alpha\right)  \varepsilon^{1+2\alpha}\left\Vert
u\right\Vert _{C_{-}^{1,\alpha}\left(  \mathtt{A}_{\varepsilon},\mathbb{S}%
\right)  }.\label{schauder_u}%
\end{align}
Notice here $v$ is involved on the left side to control $u$, compared to
similar inequality $\left(  \ref{c-alpha}\right)  $ about $v$ which needs no
$u$. This is partly because in second order elliptic equations $\left(
\ref{Dirichlet-uv}\right)  $ of $u$ and $v$, the inhomogeneous term
$\partial_{x_{1}}w_{1}$ is \textquotedblleft more
discontinuous\textquotedblright\ than $-\partial^{-}w_{1}$, for the (weak)
derivative is taking in $x_{1}$ direction along which (the extended) $w_{1}$
is discontinuous.
Plugging these into $\left(  \ref{v-C0}\right)  $ again we obtain
\begin{align}
&  C\left(  \alpha\right)  \left\Vert u\right\Vert _{C_{-}^{1,\alpha}\left(
\mathtt{A}_{\varepsilon},\mathbb{S}\right)  }\nonumber\\
&  \leq\varepsilon^{-2\alpha}\left\Vert w_{1}\right\Vert _{C^{\alpha}\left(
\mathtt{A}_{\varepsilon},\mathbb{S}^{+}\right)  }+\varepsilon^{-\alpha
}\left\Vert u\right\Vert _{C_{-}^{0}\left(  \mathtt{A}_{\varepsilon
},\mathbb{S}\right)  }+\varepsilon^{-\left(  1+2\alpha\right)  }\left\Vert
v\right\Vert _{C_{-}^{0}\left(  \mathtt{A}_{\varepsilon},\mathbb{S}\right)
}\nonumber\\
&  \leq\varepsilon^{-2\alpha}\left\Vert w_{1}\right\Vert _{C^{\alpha}\left(
\mathtt{A}_{\varepsilon},\mathbb{S}^{+}\right)  }+C\left(  p,\lambda\right)
\left\Vert w_{1}\right\Vert _{C^{0}\left(  \mathtt{A}_{\varepsilon}%
,\mathbb{S}\right)  }\varepsilon^{-\left(  \frac{3}{p}+2\alpha\right)
}.\label{u-sch}%
\end{align}
The $C^{1,\alpha}$ estimates $\left(  \ref{u-sch}\right)  $ and $\left(
\ref{v-sch}\right)  $ also hold when $\mathtt{A}_{\varepsilon}$ is equipped
with a warped product metric $g_{\mathtt{A}_{\varepsilon},h}$ with $h\left(
z\right)  \in C^{\infty}\left(  \Sigma\right)  $ such that $\frac{1}{K}\leq
h\left(  z\right)  \leq K$, up to the constant factor $K>0$ on the right hand
sides of $\left(  \ref{u-sch}\right)  $ and $\left(  \ref{v-sch}\right)  $,
because they are derived by $C^{0}$ estimates $\left(  \ref{v-C0}\right)  $,
$\left(  \ref{u-C0}\right)  $,\ and local Schauder estimates $\left(
\ref{v-ball}\right)  $, $\left(  \ref{(u-ball)}\right)  $ for $u,v$ on half
balls $B_{\varepsilon}\left(  p\right)  $ and $B_{\varepsilon/2}\left(
p\right)  $ for $p\in\partial\mathtt{A}_{\varepsilon}=\left\{  \left(
0,z\right)  ,\left(  \varepsilon,z\right)  |z\in\Sigma\right\}  $, while when
$\varepsilon\rightarrow0$, the function $h\left(  z\right)  |_{B_{\varepsilon
}\left(  p\right)  }$ converges to constant function $h\left(  p\right)  $
uniformly in $C^{0}$ norm in $B_{\varepsilon}\left(  p\right)  $, and
correspondingly the second order equations of $u$, $v$ converge to those with
constant coefficients as above. This is similar to \textquotedblleft frozen
coefficient method\textquotedblright\ in Schauder theory of second order
elliptic equations.
Finally we estimate the right inverse bound of $\mathcal{D}$ in Schauder
setting. By Theorem \ref{onto} the operator $\mathcal{D}$ is surjective, so
for any $W=\left(  w_{1},w_{2}\right)  $ we can write it as the sum of
$\left(  w_{1},0\right)  $ and $\left(  0,w_{2}\right)  $, whose preimages
$V=\mathcal{D}^{-1}\left(  w_{1},0\right)  $ and $\mathcal{D}^{-1}\left(
0,w_{2}\right)  $ exist. Therefore we can reduce general cases to case (i) and
(ii). Combining the $C^{1,\alpha}$ estimates $\left(  \ref{V-W-case1}\right)
$ in case (i) and $\left(  \ref{u-sch}\right)  $ and $\left(  \ref{v-sch}%
\right)  $ for $u$ and $v$ in case (ii), we have
\[
C\left(  \alpha,p,\lambda\right)  \left\Vert V\right\Vert _{C_{-}^{1,\alpha
}\left(  \mathtt{A}_{\varepsilon},\mathbb{S}\right)  }\leq\varepsilon
^{-\left(  \frac{3}{p}+2\alpha\right)  }\left\Vert W\right\Vert _{C^{\alpha
}\left(  \mathtt{A}_{\varepsilon},\mathbb{S}\right)  },
\]
where $C\left(  \alpha,p,\lambda\right)  =C\left(  \alpha\right)  \left(
1+C\left(  p,\lambda\right)  \right)  ^{-1}$ is independent on $\varepsilon$.
Hence the result.
\end{proof}

\begin{remark}
\bigskip(About the asymmetry of case $\left(  i\right)  $ and case $\left(
ii\right)  $) In case $\left(  i\right)  $, since $v|_{\partial A_{\varepsilon
}}=0$, after odd reflections of $v$ and even reflections of $u$ and
$w_{2\text{,}}$ the $V=\left(  u,v\right)  $ is still a $C^{1,\alpha}$
function on the domain $A_{k\left(  \varepsilon\right)  \varepsilon}$ and
$w_{2}$ is of class $C^{\alpha}$ on $A_{k\left(  \varepsilon\right)
\varepsilon}$, so we can use the Schauder estimate of the domain $A_{k\left(
\varepsilon\right)  \varepsilon}$, which has uniform ellipticity; in case
$\left(  ii\right)  $, since $w_{1}|_{\partial A_{\varepsilon}}\neq0$ in
general, after odd reflections of $v$ and \thinspace$w_{1}$,$\,$ $w_{1}$ is no
longer continuous on $A_{k\left(  \varepsilon\right)  \varepsilon}\,$, not to
mention in $C^{1,\alpha}$. So we have to directly work on the thin domain
$A_{\varepsilon}$ (which lacks uniform ellipticity as $\varepsilon
\rightarrow0$) for the Schauder estimates, with $\varepsilon$-dependent
coefficients. The steps there are from $L^{p}$, $C^{1-\frac{3}{p}},C^{0}$ to
$C^{1,\alpha}$.
\end{remark}

We remark that when $\lambda_{\partial^{+}}>0$ but $\lambda_{\partial^{-}}=0$
(i.e. $\ker\partial^{-}\neq\left\{  0\right\}  $), there still exists a right
inverse $Q_{\varepsilon}$ $:C^{\alpha}\left(  \mathtt{A}_{\varepsilon
},\mathbb{S}\right)  \rightarrow C_{-}^{1,\alpha}\left(  \mathtt{A}%
_{\varepsilon},\mathbb{S}\right)  $ of $\mathcal{D}$ with the operator norm
bound $\left\Vert Q_{\varepsilon}\right\Vert \leq C\varepsilon^{-\left(
\frac{3}{p}+2\alpha\right)  }$. To prove this, one needs to consider the
restriction%
\[
\mathcal{D}:L_{-}^{1,2}\left(  \mathtt{A}_{\varepsilon},\mathbb{S}\right)
\cap\left(  \ker\mathcal{D}\right)  ^{\perp}\rightarrow L^{2}\left(
\mathtt{A}_{\varepsilon},\mathbb{S}\right)
\]
to construct the right inverse $Q_{\varepsilon}$, where \textquotedblleft%
$\perp$\textquotedblright\ is the $L^{2}$ orthogonal complement, and replace
the constants
\begin{align*}
\lambda_{\mathcal{D}_{-}}  &  =\inf_{0\neq V\in L_{-}^{1,2}\left(
\mathtt{A}_{\varepsilon},\mathbb{S}\right)  }\frac{\left\Vert \mathcal{D}%
V\right\Vert _{L^{2}\left(  \mathtt{A}_{\varepsilon},\mathbb{S}\right)  }^{2}%
}{\left\Vert V\right\Vert _{L^{2}\left(  \mathtt{A}_{\varepsilon}%
,\mathbb{S}\right)  }^{2}}\text{, }\\
\lambda_{\partial^{-}}  &  =\inf_{0\neq V\in L_{-}^{1,2}\left(  \Sigma
,\mathbb{S}^{+}\right)  }\frac{\left\Vert \partial^{-}V\right\Vert
_{L^{2}\left(  \Sigma,\mathbb{S}^{+}\right)  }^{2}}{\left\Vert V\right\Vert
_{L^{2}\left(  \Sigma,\mathbb{S}^{+}\right)  }^{2}}%
\end{align*}
in Theorem \ref{lambda} by the constants
\begin{align*}
\widetilde{\lambda}_{\mathcal{D}_{-}}  &  =\inf_{0\neq V\in L_{-}^{1,2}\left(
\mathtt{A}_{\varepsilon},\mathbb{S}\right)  \cap\left(  \ker\mathcal{D}%
\right)  ^{\perp}}\frac{\left\Vert \mathcal{D}V\right\Vert _{L^{2}\left(
\mathtt{A}_{\varepsilon},\mathbb{S}\right)  }^{2}}{\left\Vert V\right\Vert
_{L^{2}\left(  \mathtt{A}_{\varepsilon},\mathbb{S}\right)  }^{2}}\text{, }\\
\widetilde{\lambda}_{\partial^{-}}  &  =\inf_{0\neq V\in L_{-}^{1,2}\left(
\Sigma,\mathbb{S}^{+}\right)  \cap\left(  \ker\partial^{-}\right)  ^{\perp}%
}\frac{\left\Vert \partial^{-}V\right\Vert _{L^{2}\left(  \Sigma
,\mathbb{S}^{+}\right)  }^{2}}{\left\Vert V\right\Vert _{L^{2}\left(
\Sigma,\mathbb{S}^{+}\right)  }^{2}}%
\end{align*}
respectively to get the $L^{2}$ estimate of $V$, and in the $C^{0}$ estimate
of $u$ in case (ii) of the above theorem use the fact that
\[
u\in\left(  \ker\mathcal{D}\right)  ^{\perp}\Rightarrow\overline{u}\in\left(
\ker\partial^{-}\right)  ^{\perp}%
\]
by integrating in the $x_{1}$ direction. The details are left to readers.

From now on we fix $p>3$ sufficiently large and $0<\alpha<1$ sufficiently
small such that
\[
0<\frac{3}{p}+3\alpha<\frac{1}{2}%
\]
This will be necessary for the implicit function theorem in the next section.

\section{Proof of the Main Theorem\label{proof}}

Let $C_{0}\subset M$ be a \textbf{compact} coassociative submanifold. Suppose
that $n$ is a normal vector field on $C_{0}$ such that its corresponding
self-dual two form,
\[
\eta_{0}=\iota_{n}\Omega\in\wedge_{+}^{2}\left(  C_{0}\right)
\]
is harmonic with respect to the induced metric. So $\eta_{0}$ is actually a
symplectic form on the complement of the zero set $Z\left(  \eta_{0}\right)  $
of $\eta_{0}$ in $C_{0}$. Furthermore
\begin{equation}
\ J_{n}\left(  u\right)  :=\frac{n}{\left\vert n\right\vert }\times u
\label{Jn}%
\end{equation}
defines an almost complex structure $J_{n}$ on the $C_{0}\backslash Z\left(
\eta_{0}\right)  .$ Since deformations of coassociative submanifolds are
unobstructed, we may assume that there is a one parameter family of
coassociative submanifolds $\varphi:\left[  0,\varepsilon\right]  \times
C_{0}\longrightarrow M$ such that
\[
\left.  \frac{\partial\varphi}{\partial t}\right\vert _{t=0}=n\in\Gamma\left(
C_{0},N_{C_{0}/M}\right)  .
\]

For $\varepsilon$ small, let
\[
\mathtt{C:}=\left[  0,\varepsilon\right]  \times C_{0},\text{ }\mathcal{C}%
:=\varphi\left(  \mathtt{C}\right)  ,\text{ and }C_{t}:=\varphi\left(
\left\{  t\right\}  \times C_{0}\right)  .\text{ \ }%
\]
then $\mathcal{C}$\ is diffeomorphic to $\mathtt{C}$, and $\varphi\left(
t,\cdot\right)  $ is an embedding for $\forall t\in\left[  0,\varepsilon
\right]  $. We remind the readers about the typefaces of our notations: both
$\mathtt{C}$ and $\mathcal{C}$ are $5$-dimensional, while each $C_{t}$ is
$4$-dimensional. Let
\[
g_{\mathtt{C}}:=dt^{2}\oplus g|_{C_{0}}%
\]
be the product metric on $\mathtt{C}$ and $\exp^{\mathtt{C}}$ be the
exponential map associated to the metric $g_{\mathtt{C}}.$

In the remaining part of this article we assume that $\eta_{0}$ is
\emph{nowhere vanishing} on $C_{0}$, that implies that $\left(  C_{0},\eta
_{0}\right)  $\ is a symplectic four manifold, and all coassociative
submanifolds$\ C_{t}$'s are mutually disjoint. We are going to establish a
correspondence between the\ regular $J_{n}$-holomorphic curves $\Sigma$ in
$C_{0}$ and the instantons in $M$ with coassociative boundary conditions.

Given such a $\Sigma\subset C_{0}$, we denote
\[
\mathtt{A}_{\varepsilon}:=\left[  0,\varepsilon\right]  \times\Sigma\text{ and
}\mathtt{A}_{\varepsilon}^{\prime}:=\varphi\left(  \mathtt{A}_{\varepsilon
}\right)
\]
then $\mathtt{A}_{\varepsilon}^{\prime}$ is close to being associative in the
sense that $\left\vert \tau|_{\mathtt{A}_{\varepsilon}^{\prime}}\right\vert
\leq K\varepsilon$ for some constant $K$ depending on the geometry of the
family $\left\{  C_{t}\right\}  $ and $M$ for small $\varepsilon$. This is due
to the smooth dependence of $\tau\left(  \varphi\left(  t,z\right)  \right)  $
on $\left(  t,z\right)  \in\left[  0,\varepsilon\right]  \times\Sigma$, and
that%
\[
\tau|_{T\mathtt{A}_{\varepsilon}^{\prime}|_{\varphi\left(  0,\Sigma\right)  }%
}=\tau|_{T\mathtt{A}_{\varepsilon}^{\prime}|_{\Sigma}}=0,
\]
since $\Sigma$ is a $J_{n}$-holomorphic curve, the tangent spaces $T\Sigma$ is
closed under $J_{n}=\frac{n}{\left\vert n\right\vert }\times$, and
$T\mathtt{A}_{\varepsilon}^{\prime}|_{\Sigma}=T\Sigma\oplus$span$\left\{
n\right\}  $.

We want to perturb $A_{\varepsilon}^{\prime}$ to become an honest associative
submanifold in $M$. In order to apply the implicit function theorem to obtain
the desired perturbation for $\mathtt{A}_{\varepsilon}^{\prime}$, we need the
estimates for the linearized problem to behave well as $\varepsilon$
approaches zero. Notice that for small $\varepsilon$, the induced metric on
$\mathtt{A}_{\varepsilon}^{\prime}$ is close to be a warped product metric
\[
g_{\mathtt{A}_{\varepsilon},h}:=h\left(  z\right)  dx_{1}^{2}+g_{\Sigma}%
\]
on $\mathtt{A}_{\varepsilon}$, where $h\left(  z\right)  =\left\vert n\left(
z\right)  \right\vert ^{2}$ is the length squared of the normal vector $n$.
Namely
\[
\left(  1-K\varepsilon\right)  g_{\mathtt{A}_{\varepsilon},h}\leq\varphi
^{\ast}g_{M}\leq\left(  1+K\varepsilon\right)  g_{\mathtt{A}_{\varepsilon},h}%
\]
for some uniform constant $K$. This is because for any vectors $X,Y\in
T_{\left(  t,z\right)  }A_{\varepsilon_{0}}$ and $t\in\left[  0,\varepsilon
_{0}\right]  $, the function $G:\left[  0,\varepsilon_{0}\right]  \times
TA_{\varepsilon_{0}}\times TA_{\varepsilon_{0}}\rightarrow\mathbb{R}$,%
\[
G\left(  t,X,Y\right)  :=g_{M}\left(  d\varphi\left(  t,z\right)
X,d\varphi\left(  t,z\right)  Y\right)
\]
is smooth with respect to $\left(  t,X,Y\right)  $, bilinear in $\left(
X,Y\right)  $, and
\begin{align*}
&  G\left(  0,TA_{\varepsilon_{0}}|_{\left\{  0\right\}  \times\Sigma
},TA_{\varepsilon_{0}}|_{\left\{  0\right\}  \times\Sigma}\right) \\
&  =g_{M}\left(  d\varphi|_{\left\{  0\right\}  \times\Sigma}\left(
TA_{\varepsilon_{0}}|_{\left\{  0\right\}  \times\Sigma}\right)
,d\varphi|_{\left\{  0\right\}  \times\Sigma}\left(  TA_{\varepsilon_{0}%
}|_{\left\{  0\right\}  \times\Sigma}\right)  \right) \\
&  =g_{\mathtt{A}_{\varepsilon_{0}},h}|_{\left\{  0\right\}  \times\Sigma},
\end{align*}
where in the last identity we have used that $d\varphi|_{\left\{  0\right\}
\times\Sigma}=id:T\Sigma\rightarrow T\Sigma$ and $d\varphi|_{\left\{
0\right\}  \times\Sigma}:\frac{\partial}{\partial x_{1}}\rightarrow n$.

For the warped product metric on $\mathtt{A}_{\varepsilon}$ and the
corresponding operator $\mathcal{D}$, the estimates for its inverse have been
established in Theorem \ref{e-inverse-bound} in the previous section. The
above discussion indicates that the linearization of the instanton equation on
A$_{\varepsilon}^{\prime}$ may be compared with $\mathcal{D}$. We will first
show they agree on $\Sigma$ in the next subsection. Then the comparison on
A$_{\varepsilon}^{\prime}$ is a small perturbation from their agreement on
$\Sigma$.

\subsection{Geometry of $\Sigma\subset\mathcal{C}\subset M$}

We study the geometry of the $J_{n}$-holomorphic curves $\Sigma$ in a $G_{2}%
$-manifold $M$. Let $\mathcal{C}=\cup_{0\leq t\leq\varepsilon}C_{t}$ be the
family of coassociative manifolds in the previous subsection and
$\Sigma\subset C=C_{0}$. We remind the readers that the $\mathcal{C}$ is a
$5$-dimensional submanifold so its normal bundle in $M$ has real rank $2$,
while $C$ is a $4$-dimensional submanifold. We recall the following results in
Lemma 3.2 of \cite{GayetWitt}.

\begin{proposition}
\begin{enumerate}
\item \bigskip$N_{\Sigma/C}$ and $N_{\mathcal{C}/M}|_{\Sigma}$ are complex
line bundles with the almost complex structure $J_{n}=$ $\frac{n}{\left\vert
n\right\vert }\times$ on fibers.

\item $\overline{N_{\mathcal{C}/M}|_{\Sigma}}\simeq\wedge_{\mathbb{C}}%
^{0,1}\left(  N_{\Sigma/C}\right)  $ as complex line bundles, where
$\overline{L}$ is the conjugate of a complex line bundle $L$, and
$\wedge_{\mathbb{C}}^{0,1}\left(  N_{\Sigma/C}\right)  =N_{\Sigma/C}%
\otimes_{\mathbb{C}}\wedge_{\mathbb{C}}^{0,1}\left(  T^{\ast}\Sigma\right)  $.
\end{enumerate}
\end{proposition}

Note that the notations $\partial Y,X,\nu_{X}$ and $\mu_{X}$ in
\cite{GayetWitt} correspond to our $\Sigma,C,N_{\Sigma/C}$ and $N_{\mathcal{C}%
/M}|_{\Sigma}$ respectively.

In the following we will use the notation $N_{\mathcal{C}/M}|_{\Sigma}$ rather
than $\overline{N_{\mathcal{C}/M}|_{\Sigma}}$, but we emphasize that the
isomorphism $N_{\mathcal{C}/M}|_{\Sigma}\simeq\wedge_{\mathbb{C}}^{0,1}\left(
N_{\Sigma/C}\right)  $\ \emph{is complex conjugate linear }and should be
regarded as an isomorphism between \emph{real vector bundles}. Note that we do
NOT complexify $N_{\Sigma/C}$, since it is already a complex line bundle, and
more importantly, we want to use its complex structure $J_{n}=\frac
{n}{\left\vert n\right\vert }\times$ to interplay with the $G_{2}$ geometry.

In the following Proposition we will show that $N_{\Sigma/C}$ and
$N_{\mathcal{C}/M}|_{\Sigma}$ are Hermitian line bundles, and $N:=N_{\Sigma
/C}\oplus N_{\mathcal{C}/M}|_{\Sigma}\rightarrow\Sigma$ is a \emph{Dirac
bundle }(in the sense of Definition 5.2 in \cite{LM}), i.e. $N$ is a left
Clifford module over $\Sigma$ with respect to the $G_{2}$ multiplication
$\times$, together with a Riemannian metric $\left\langle ,\right\rangle $ and
connection $\nabla^{N}$ on $N$ satisfying: (a) at each $p\in\Sigma$, for any
$\sigma_{1},\sigma_{2}\in N_{p}$ and any unit vector $e\in T_{p}\Sigma$,
$\left\langle e\sigma_{1},e\sigma_{2}\right\rangle =\left\langle e\sigma
_{1},e\sigma_{2}\right\rangle $; (b) for any section $\varphi$ of $T\Sigma$
and section $\sigma$ of $N$, $\nabla^{N}\left(  \varphi\times\sigma\right)
=\left(  \nabla^{T\Sigma}\varphi\right)  \times\sigma+\varphi\times\nabla
^{N}\sigma$.

\begin{proposition}
\label{Dirac-bdl}Let $\nabla$ be the Levi-Civita connection of $\left(
M,g\right)  $. For any subbundle $L\rightarrow\Sigma_{0}$ of $TM|_{\Sigma_{0}%
}$, let the \emph{induced connection} $\nabla^{L}$ of $L$ be
\[
\nabla^{L}:=\pi^{L}\circ\nabla,
\]
where $\pi^{L}$ is the orthogonal projection of $TM|_{\Sigma}$ to subbundle
$L$ according to the metric $g$. Then

\begin{enumerate}
\item For $L=N_{\Sigma/C}$, $N_{\mathcal{C}/M}|_{\Sigma}$ or $T\Sigma$, the
induced connection $\nabla^{L}$ is Hermitian, namely $\nabla^{L}J_{n}=0$.

\item Let $N=N_{\Sigma/C}\oplus N_{\mathcal{C}/M}|_{\Sigma}$. Then with
respect to the $G_{2}$ multiplication $\times$ as the Clifford multiplication
and the induced connection $\nabla^{N}$, $N$ is a Dirac bundle.
\end{enumerate}
\end{proposition}

\begin{proof}
For each $p\in\Sigma_{0},$ the subspace $T_{p}\Sigma_{0}\oplus\text{span}
\{n(p)\}\subset T_{p}M$ is associative since $\Sigma_{0}$ is
$J_{n}$-holomorphic. By the same Cayley-Dickson construction in Theorem
\ref{McLean-thm}, we may choose orthonormal frame $\left\{  W_{\alpha
}\right\}  _{\alpha=1,\cdots7}$ in a neighborhood $B_{\varepsilon}\left(
p\right)  $ $\subset\Sigma$ of $p$ satisfying standard $\cdot,\times$ relation
as the basis of $\operatorname{Im}\mathbb{O}$, such that
\begin{align*}
T\Sigma_{0} &  =\text{span}\left\{  W_{2},W_{3}\right\}  ,\text{ }W_{1}%
:=W_{2}\times W_{3},\text{ }\\
W_{4} &  \in N_{\Sigma\,_{0}/C_{0}}\text{, }W_{i+4}:=W_{i}\times W_{4}\text{
for }1\leq i\leq3,\\
N_{\Sigma\,_{0}/C_{0}}\text{ } &  =\text{span}\left\{  W_{4},W_{5}\right\}
\text{, \ and }N_{\mathcal{C}/M}|_{\Sigma_{0}}=\text{span}\left\{  W_{6}%
,W_{7}\right\}  .
\end{align*}
For $i=2,3$, we have%
\[
\nabla_{Wi}^{L}J_{n}=\nabla_{Wi}^{L}\left(  W_{1}\times|_{L}\right)  =\pi
^{L}\left(  \nabla_{Wi}W_{1}\right)  \times|_{L}=\left\langle W_{1}%
,\nabla_{Wi}W_{1}\right\rangle W_{1}\times|_{L}=0,
\]
where the third identity is because for $L=T\Sigma=$span$\left\{  W_{2}%
,W_{3}\right\}  $, $L=N_{\Sigma_{0}/C}=$span$\left\{  W_{4},W_{5}\right\}  $
or $L=$ $N_{\mathcal{C}/M}|_{\Sigma_{0}}=$span$\left\{  W_{6},W_{7}\right\}
$, under cross product $\times$, only the $W_{1}$ component of $\nabla
_{Wi}W_{1}$ can preserve $L$, i.e.
\[
W_{1}\times L\subset L,\text{ and }\left(  W_{k}\times L\right)  \cap
L=\left\{  0\right\}  \text{ for }2\leq k\leq7
\]
from octonion multiplication relation, and the last identity is because
$\left\langle W_{1},\nabla_{Wi}W_{1}\right\rangle =\frac{1}{2}\nabla
_{Wi}\left\langle W_{1},W_{1}\right\rangle =0$. Thus $\nabla_{v}^{L}J=0$ for
any $v\in T\Sigma_{0}$.
Using the standard $G_{2}$ multiplication relation of the \ above
\textquotedblleft good\textquotedblright\ frame $\left\{  W_{\alpha}\right\}
_{\alpha=1,\cdots7}$, it is easy to check that $N:=N_{\Sigma/C}\oplus
N_{\mathcal{C}/M}|_{\Sigma}\rightarrow\Sigma$ is a Clifford module.
To show
$N$ is a Dirac bundle, we first note that for any unit vector $e\in
T_{p}\Sigma$, $e\times$ is an isometry on $N$\ by the property of $\times$;
second, for any section $\varphi$ of $T\Sigma$ and section $\sigma$ of
$N_{\Sigma/C}\oplus N_{\mathcal{C}/M}|_{\Sigma}$, with respect to the induced
connection $\nabla^{N}$ on $N$ we have%
\[
\nabla^{N}\left(  \varphi\times\sigma\right)  =\pi^{N}\circ\left(
\nabla\varphi\times\sigma+\varphi\times\nabla\sigma\right)  =\left(
\nabla^{T\Sigma}\varphi\right)  \times\sigma+\varphi\times\nabla^{N}\sigma,
\]
where in the second identity, the first term is because for distinct
$i,j\in\left\{  4,5,6,7\right\}  $, $W_{i}\times W_{j}\in$span$\left\{
W_{1},W_{2},W_{3}\right\}  $, and the second term is because span$\left\{
W_{1},W_{2},W_{3}\right\}  $ is closed under $\times$.
\end{proof}

It is well-known that on a compact K\"{a}hler manifold with a Hermitian line
bundle $L$, one can define a Dirac operator on the Dolbeault complex
$\Omega_{\mathbb{C}}^{0,\ast}\left(  L\right)  $ (c.f. Proposition 3.67 of
\cite{BGV} or Proposition 1.4.25 of \cite{Nicolaescu}). Taking the K\"{a}hler
manifold to be $\Sigma$ and the Hermitian line bundle to be $N_{\Sigma/C}$, we
have the Dolbeault Dirac operator on $N_{\Sigma/C}\oplus\wedge_{\mathbb{C}%
}^{0,1}\left(  N_{\Sigma/C}\right)  $. On the other hand, by the above Lemma,
$N_{\Sigma/C}\oplus N_{\mathcal{C}/M}|_{\Sigma}$ is a Dirac bundle and has a
canonically associated Dirac operator. The following Proposition compares the
two Dirac operators.

\begin{proposition}
\label{dirac-agree}\bigskip The Dolbeault Dirac operator on $N_{\Sigma
/C}\oplus\wedge_{\mathbb{C}}^{0,1}\left(  N_{\Sigma/C}\right)  $ agrees with
the Dirac operator on $N_{\Sigma/C}\oplus N_{\mathcal{C}/M}|_{\Sigma}$, whose
Clifford multiplication is the $G_{2}$ multiplication $\times$ and connection
is the induced connection from $M$.
\end{proposition}

\begin{proof}
Given any $p\in\Sigma$, we may further assume the \textquotedblleft
good\textquotedblright\ frame $\left\{  W_{\alpha}\right\}  _{\alpha
=1,2,\cdots7}$ on $B_{\varepsilon}\left(  p\right)  $ in Proposition
\ref{Dirac-bdl} satisfies
\begin{equation}
\nabla_{W_{i}}^{T\Sigma_{0}}W_{j}\left(  p\right)  =\text{ }\nabla_{W_{i}%
}^{N_{\Sigma\,_{0}/C}}W_{k}\left(  p\right)  =\nabla_{W_{i}}^{N_{\mathcal{C}%
/M}|_{\Sigma_{0}}}W_{k+2}\left(  p\right)  =0\text{ } \label{base-on-sigma}%
\end{equation}
for $2\leq i,j\leq3$ and $4\leq k\leq5$, where the $\nabla^{T\Sigma_{0}}%
$,$\nabla^{N_{\Sigma\,_{0}/C_{0}}}$ and $\nabla^{N_{\mathcal{C}/M}%
|_{\Sigma_{0}}}$ are orthogonal projections of the Levi-Civita connection
$\nabla$ on $M$ to $T\Sigma_{0},N_{\Sigma\,_{0}/C_{0}}$ and $N_{\mathcal{C}%
/M}|_{\Sigma_{0}}$ respectively (with respect to the metric $g$). To achieve
$\left(  \ref{base-on-sigma}\right)  $, we can first require $\nabla_{W_{i}%
}^{\Sigma_{0}}W_{j}\left(  p\right)  =\nabla_{W_{i}}^{N_{\Sigma\,_{0}/C}}%
W_{4}\left(  p\right)  =0$ for $2\leq i,j\leq3$, and then use $W_{i+4}%
:=W_{i}\times W_{4}$ for $1\leq i\leq3$ to show all other covariant
derivatives in $\left(  \ref{base-on-sigma}\right)  $ vanish at $p$.
Let $Z=\frac{1}{\sqrt{2}}\left(  W_{2}-iW_{3}\right)  \in T\Sigma_{\mathbb{C}%
}^{1,0}$, $\overline{Z}=\frac{1}{\sqrt{2}}\left(  W_{2}+iW_{3}\right)  \in
T\Sigma_{\mathbb{C}}^{0,1}$and $\overline{Z}^{\ast}=\frac{W_{2}^{\ast}%
-iW_{3}^{\ast}}{\sqrt{2}}\in T^{\ast}\Sigma_{\mathbb{C}}^{0,1}$, then
$\left\langle \overline{Z}^{\ast},\overline{Z}\right\rangle =1$. We recall
that the real bundle isomorphism $f:\wedge_{\mathbb{C}}^{0,1}\left(
N_{\Sigma/C}\right)  \simeq N_{\mathcal{C}/M}|_{\Sigma}$ from \cite{GayetWitt}
(up to factor $-\frac{1}{\sqrt{2}}$) in local frame was given by%
\begin{equation}%
\begin{tabular}
[c]{llll}%
$f:$ & $\wedge_{\mathbb{C}}^{0,1}\left(  N_{\Sigma/C}\right)  =W_{4}%
\otimes_{\mathbb{C}}T^{\ast}\Sigma^{0,1}$ & $\longrightarrow$ & $W_{4}\times
T\Sigma^{1,0}=N_{\mathcal{C}/M}|_{\Sigma},$\\
& $\ \ \ \ \ \ \ \ \ \ W_{4}\otimes_{\mathbb{C}}\overline{Z}^{\ast}$ &
$\mapsto$ & $-\frac{1}{\sqrt{2}}W_{4}\times Z,$%
\end{tabular}
\label{iso-bdl}%
\end{equation}
which is complex conjugate linear and independent on the choice of $W_{4}\in
N_{\Sigma/C}$. So we have the map
\[
\Phi:=id\oplus f:N_{\Sigma/C}\oplus\wedge_{\mathbb{C}}^{0,1}\left(
N_{\Sigma/C}\right)  \rightarrow N_{\Sigma/C}\oplus N_{\mathcal{C}/M}%
|_{\Sigma}%
\]
\emph{ }whose action on the basis of (rank two \emph{real} \emph{vector
bundles}) $N_{\Sigma/C}$ and $\wedge_{\mathbb{C}}^{0,1}\left(  N_{\Sigma
/C}\right)  $ is the following:
\begin{align}
\Phi &  :N_{\Sigma/C}\overset{id}{\rightarrow}N_{\Sigma/C},\text{ \ }\left\{
W_{4},W_{5}\right\}  \rightarrow\left\{  W_{4},W_{5}\right\}  ,\nonumber\\
\Phi &  :\wedge_{\mathbb{C}}^{0,1}\left(  N_{\Sigma/C}\right)  \overset
{f}{\rightarrow}N_{\mathcal{C}/M}|_{\Sigma},\text{ \ }\left\{  W_{4}%
\otimes\overline{Z}^{\ast},W_{5}\otimes\overline{Z}^{\ast}\right\}
\rightarrow\left\{  W_{6},W_{7}\right\}  . \label{Dirac-basis-corresp}%
\end{align}
This is because under $f$,
\begin{align*}
W_{4}\otimes\left(  \frac{W_{2}^{\ast}-iW_{3}^{\ast}}{\sqrt{2}}\right)   &
\rightarrow-\frac{1}{\sqrt{2}}\frac{W_{4}\times W_{2}-W_{4}\times iW_{3}%
}{\sqrt{2}}\\
&  =-\frac{-W_{6}-\left(  W_{1}\times W_{4}\right)  \times W_{3}}{2}=W_{6},\\
W_{5}\otimes\left(  \frac{W_{2}^{\ast}-iW_{3}^{\ast}}{\sqrt{2}}\right)   &
=J_{n}W_{4}\otimes\left(  \frac{W_{2}^{\ast}-iW_{3}^{\ast}}{\sqrt{2}}\right)
\\
&  \rightarrow-J_{n}\left(  W_{6}\right)  =-W_{1}\times W_{6}=W_{7},
\end{align*}
where the second row is from the tensor property of $f$ and the last row is
because $f:\wedge_{\mathbb{C}}^{0,1}\left(  N_{\Sigma/C}\right)  \rightarrow
N_{\mathcal{C}/M}|_{\Sigma}$ is complex conjugate linear. So for the sections
\begin{align*}
U  &  =V^{4}W_{4}+V^{5}W_{5}+V^{6}W_{4}\otimes\overline{Z}^{\ast}+V^{7}%
W_{5}\otimes\overline{Z}^{\ast}\in\Gamma\left(  N_{\Sigma/M}\oplus
\wedge_{\mathbb{C}}^{0,1}\left(  N_{\Sigma/C}\right)  \right)  \text{,}\\
V  &  =V^{4}W_{4}+V^{5}W_{5}+V^{6}W_{6}+V^{7}W_{7}\in\Gamma\left(
N_{\Sigma/C}\oplus N_{\mathcal{C}/M}|_{\Sigma}\right)  ,
\end{align*}
we have
\[
\Phi\left(  U\right)  =V.
\]
We compute the Dolbeault Dirac operator using the above frame. The Clifford
multiplication $c$ on the Dolbeault complex $N_{\Sigma/C}\oplus\wedge
_{\mathbb{C}}^{0,1}\left(  N_{\Sigma/C}\right)  $ at $p$ is (see Section 1.4.3
in \cite{Nicolaescu})
\begin{align*}
c\left(  \frac{W_{2}+iW_{3}}{\sqrt{2}}\right)   &  :V^{4}W_{4}+V^{5}%
W_{5}\overset{\sqrt{2}e\left(  \overline{Z}^{\ast}\right)  }{\rightarrow}%
\sqrt{2}\left(  V^{4}W_{4}\otimes\overline{Z}^{\ast}+V^{5}W_{5}\otimes
\overline{Z}^{\ast}\right)  ,\\
c\left(  \frac{W_{2}-iW_{3}}{\sqrt{2}}\right)   &  :V^{6}W_{4}\otimes
\overline{Z}^{\ast}+V^{7}W_{5}\otimes\overline{Z}^{\ast}\overset{-\sqrt
{2}i\left(  \overline{Z}\right)  }{\rightarrow}\sqrt{2}\left(  -V^{6}%
W_{4}-V^{7}W_{5}\right)  ,
\end{align*}
where $e\left(  \overline{Z}^{\ast}\right)  $ is the wedge by $\overline
{Z}^{\ast}$ and $i\left(  \overline{Z}\right)  $ is the contraction by
$\overline{Z}$. Using $\left(  \ref{base-on-sigma}\right)  $ at $p$ and
complex linearity, we have
\begin{align*}
&  \nabla_{\left(  \frac{W_{2}+iW_{3}}{\sqrt{2}}\right)  }\left(  V^{4}%
W_{4}+V^{5}W_{5}\right)  \left(  p\right)  \\
&  =\frac{1}{\sqrt{2}}\left(  V_{2}^{4}W_{4}+V_{3}^{4}J_{n}W_{4}+V_{2}%
^{5}W_{5}+V_{3}^{5}J_{n}W_{5}\right)  \\
&  =\frac{1}{\sqrt{2}}\left(  \left(  V_{2}^{4}-V_{3}^{5}\right)
W_{4}+\left(  V_{2}^{5}+V_{3}^{4}\right)  W_{5}\right)
\end{align*}
and
\begin{align*}
&  c\left(  \frac{W_{2}+iW_{3}}{\sqrt{2}}\right)  \circ\nabla_{\left(
\frac{W_{2}+iW_{3}}{\sqrt{2}}\right)  }\left(  V^{4}W_{4}+V^{5}W_{5}\right)
\left(  p\right)  \\
&  =\left(  V_{2}^{4}-V_{3}^{5}\right)  W_{4}\otimes\overline{Z}^{\ast
}+\left(  V_{2}^{5}+V_{3}^{4}\right)  W_{5}\otimes\overline{Z}^{\ast}.
\end{align*}
Similarly we have
\begin{align*}
&  \nabla_{\left(  \frac{W_{2}-iW_{3}}{\sqrt{2}}\right)  }\left(  V^{6}%
W_{4}\otimes\overline{Z}^{\ast}+V^{7}W_{5}\otimes\overline{Z}^{\ast}\right)
\left(  p\right)  \\
&  =\frac{1}{\sqrt{2}}\left[  \left(  V_{2}^{6}+V_{3}^{7}\right)  W_{4}%
\otimes\overline{Z}^{\ast}+\left(  -V_{3}^{6}+V_{2}^{7}\right)  W_{5}%
\otimes\overline{Z}^{\ast}\right]
\end{align*}
and
\begin{align*}
&  c\left(  \frac{W_{2}-iW_{3}}{\sqrt{2}}\right)  \circ\nabla_{\left(
\frac{W_{2}-iW_{3}}{\sqrt{2}}\right)  }\left(  V^{6}W_{4}\otimes\overline
{Z}^{\ast}+V^{7}W_{5}\otimes\overline{Z}^{\ast}\right)  \left(  p\right)  \\
&  =-\left(  V_{2}^{6}+V_{3}^{7}\right)  W_{4}-\left(  -V_{3}^{6}+V_{2}%
^{7}\right)  W_{5}.
\end{align*}
The Dolbeault Dirac operator $\overline{\partial}$ on $N_{\Sigma/M}%
\oplus\wedge_{\mathbb{C}}^{0,1}\left(  N_{\Sigma/M}\right)  $ is defined by
\[
\overline{\partial}=c\left(  \frac{W_{2}+iW_{3}}{\sqrt{2}}\right)  \circ
\nabla_{\left(  \frac{W_{2}+iW_{3}}{\sqrt{2}}\right)  }+c\left(  \frac
{W_{2}-iW_{3}}{\sqrt{2}}\right)  \circ\nabla_{\left(  \frac{W_{2}-iW_{3}%
}{\sqrt{2}}\right)  }\text{,}%
\]
so at $p$ we have
\begin{align*}
\overline{\partial}U\left(  p\right)   &  =-\left(  V_{2}^{6}+V_{3}%
^{7}\right)  W_{4}+\left(  V_{3}^{6}-V_{2}^{7}\right)  W_{5}\\
&  +\left(  V_{2}^{4}-V_{3}^{5}\right)  W_{4}\otimes Z^{\ast}+\left(
V_{2}^{5}+V_{3}^{4}\right)  W_{5}\otimes Z^{\ast},
\end{align*}
and
\begin{align*}
\Phi\left(  \overline{\partial}U\left(  p\right)  \right)   &  =-\left(
V_{2}^{6}+V_{3}^{7}\right)  W_{4}+\left(  V_{3}^{6}-V_{2}^{7}\right)  W_{5}\\
&  +\left(  V_{2}^{4}-V_{3}^{5}\right)  W_{6}+\left(  V_{2}^{5}+V_{3}%
^{4}\right)  W_{7}.
\end{align*}
On the other hand, the twisted Dirac operator on $N:=N_{\Sigma/C}\oplus
N_{\mathcal{C}/M}|_{\Sigma}$ is
\begin{equation}
D:=W_{2}\times\nabla_{W_{2}}^{N}+W_{3}\times\nabla_{W_{3}}^{N}%
.\label{twisted-Dirac-Sigma}%
\end{equation}
Using $\left(  \ref{base-on-sigma}\right)  $ at $p$, we have
\begin{align}
DV\left(  p\right)   &  =\left(  W_{2}\times\nabla_{W_{2}}^{N}+W_{3}%
\times\nabla_{W_{3}}^{N}\right)  \left(  V^{4}W_{4}+V^{5}W_{5}+V^{6}%
W_{6}+V^{7}W_{7}\right)  \nonumber\\
&  =\left(  V_{2}^{4}W_{2}\times W_{4}+V_{3}^{4}W_{3}\times W_{4}\right)
+\left(  V_{2}^{5}W_{2}\times W_{5}+V_{3}^{5}W_{3}\times W_{5}\right)
\nonumber\\
&  +\left(  V_{2}^{6}W_{2}\times W_{6}+V_{3}^{6}W_{3}\times W_{6}\right)
+\left(  V_{2}^{7}W_{2}\times W_{7}+V_{3}^{7}W_{3}\times W_{7}\right)
\nonumber\\
&  =-\left(  V_{2}^{6}+V_{3}^{7}\right)  W_{4}+\left(  V_{3}^{6}-V_{2}%
^{7}\right)  W_{5}\nonumber\\
&  +\left(  V_{2}^{4}-V_{3}^{5}\right)  W_{6}+\left(  V_{3}^{4}+V_{2}%
^{5}\right)  W_{7}.\label{twisted-Dirac-Sigma-expression}%
\end{align}
Therefore
\[
\Phi\left(  \overline{\partial}U\left(  p\right)  \right)  =DV\left(
p\right)  .
\]
Note that Dirac operators are independent on the choice of orthonormal basis
at $p$, and our $p\in\Sigma$ is arbitrary, so $\Phi\left(  \overline{\partial
}U\right)  =DV$ on $\Sigma$. The proof of the proposition is completed.
\end{proof}

We now make connection to Subsection \ref{linear-model}. Taking $L=N_{\Sigma
/C}$ in that subsection, then $\mathbb{S}^{+}=N_{\Sigma/C}$ and $\mathbb{S}%
^{-}=\wedge_{\mathbb{C}}^{0,1}\left(  N_{\Sigma/C}\right)  \,$. Our assumption
that $\Sigma$ is regular implies $\dim H_{\overline{\partial}}^{0,1}\left(
N_{\Sigma/C}\right)  =0$, for $H_{\overline{\partial}}^{0,1}\left(
N_{\Sigma/C}\right)  $ corresponds to the cokernel of $\overline{\partial}$.
By the Dolbeault isomorphism, we have
\[
\dim H^{1}\left(  \Sigma,N_{\Sigma/C}\right)  =\dim H_{\overline{\partial}%
}^{0,1}\left(  N_{\Sigma/C}\right)  =0.
\]
Since the dimension for the Seiberg-Witten moduli is $0$, by the equivalence
to Gromov-Witten moduli we have
\[
\dim H^{0}\left(  \Sigma,N_{\Sigma/C}\right)  =\dim H^{1}\left(
\Sigma,N_{\Sigma/C}\right)  =0,
\]
for we only count $J_{n}$-holomorphic curves $\Sigma$ of index $0$. Hence for
\begin{align*}
\overline{\partial}  &  :\Omega^{0}\left(  N_{\Sigma/C}\right)  \rightarrow
\Omega^{0,1}\left(  N_{\Sigma/C}\right) \\
\overline{\partial}^{\ast}  &  :\Omega^{0,1}\left(  N_{\Sigma/C}\right)
\rightarrow\Omega^{0}\left(  N_{\Sigma/C}\right)
\end{align*}
in Subsection \ref{linear-model}, $\ker\overline{\partial}$ and $\ker
\overline{\partial}^{\ast}$ are trivial. Moreover by Theorem \ref{onto}, the
linear operator
\[
\mathcal{D}:L_{-}^{1,2}\left(  \mathtt{A}_{\varepsilon},\mathbb{S}\right)
\rightarrow L^{2}\left(  \mathtt{A}_{\varepsilon},\mathbb{S}\right)
\]
is one-to-one and onto. \

\subsection{Linearization of Instanton Equation and Comparison with the
Operator $\mathcal{D}$\label{compare-Dirac}}

To prove the main theorem (Theorem \ref{main}), we will construct a map
\[
F_{\varepsilon}:C_{-}^{m,\alpha}\left(  N_{A_{\varepsilon}^{\prime}/M}\right)
\rightarrow C^{m-1,\alpha}\left(  N_{A_{\varepsilon}^{\prime}/M}\right)
\]
such that the solution to the equation $F_{\varepsilon}\left(  V\right)  =0$
will give rise to an associative submanifold (instanton) with boundary lying
on $C_{0}\cup C_{\varepsilon}.$ The spaces $C^{m,\alpha}\left(
N_{A_{\varepsilon}^{\prime}/M}\right)  $ and $C_{-}^{m,\alpha}\left(
N_{A_{\varepsilon}^{\prime}/M}\right)  $ are defined by
\begin{align*}
C^{m}\left(  N_{A_{\varepsilon}^{\prime}/M}\right)   &  :=\left\{  \left.
V\in\Gamma\left(  N_{A_{\varepsilon}^{\prime}/M}\right)  \right\vert
\left\Vert V\right\Vert _{C^{m}\left(  N_{A_{\varepsilon}^{\prime}/M}\right)
}<+\infty\right\}  ,\\
C_{-}^{m}\left(  N_{A_{\varepsilon}^{\prime}/M}\right)   &  :=\left\{
V\in\Gamma\left(  N_{A_{\varepsilon}^{\prime}/M}\right)  \left\vert
\begin{array}
[c]{c}%
\left\Vert V\right\Vert _{C^{m}\left(  N_{A_{\varepsilon}^{\prime}/M}\right)
}<+\infty,\text{ and }\\
V|_{\varphi\left(  \left\{  0\right\}  \times\Sigma\right)  }\subset
TC_{0},V|_{\varphi\left(  \left\{  \varepsilon\right\}  \times\Sigma\right)
}\subset TC_{\varepsilon}.
\end{array}
\right.  \right\}  .
\end{align*}

We construct a three dimensional submanifold $\mathtt{A}_{\varepsilon}%
^{\prime}=\varphi\left(  \mathtt{A}_{\varepsilon}\right)  \subset M$ by
flowing $\Sigma$ along with $C_{t}$. For the spinor bundle
$\mathbb{S\rightarrow}\mathtt{A}_{\varepsilon}$, we will construct an
exponential-like map $\widetilde{\exp}:$ $\mathbb{S\rightarrow}M$, with the
following properties of the differential $d\widetilde{\exp}|_{\mathtt{A}%
_{\varepsilon}}$ on $\left\{  0\right\}  \times\Sigma$:

\begin{enumerate}
\item On fiber directions of $\mathbb{S}$, $d\widetilde{\exp}|_{\left\{
0\right\}  \times\Sigma}=\left(  id,f\right)  :N_{\Sigma/C}\oplus
\wedge_{\mathbb{C}}^{0,1}\left(  N_{\Sigma/C}\right)  \rightarrow N_{\Sigma
/C}\oplus N_{\mathcal{C}/M}|_{\Sigma},$ where $f$ $:\wedge_{\mathbb{C}}%
^{0,1}\left(  N_{\Sigma/C}\right)  \rightarrow N_{\mathcal{C}/M}|_{\Sigma}$ is
the real vector bundle isomorphism in $\left(  \ref{iso-bdl}\right)  $;

\item On base directions of $\mathbb{S}$, $d\widetilde{\exp}|_{\left\{
0\right\}  \times\Sigma}=id:T\Sigma\rightarrow T\Sigma$ and $d\widetilde{\exp
}|_{\left\{  0\right\}  \times\Sigma}:\frac{\partial}{\partial x_{1}%
}\rightarrow n\left(  z\right)  ;$

\item On the boundary $\left\{  0,\varepsilon\right\}  \times\Sigma$ of
$\mathtt{A}_{\varepsilon}$, $\widetilde{\exp}|_{\left\{  0\right\}
\times\Sigma}\left(  \mathbb{S}^{+}\oplus0\right)  \subset C_{0}$, and
$\widetilde{\exp}|_{\left\{  \varepsilon\right\}  \times\Sigma}\left(
\mathbb{S}^{+}\oplus0\right)  \subset C_{\varepsilon}$.
\end{enumerate}

The construction of the map $\widetilde{\exp}$ is somewhat technical. To keep
the main flow of our paper, we postpone it to the appendix.

To make the linear theory developed in the previous section applicable, we
compare the linearization $DF_{\varepsilon}\left(  0\right)  $ with the
operator $\mathcal{D}$ in Section \ref{linear-model} by the following diagram
to get $\varepsilon$-dependent bound of its right inverse:%
\[%
\begin{tabular}
[c]{lll}%
$C_{-}^{m,\alpha}\left(  \mathtt{A}_{\varepsilon},\mathbb{S}\right)  $ &
$\overset{\mathcal{D}}{\longrightarrow}$ & $C_{-}^{m-1,\alpha}\left(
\mathtt{A}_{\varepsilon},\mathbb{S}\right)  $\\
$d\widetilde{\exp}\downarrow$ &  & $\ \downarrow d\widetilde{\exp}$\\
$C_{-}^{m,\alpha}\left(  N_{\mathtt{A}_{\varepsilon}^{\prime}/M}\right)  $ &
$\overset{DF_{\varepsilon}\left(  0\right)  }{\longrightarrow}$ &
$C^{m-1,\alpha}\left(  N_{\mathtt{A}_{\varepsilon}^{\prime}/M}\right)  $%
\end{tabular}
\ \ \ \ \ \
\]

Now we define the nonlinear map $F_{\varepsilon}$ with the important property
that elements in $F_{\varepsilon}^{-1}\left(  0\right)  $ with small norm
correspond to associative submanifolds in $M$ near $A_{\varepsilon}^{\prime}$
for small $\varepsilon$. Given any $C_{0}\cup C_{\varepsilon}$, we modify the
metric $g$ near $C_{0}$ and $C_{\varepsilon}$ to make them \emph{totally
geodesic}. We denote this new smooth metric by $g_{\varepsilon}$, and we make $g_{\varepsilon}$
$C^1$-continuously depend on ${\varepsilon}$ in our construction. Let
$\exp^{g_{\varepsilon}}$ be the exponential map of $g_{\varepsilon}$. Then $\exp^{g_{\varepsilon}}$
has the following properties:

\begin{enumerate}
\item For sections $V$ of $N_{\mathtt{A}_{\varepsilon}^{\prime}/M}$ with
$C^{0}$ norm smaller than a fixed constant \ $\delta_{0}$ (depending on the
uniform injectivity radius of the family of metrics $\left\{  g_{\varepsilon
}\right\}  _{0\leq\varepsilon\leq\varepsilon_{0}}$),
\[
\exp^{g_{\varepsilon}}V:\mathtt{A}_{\varepsilon}^{\prime}\rightarrow M
\]
is a smooth embedding.

\item For $V\in C_{-}^{m,\alpha}\left(  N_{A_{\varepsilon}^{\prime}/M}\right)
$, let
\begin{equation}
\mathtt{A}_{\varepsilon}\left(  V\right)  :=\left(  \exp^{g_{\varepsilon}%
}V\right)  \left(  \mathtt{A}_{\varepsilon}^{\prime}\right)  . \label{A_e_V}%
\end{equation}
Then $\mathtt{A}_{\varepsilon}\left(  V\right)  $ is a submanifold of $M$
nearby $\mathtt{A}_{\varepsilon}^{\prime}$ satisfying the boundary condition
\begin{equation}
\partial\mathtt{A}_{\varepsilon}\left(  V\right)  \subset C_{0}\cup
C_{\varepsilon}. \label{bdryAe}%
\end{equation}

\item For small $t\geq0$, the family of embeddings $\exp^{g_{\varepsilon}%
}\left(  tV\right)  :\mathtt{A}_{\varepsilon}^{\prime}\rightarrow M$ satisfies%
\begin{equation}
\left.  \frac{d}{dt}\right\vert _{t=0}\exp^{g_{\varepsilon}}\left(  tV\right)
=V. \label{expgeV}%
\end{equation}

\end{enumerate}

Next we define $F_{\varepsilon}:C_{-}^{m,\alpha}\left(  \mathtt{A}%
_{\varepsilon}^{\prime},N_{\mathtt{A}_{\varepsilon}^{\prime}/M}\right)
\rightarrow C^{m-1,\alpha}\left(  \mathtt{A}_{\varepsilon}^{\prime
},N_{\mathtt{A}_{\varepsilon}^{\prime}/M}\right)  $,
\begin{equation}
F_{\varepsilon}\left(  V\right)  =\ast_{\mathtt{A}_{\varepsilon}^{\prime}%
}\circ\bot_{\mathtt{A}_{\varepsilon}^{\prime}}\circ\left(  T_{V}\circ\left(
\exp^{g_{\varepsilon}}V\right)  ^{\ast}\tau\right)  , \label{F_epsilon}%
\end{equation}
where

\begin{enumerate}
\item $\left(  \exp^{g_{\varepsilon}}V\right)  ^{\ast}$ pulls back the
\emph{differential form part} of $\tau,$

\item $T_{V}:T_{\exp_{p}^{g_{\varepsilon}}\left(  tV\right)  }M\rightarrow
T_{p}M$ \ pulls back the \emph{vector part }of $\tau$ by the parallel
transport with respect to $g$ along the path $\exp_{p}^{g_{\varepsilon}%
}\left(  tV\right)  ,$

\item $\bot_{\mathtt{A}_{\varepsilon}^{\prime}}:TM|_{\mathtt{A}_{\varepsilon
}^{\prime}}\rightarrow N_{\mathtt{A}_{\varepsilon}^{\prime}/M}$ \ \ is
the\ $\ $orthogonal$\ $projection with respect to $g,$

\item $\ast_{\mathtt{A}_{\varepsilon}^{\prime}}:\Omega^{3}\left(
\mathtt{A}_{\varepsilon}^{\prime}\right)  \rightarrow\Omega^{0}\left(
\mathtt{A}_{\varepsilon}^{\prime}\right)  $ is the quotient by the volume form
$dvol_{\mathtt{A}_{\varepsilon}^{\prime}}$ induced from $g.$\
\end{enumerate}

\bigskip We stress that $g_{\varepsilon}$ is only used to construct a map
$\exp^{g_{\varepsilon}}:N_{\mathtt{A}_{\varepsilon}^{\prime}/M}\rightarrow M$
satisfying the coassociative boundary condition $\left(  \ref{bdryAe}\right)
$ and derivative condition $\left(  \ref{expgeV}\right)  $. For our covariant
derivatives, parallel transport, orthogonal projection and volume form, we
still use the original metric $g$.

To better understand $F_{\varepsilon}$, we let%
\[
P_{\varepsilon}:=\ast_{\mathtt{A}_{\varepsilon}^{\prime}}\circ\bot
_{\mathtt{A}_{\varepsilon}^{\prime}}:\Gamma\left(  TM|_{\mathtt{A}%
_{\varepsilon}^{\prime}}\right)  \otimes\Omega^{3}\left(  \mathtt{A}%
_{\varepsilon}^{\prime}\right)  \rightarrow\Gamma\left(  N_{\mathtt{A}%
_{\varepsilon}^{\prime}/M}\right)  ,
\]
and
\begin{equation}
F^{g_{\varepsilon}}\left(  V\right)  :=T_{V}\circ\left(  \exp^{g_{\varepsilon
}}V\right)  ^{\ast}\tau. \label{Fge}%
\end{equation}
Then
\[
F_{\varepsilon}\left(  V\right)  =P_{\varepsilon}\circ F^{g_{\varepsilon}%
}\left(  V\right)  .
\]
At any $p\in\mathtt{A}_{\varepsilon}^{\prime}$, both $\bot_{\mathtt{A}%
_{\varepsilon}^{\prime}}$ and $\ast_{\mathtt{A}_{\varepsilon}^{\prime}}$are
linear operators between finite dimensional spaces, so $\left\Vert
P_{\varepsilon}\right\Vert \leq C$, where the constant $C$ only depends on
$\varphi$ and is uniform for all $0<\varepsilon\leq\varepsilon_{0}$. We notice
that $P_{\varepsilon}$ does not involve $V$ so the essential part of
$F_{\varepsilon}\left(  V\right)  $ is $F^{g_{\varepsilon}}\left(  V\right)  $.

By Proposition \ref{1Prop alm instanton}, if $\mathtt{A}_{\varepsilon}%
^{\prime}$ is sufficiently close to being associative, then there exists
$\delta>0$, such that for $\left\Vert V\right\Vert _{C_{-}^{1,\alpha}\left(
A_{\varepsilon}^{\prime},N_{A_{\varepsilon}^{\prime}/M}\right)  }<\delta$,
\[
F_{\varepsilon}\left(  V\right)  =0\text{ }\Leftrightarrow F^{g_{\varepsilon}%
}\left(  V\right)  =0\Leftrightarrow\mathtt{A}_{\varepsilon}\left(  V\right)
\text{ associative. }%
\]
We also note that $\partial\mathtt{A}_{\varepsilon}\left(  V\right)  \subset
C_{0}\cup C_{\varepsilon}$ for $V\in$ $C_{-}^{1,\alpha}\left(  A_{\varepsilon
}^{\prime},N_{A_{\varepsilon}^{\prime}/M}\right)  $.

Before we compute $F^{g_{\varepsilon}\prime}\left(  0\right)  V$, we observe
the following useful fact. The original exponential map $\exp^{g}$
$:N_{\mathtt{A}_{\varepsilon}^{\prime}/M}\rightarrow M$ does not satisfy the
the coassociative boundary condition $\left(  \ref{bdryAe}\right)  $, but
satisfies the same derivative condition $\left(  \ref{expgeV}\right)  $:
\[
\left.  \frac{d}{dt}\right\vert _{t=0}\exp^{g}\left(  tV\right)  =V=\left.
\frac{d}{dt}\right\vert _{t=0}\exp^{g_{\varepsilon}}\left(  tV\right)  .
\]
For smooth maps $f:A_{\varepsilon}^{\prime}\rightarrow M$ and a smooth form
$\tau$, the nonlinear map $\Gamma:f\rightarrow f^{\ast}\tau$ is differentiable
with respect to $f$. Therefore%
\begin{equation}
\left.  \frac{d}{dt}\right\vert _{t=0}\left(  \exp^{g}\left(  tV\right)
\right)  ^{\ast}\tau=\Gamma^{\prime}\left(  0\right)  V=\left.  \frac{d}%
{dt}\right\vert _{t=0}\left(  \exp^{g_{\varepsilon}}\left(  tV\right)
\right)  ^{\ast}\tau. \label{exp-tau-derivative-agree}%
\end{equation}
Recall the $F\left(  V\right)  $ $\left(  \ref{pull-vec-form}\right)  $
defined in our proof of McLean's theorem is%
\[
F\left(  V\right)  :=T_{V}\circ\left(  \exp^{g}V\right)  ^{\ast}\tau,
\]
where the parallel transport $T_{V}$ is with respect to $g$ along the geodesic
$\exp^{g}\left(  tV\right)  $. $F^{g_{\varepsilon}}\left(  V\right)  $ and
$F\left(  V\right)  $ are different nonlinear maps, but $\left(
\ref{exp-tau-derivative-agree}\right)  $ says that
\begin{equation}
F^{g_{\varepsilon}\prime}\left(  0\right)  V=F^{\prime}\left(  0\right)
V\text{.} \label{F-deri-agree}%
\end{equation}
We give an alternative proof of $\left(  \ref{F-deri-agree}\right)  $ in the
following Lemma \ref{F-derivative-agree}. The proof gives more information of
$F^{g_{\varepsilon}\prime}\left(  0\right)  V$ when there is a good frame
field for the vector-valued form $\tau$.

\bigskip

\begin{lemma}
\label{F-derivative-agree}For any $V\in C_{-}^{1,\alpha}\left(
N_{A_{\varepsilon}^{\prime}/M}\right)  $, $F^{g_{\varepsilon}\prime}\left(
0\right)  V=F^{\prime}\left(  0\right)  V$.
\end{lemma}

\begin{proof}
For any $p\in A_{\varepsilon}^{\prime}$, we choose a frame field $\left\{
W_{a}\right\}  _{1\leq\alpha\leq7}$ in its neighborhood $B\subset
A_{\varepsilon}^{\prime}$, and then extend $V$ and $\left\{  W_{a}\right\}
_{1\leq\alpha\leq7}$ to $M$ by the parallel transport with respect to $g\,$\ along the
curve $\exp_p^{g_{\varepsilon}}\left(  tV\right)  $. We write $\tau
=\omega^{\alpha}\otimes W_{\alpha}$ in the neighborhood of $p$ in $M$, following Einstein's summation convention. Similar
to $\left(  \ref{DF}\right)  $, we compute%
\[
F^{g_{\varepsilon}\prime}\left(  0\right)  V=d\left(  i_{V}\omega^{\alpha
}\right)  \otimes W_{\alpha}+i_{V}d\omega^{\alpha}\otimes W_{\alpha}%
+\omega^{\alpha}\otimes\nabla_{V}W_{\alpha},
\]
where the covariant derivative $\nabla$ is with respect to $g$, and we have used that $\frac{d}{dt}%
|_{t=0}\exp^{g_{\varepsilon}}\left(  tV\right)  =V$. By the parallel property
of $\tau$ and $W_{\alpha}$ with respect to $g$, we have%
\[
i_{V}d\omega^{\alpha}\otimes W_{\alpha}=0=\omega^{\alpha}\otimes\nabla
_{V}W_{\alpha}%
\]
(also see Remark 10 item 3). The first term $d\left(  i_{V}\omega^{\alpha
}\right)  |_{A_{\varepsilon}^{\prime}}$ only depends on the restriction of $V$ and $\omega^{\alpha}$
on $A_{\varepsilon}^{\prime}$. (Remark 10 item 1). Therefore%
\[
F^{g_{\varepsilon}\prime}\left(  0\right)  V=d\left(  i_{V}\omega^{\alpha
}\right)  \otimes W_{\alpha}|_{A_{\varepsilon}^{\prime}}=F^{\prime}\left(
0\right)  V.
\]
\end{proof}

\begin{remark}
To apply the implicit function theorem to $F_{\varepsilon}\left(  V\right)  $,
in the remaining part of our paper we will only need the estimate of
$F_{\varepsilon}^{\prime}\left(  0\right)  $, and the quadratic estimate
$\left(  \ref{quadratic-estmt}\right)  $ of $F_{\varepsilon}^{\prime}\left(
V\right)  $. Because of the above lemma, to compute $F_{\varepsilon}^{\prime
}\left(  0\right)  =P_{\varepsilon}\circ F^{g_{\varepsilon}\prime}\left(
0\right)  $, we can replace the metric $g_{\varepsilon}$ by $g$ in $\left(
\ref{Fge}\right)  $. This will simplify the exposition in many places. For our
quadratic estimate $\left(  \ref{quadratic-estmt}\right)  $, the proof uses no
feature of the $G_{2}$ metric $g$ and is valid for any Riemannian metrics,
including $g_{\varepsilon}$. The constant $C$ in $\left(  \ref{quadratic-estmt}\right)  $ is uniform for
$\left\{  g_{\varepsilon}\right\}  _{0\leq\varepsilon\leq\varepsilon_{0}}$,
since this is a compact family of metrics $C^1$-continuously depending on $\varepsilon$.
So although we used $g_{\varepsilon
}$ in the definition of $F_{\varepsilon}\left(  V\right)  $ $\left(
\ref{F_epsilon}\right)  $, from now on we will pretend $g_{\varepsilon}$ is
$g$ in $F_{\varepsilon}\left(  V\right)  $, and we will simply write $\exp
^{g}$ as $\exp$.
\end{remark}

\begin{proposition}
\label{compare-linear-model}\bigskip For any section $V_{1}$ of $\mathbb{S}$
and section $V_{2}:=d\widetilde{\exp}\cdot V_{1}$ of $N_{\mathtt{A}%
_{\varepsilon}^{\prime}/M}$, we have%
\[
\left\Vert F^{\prime}\left(  0\right)  V_{2}-\left(  d\widetilde{\exp}%
\circ\mathcal{D}V_{1}\right)  \otimes dvol_{A_{\varepsilon}^{\prime}%
}\right\Vert _{C^{\alpha}\left(  A_{\varepsilon}^{\prime},N_{A_{\varepsilon
}^{\prime}/M}\right)  }\leq C\varepsilon^{1-\alpha}\left\Vert V_{1}\right\Vert
_{C^{1,\alpha}\left(  A_{\varepsilon},\mathbb{S}\right)  },
\]
and%
\[
\left\Vert F_{\varepsilon}^{\prime}\left(  0\right)  V_{2}-\left(
d\widetilde{\exp}\right)  \circ\mathcal{D}V_{1}\right\Vert _{C^{\alpha}\left(
A_{\varepsilon}^{\prime},N_{A_{\varepsilon}^{\prime}/M}\right)  }\leq
C\varepsilon^{1-\alpha}\left\Vert V_{1}\right\Vert _{C^{1,\alpha}\left(
A_{\varepsilon},\mathbb{S}\right)  },
\]
where $\mathcal{D}$ is the operator on $\mathbb{S}$ in linear model, and the
constant $C$ is uniform for all $\varepsilon$.
\end{proposition}

\begin{proof}
For each $p=\varphi\left(  0,z\right)  $ in $\Sigma_{0}:=\varphi\left(
\left\{  0\right\}  \times\Sigma\right)  \subset A_{\varepsilon}^{\prime}$, we
choose \textquotedblleft good\textquotedblright\ frame $\left\{  W_{\alpha
}\right\}  _{\alpha=1,2,\cdots7}$ on $B_{\varepsilon}\left(  z\right)
\subset\Sigma$ as before. Then we extend the frame to $U_{\varepsilon}\left(
p\right)  :=\varphi\left(  \left[  0,\varepsilon\right]  \times B_{\varepsilon
}\left(  z\right)  \right)  $ such that
\[
W_{\alpha}\left(  \varphi\left(  t,z\right)  \right)  =T_{\gamma}\cdot
W_{\alpha}\left(  \varphi\left(  0,z\right)  \right)  \text{ \ for }%
\alpha=1,2,\cdots7,
\]
where $T_{\gamma}$ is the parallel transport along the path $\gamma
:=\varphi\left(  \left[  0,t\right]  \times\left\{  z\right\}  \right)  $ by
the Levi-Civita connection on $M$. We further extend $\left\{  W_{\alpha
}\right\}  _{\alpha=1,2,\cdots7}$ to a tubular neighborhood of $U_{\varepsilon
}\left(  p\right)  \subset M$ by parallel transport in fiber directions of
$N_{A_{\varepsilon}^{\prime}/M}$. Both $F^{\prime}\left(  0\right)  V_{2}$ and
$\mathcal{D}V_{1}$ are globally defined on $A_{\varepsilon}^{\prime}$ and
$A_{\varepsilon}$ respectively. We are going to compare them in
$U_{\varepsilon}\left(  p\right)  $ using the \textquotedblleft
good\textquotedblright\ frame.
We write $\tau=\omega^{\alpha}\otimes W_{\alpha}$ in the neighborhood of $p$.
By the parallel property of $\left\{  W_{\alpha}\right\}  _{\alpha
=1,2,\cdots7}$, the $3$-forms $\omega^{\alpha}$ are similar to the standard
ones in $\operatorname{Im}\mathbb{O}$, in the sense that $dx_{i}$ is replaced
by $\left(  W_{i}\right)  ^{\ast}$for $i=1,2,\cdots7$. \ We still have the
formula $\left(  \ref{DF}\right)  $ at $q:$%
\begin{align}
F^{\prime}\left(  0\right)  V_{2}\left(  q\right)   &  =d\left(  i_{V_{2}%
}\omega^{\alpha}\right)  \otimes W_{\alpha}\left(  q\right) \nonumber\\
&  +i_{V_{2}}d\omega^{\alpha}\otimes W_{\alpha}\left(  q\right)
+\omega^{\alpha}\otimes\nabla_{V_{2}}W_{\alpha}\left(  q\right)  .
\label{DFV-q}%
\end{align}
The first term $d\left(  i_{V_{2}}\omega^{\alpha}\right)  \otimes W_{\alpha
}\left(  q\right)  $ is the principal symbol part of the differential operator
$F^{\prime}\left(  0\right)  $. We claim that
\begin{equation}
d\left(  i_{V_{2}}\omega^{\alpha}\right)  \otimes W_{\alpha}=\left(
d\widetilde{\exp}\cdot\mathcal{D}V_{1}\right)  \otimes dvol_{A_{\varepsilon
}^{\prime}}+E\left(  q\right)  V_{1}\label{relate-linear-model}%
\end{equation}
where $E\left(  q\right)  $ is a smooth tensor on $U_{\varepsilon}\left(
p\right)  $\ with $\left\Vert E\right\Vert _{C^{1}\left(  U_{\varepsilon
}\left(  p\right)  \right)  }\leq C_{3}\varepsilon$.
To see this, we first compute $d\left(  i_{V_{2}}\omega^{\alpha}\right)
\otimes W_{\alpha}$. For section $V_{2}$ of $N_{A_{\varepsilon}^{\prime}/M}$
we write
\[
V_{2}\left(  q\right)  =\Sigma_{\alpha=4}^{7}\phi^{\alpha}\left(  q\right)
W_{\alpha}\left(  q\right)  .
\]
Similar to our proof of McLean's theorem, we have
\begin{equation}
d\left(  i_{V_{2}}\omega^{\alpha}\right)  \otimes W_{\alpha}\left(  q\right)
|_{T_{q}A_{\varepsilon}^{\prime}}=\mathcal{D}V_{2}\left(  q\right)
dvol_{A_{\varepsilon}^{\prime}}+E_{1}\left(  q\right)  V_{2}
\label{almost-Dirac}%
\end{equation}
where
\begin{align}
\mathcal{D}V_{2}\left(  q\right)   &  =-\left(  \phi_{1}^{5}+\phi_{2}^{6}%
+\phi_{3}^{7}\right)  W_{4}+\left(  \phi_{1}^{4}+\phi_{3}^{6}-\phi_{2}%
^{7}\right)  W_{5}\nonumber\\
&  +\left(  \phi_{2}^{4}-\phi_{3}^{5}+\phi_{1}^{7}\right)  W_{6}+\left(
\phi_{3}^{4}+\phi_{2}^{5}-\phi_{1}^{6}\right)  W_{7}, \label{DV2}%
\end{align}
and
\[
\phi_{i}^{k}\left(  q\right)  :=d\phi^{k}\left(  q\right)  \left(
W_{i}\right)
\]
The reason is the following: From our construction of $W_{\alpha}$,
$\nabla_{W_{i}}W_{j}\left(  p\right)  =0$ for $1\leq i,j\leq7$, and
$\omega^{\alpha}$ are similar to the standard ones in $\operatorname{Im}%
\mathbb{O}$. Thus when $q=p$,%
\[
d\left(  i_{V_{2}}\omega^{\alpha}\right)  \otimes W_{\alpha}\left(  p\right)
|_{T_{p}A_{\varepsilon}^{\prime}}=\mathcal{D}V_{2}\left(  p\right)
dvol_{A_{\varepsilon}^{\prime}}\text{.}%
\]
If $q$ is $\varepsilon$-close to $p$, then in $\left(  \ref{almost-Dirac}%
\right)  $ the error term $E_{1}\left(  q\right)  $ is of order $\varepsilon$
in $C^{1}$ norm, because $\nabla_{W_{i}}W_{j}\left(  q\right)  =o\left(
\varepsilon\right)  $ in $C^{1}$ and span$\left\{  W_{\alpha}\left(  q\right)
\right\}  _{\alpha=1,2,3}$ has $\varepsilon$-order deviation from
$T_{q}A_{\varepsilon}^{\prime}$ in $C^{1}$.
We want to show
\begin{equation}
\mathcal{D}V_{2}\left(  q\right)  =d\widetilde{\exp}\circ\left[  \left(
h^{-1/2}\left(  z\right)  e_{1}\cdot\frac{\partial}{\partial x_{1}}%
+\bar{\partial}\right)  V_{1}\right]  +E_{2}\left(  q\right)  V_{2},
\label{Compare-Dirac}%
\end{equation}
where the error term $E_{2}\left(  q\right)  $ is of order $\varepsilon$ in
$C^{1}$ norm, and the $\bar{\partial}$ is the Dolbeault Dirac operator on
$N_{\Sigma/C}\oplus\wedge_{\mathbb{C}}^{0,1}\left(  N_{\Sigma/C}\right)  $.
Similar to our argument for $E_{1}\left(  q\right)  $, it is enough to show
\[
\mathcal{D}V_{2}\left(  p\right)  =d\widetilde{\exp}|_{\left\{  0\right\}
\times\Sigma}\circ\left[  \left(  h^{-1/2}\left(  z\right)  e_{1}\cdot
\frac{\partial}{\partial x_{1}}+\bar{\partial}\right)  V_{1}\left(
0,z\right)  \right]  .
\]
We observe that
\[
\mathcal{D}V_{2}\left(  p\right)  =DV_{2}\left(  p\right)  -\phi_{1}^{5}%
W_{4}+\phi_{1}^{4}W_{5}+\phi_{1}^{7}W_{6}-\phi_{1}^{6}W_{7}%
\]
from $\left(  \ref{DV2}\right)  $\ and $\left(
\ref{twisted-Dirac-Sigma-expression}\right)  $, where $D$ is the twisted Dirac
operator $\left(  \ref{twisted-Dirac-Sigma}\right)  $ of $N_{\Sigma/C}\oplus
N_{\mathcal{C}/M}|_{\Sigma}$ over $\Sigma$ in previous subsection. By
Proposition \ref{dirac-agree}, we have
\[
d\widetilde{\exp}|_{\left\{  0\right\}  \times\Sigma}\cdot\bar{\partial}%
V_{1}=\Phi\left(  \bar{\partial}V_{1}\right)  =DV_{2}.
\]
So it is enough to prove, at $p=\varphi\left(  0,z\right)  $, that%
\begin{equation}
d\widetilde{\exp}\circ\left[  h^{-1/2}\left(  z\right)  e_{1}\cdot
\frac{\partial}{\partial x_{1}}V_{1}\left(  0,z\right)  \right]  =\left(
-\phi_{1}^{5}W_{4}+\phi_{1}^{4}W_{5}+\phi_{1}^{7}W_{6}-\phi_{1}^{6}%
W_{7}\right)  \left(  p\right)  . \label{Compare-Dirac-n}%
\end{equation}
From
\[
V_{2}\left(  q\right)  =\Sigma_{\alpha=4}^{7}\phi^{\alpha}\left(  q\right)
W_{\alpha}\left(  q\right)  =\Sigma_{\alpha=4}^{7}\phi^{\alpha}\left(
q\right)  \cdot T_{\gamma}W_{\alpha}\left(  \varphi\left(  0,z\right)
\right)
\]
we get%
\begin{align*}
V_{1}\left(  t,z\right)   &  =\Sigma_{\alpha=4}^{7}\phi^{\alpha}\left(
\varphi\left(  t,z\right)  \right)  \cdot\left(  d\widetilde{\exp}\right)
^{-1}T_{\gamma}W_{\alpha}\left(  \varphi\left(  0,z\right)  \right) \\
&  =\Sigma_{\alpha=4}^{7}\phi^{\alpha}\left(  \varphi\left(  t,z\right)
\right)  \cdot\Phi^{-1}W_{\alpha}\left(  \varphi\left(  0,z\right)  \right)
+E_{3}\left(  q\right)  V_{2},
\end{align*}
where $\left\{  \Phi^{-1}W_{\alpha}\left(  \varphi\left(  0,z\right)  \right)
\right\}  _{\alpha=4}^{7}$ is regarded as a frame of $\mathbb{S\rightarrow
}A_{\varepsilon}$ that is invariant along $t$ direction, and $E_{3}\left(
q\right)  $ is of order $\varepsilon$ in $C^{1}$ norm. The second identity of
$V_{1}\left(  t,z\right)  $ is because the maps $d\widetilde{\exp}$ and
$T_{\gamma}$, at $t=0$, are maps $\Phi$ and the identity map respectively on
span$\left\{  W_{\alpha}\right\}  _{4\leq\alpha\leq7}$, and for $0\leq
t\leq\varepsilon$ we have the error term $E_{3}\left(  q\right)  $ of desired
order by smoothness of $\varphi,\phi^{\alpha}$ and $W_{\alpha}$. Let
\[
\psi^{\alpha}\left(  x_{1},z\right)  :=\phi^{\alpha}\left(  \varphi\left(
x_{1},z\right)  \right)  \text{, and }\psi_{1}^{\alpha}\left(  x_{1},z\right)
=\frac{\partial}{\partial x_{1}}\psi^{\alpha}\left(  x_{1},z\right)  .\text{ }%
\]
We have%
\begin{align}
&  e_{1}\cdot\frac{\partial}{\partial x_{1}}\left[  \Sigma_{\alpha=4}^{7}%
\psi^{\alpha}\left(  x_{1},z\right)  \cdot\Phi^{-1}W_{\alpha}\left(
\varphi\left(  0,z\right)  \right)  \right] \nonumber\\
&  =\Phi^{-1}\left[  W_{1}\times\Sigma_{\alpha=4}^{7}\psi_{1}^{\alpha}\left(
x_{1},z\right)  \cdot W_{\alpha}\left(  \varphi\left(  0,z\right)  \right)
\right] \nonumber\\
&  =\Phi^{-1}\left[  -\psi_{1}^{5}W_{4}+\psi_{1}^{4}W_{5}+\psi_{1}^{7}%
W_{6}-\psi_{1}^{6}W_{7}\right]  . \label{DV1}%
\end{align}
By the chain rule, for $4\leq\alpha\leq7$,
\begin{equation}
\psi_{1}^{\alpha}\left(  0,z\right)  =\nabla_{W_{1}}\phi^{\alpha}\left(
\varphi\left(  0,z\right)  \right)  \cdot h\left(  z\right)  ^{\frac{1}{2}%
}=h\left(  z\right)  ^{\frac{1}{2}}\phi_{1}^{\alpha}\left(  \varphi\left(
0,z\right)  \right)  \label{Chain-rule}%
\end{equation}
where the factor $h\left(  z\right)  ^{\frac{1}{2}}$ is from $\left\langle
W_{1},\frac{d}{dt}\varphi\left(  t,z\right)  |_{t=0}\right\rangle =h\left(
z\right)  ^{\frac{1}{2}}$. Comparing $\left(  \ref{Compare-Dirac-n}\right)  $
and $\left(  \ref{DV1}\right)  $, and noticing $\left(  \ref{Chain-rule}%
\right)  $, we have%
\[
d\widetilde{\exp}|_{\left\{  0\right\}  \times\Sigma}\circ\mathcal{D}%
V_{1}=\mathcal{D}V_{2}.
\]
Therefore for $q$ that is $\varepsilon$-close to $p$, claim $\left(
\ref{relate-linear-model}\right)  $\ is proved. So we get
\begin{align}
&  \left\Vert d\left(  i_{V_{2}}\omega^{\alpha}\right)  \otimes W_{\alpha
}-\left(  d\widetilde{\exp}\circ\mathcal{D}V_{1}\right)  \otimes
dvol_{A_{\varepsilon}^{\prime}}\right\Vert _{C^{\alpha}\left(  U_{\varepsilon
}\left(  p\right)  \right)  }\nonumber\\
&  =\left\Vert E\left(  q\right)  V_{2}\right\Vert _{C^{\alpha}\left(
U_{\varepsilon}\left(  p\right)  \right)  }\nonumber\\
&  \leq C_{1}\varepsilon^{1-\alpha}\left\Vert V_{2}\right\Vert _{C^{1,\alpha
}\left(  U_{\varepsilon}\left(  p\right)  \right)  },
\label{1st-order-Deviation}%
\end{align}
where the constant $C_{1}$ is uniform for all $\varepsilon$.
The remaining term
\[
BV_{2}:=i_{V_{2}}d\omega^{\alpha}\otimes W_{\alpha}+\omega^{\alpha}%
\otimes\nabla_{V_{2}}W_{\alpha}%
\]
is a $0$-th order linear operator on $V_{2}$, where $B=B\left(  q\right)  $ is
a smooth tensor on $A_{\varepsilon}^{\prime}$. Since $\nabla\tau=0$, same as
in our proof in Theorem \ref{McLean-thm}, we have
\[
i_{W}d\omega^{\alpha}\left(  p\right)  =\nabla_{W}W_{\alpha}\left(  p\right)
=0\text{ for all }W\in N_{A_{\varepsilon}^{\prime}/M}\left(  p\right)  ,
\]
so we have $B\left(  p\right)  |_{N_{A_{\varepsilon}^{\prime}/M}\left(
p\right)  }=0$. Since dist$\left(  p,q\right)  \leq C_{0}\varepsilon$ for some
uniform constant $C_{0}$, and $N_{A_{\varepsilon}^{\prime}/M}$ is a smooth
fiber bundle over $A_{\varepsilon}^{\prime}$, we conclude that the operator
norm
\[
\left\Vert B\left(  q\right)  |_{N_{A_{\varepsilon}^{\prime}/M}\left(
q\right)  }\right\Vert \leq C_{2}\varepsilon
\]
for some uniform constant $C_{2}$. From this it is easy to prove
\begin{equation}
\left\Vert BV_{2}\right\Vert _{C^{\alpha}\left(  U_{\varepsilon}\left(
p\right)  \right)  }\leq C_{3}\varepsilon^{1-\alpha}\left\Vert V_{2}%
\right\Vert _{C^{1,\alpha}\left(  A_{\varepsilon}^{\prime},N_{A_{\varepsilon
}^{\prime}/M}\right)  }.\label{0th-order-deviation}%
\end{equation}
The constant $C_{3}$ is uniform for all $\varepsilon$, by the compactness of
$\Sigma$ and finite covering of $A_{\varepsilon}^{\prime}$ by $U_{\varepsilon
}\left(  p\right)  $.
Putting $\left(  \ref{0th-order-deviation}\right)  $ and $\left(
\ref{1st-order-Deviation}\right)  $ in $\left(  \ref{DFV-q}\right)  $, we have%
\begin{align*}
&  \left\Vert F^{\prime}\left(  0\right)  V_{2}-\left(  d\widetilde{\exp}%
\cdot\mathcal{D}V_{1}\right)  \otimes dvol_{A_{\varepsilon}^{\prime}%
}\right\Vert _{C^{\alpha}\left(  U_{\varepsilon}\left(  p\right)  \right)  }\\
&  \leq C_{4}\varepsilon^{1-\alpha}\left\Vert V_{2}\right\Vert _{C^{1,\alpha
}\left(  A_{\varepsilon}^{\prime},N_{A_{\varepsilon}^{\prime}/M}\right)  }\\
&  \leq C_{5}\varepsilon^{1-\alpha}\left\Vert V_{1}\right\Vert _{C^{1,\alpha
}\left(  A_{\varepsilon},\mathbb{S}\right)  }%
\end{align*}
where the constants $C_{4}$ and $C_{5}$ are uniform for all $\varepsilon$.
Using finitely many $U_{\varepsilon}\left(  p\right)  $ covering
$A_{\varepsilon}^{\prime}$ and then taking supremum, we have
\begin{equation}
\left\Vert F^{\prime}\left(  0\right)  V_{2}-\left(  d\widetilde{\exp}%
\cdot\mathcal{D}V_{1}\right)  \otimes dvol_{A_{\varepsilon}^{\prime}%
}\right\Vert _{C^{\alpha}\left(  A_{\varepsilon}^{\prime}\right)  }\leq
C\varepsilon^{1-\alpha}\left\Vert V_{1}\right\Vert _{C^{1,\alpha}\left(
A_{\varepsilon},\mathbb{S}\right)  }. \label{almost-Dirac1}%
\end{equation}
Applying $P_{\varepsilon}$ on the left hand side of the above inequality, and
noticing that $\left(  d\widetilde{\exp}\right)  |_{\left\{  0\right\}
\times\Sigma}\circ\mathcal{D}V_{1}=\mathcal{D}V_{2}$ is a section of
$N_{A_{\varepsilon}^{\prime}/M}$, we have
\[
P_{\varepsilon}\circ\left(  \left(  d\widetilde{\exp}\right)  \circ
\mathcal{D}V_{1}\left(  t,z\right)  \otimes dvol_{A_{\varepsilon}^{\prime}%
}\right)  =\left(  d\widetilde{\exp}\right)  \circ\mathcal{D}V_{1}\left(
t,z\right)  +E_{3}\left(  q\right)  V_{1},
\]
where $E_{3}\left(  q\right)  $ is of order $\varepsilon$ in $C^{1}$ norm,
because at $q=\varphi\left(  t,z\right)  $ the vector $\left(  d\widetilde
{\exp}\right)  \circ\mathcal{D}V_{1}\left(  t,z\right)  $ may have component
of order $\varepsilon$ in $C^{1}$ norm orthogonal to $N_{A_{\varepsilon
}^{\prime}/M}\left(  q\right)  $. So we get
\begin{align}
&  \left\Vert F_{\varepsilon}^{\prime}\left(  0\right)  V_{2}-\left(
d\widetilde{\exp}\right)  \circ\mathcal{D}V_{1}\right\Vert _{C^{\alpha}\left(
A_{\varepsilon}^{\prime},N_{A_{\varepsilon}^{\prime}/M}\right)  }\nonumber\\
&  \leq\left\Vert P_{\varepsilon}\cdot\left(  F^{\prime}\left(  0\right)
V_{2}-\left(  d\widetilde{\exp}\right)  \circ\mathcal{D}V_{1}\otimes
dvol_{A_{\varepsilon}^{\prime}}\right)  \right\Vert _{C^{\alpha}\left(
A_{\varepsilon}^{\prime},N_{A_{\varepsilon}^{\prime}/M}\right)  }+\left\Vert
E_{3}\left(  q\right)  V_{1}\right\Vert _{C^{\alpha}\left(  A_{\varepsilon
}^{\prime}\right)  }\nonumber\\
&  \leq C\varepsilon^{1-\alpha}\left\Vert V_{1}\right\Vert _{C^{1,\alpha
}\left(  A_{\varepsilon},\mathbb{S}\right)  } \nonumber\label{almost-Dirac2}%
\end{align}
where $C$ is a uniform constant independent on $\varepsilon$.
\end{proof}

\begin{proposition}
\label{DFe-inverse-bd}There exists a right inverse $\widetilde{Q_{\varepsilon
}\bigskip}^{true}$ of $F_{\varepsilon}^{\prime}\left(  0\right)  $, such that
$\left\Vert \widetilde{Q_{\varepsilon}\bigskip}^{true}\right\Vert \leq
C\varepsilon^{-\left(  \frac{3}{p}+2\alpha\right)  }$, where the constant $C$
is uniform for all $0<\varepsilon\leq\varepsilon_{0}$.
\end{proposition}

\begin{proof}
From Theorem \ref{e-inverse-bound}, for operator $\mathcal{D}$ on spinor
bundle $\mathbb{S}$ over $\mathtt{A}_{\varepsilon}$, there is a right inverse
$Q_{\varepsilon}$ of $\mathcal{D}$ such that $\left\Vert Q_{\varepsilon
}\right\Vert \leq C\varepsilon^{-\left(  \frac{3}{p}+2\alpha\right)  }$. Let
\begin{align*}
\widetilde{\mathcal{D}}  &  =d\widetilde{\exp}\circ\mathcal{D}\circ\left(
d\widetilde{\exp}\right)  ^{-1},\\
\widetilde{Q_{\varepsilon}}  &  =d\widetilde{\exp}\circ Q_{\varepsilon}%
\circ\left(  d\widetilde{\exp}\right)  ^{-1},
\end{align*}
then
\[
\left\Vert \widetilde{Q_{\varepsilon}}\right\Vert \leq\left\Vert
d\widetilde{\exp}\right\Vert \cdot C\varepsilon^{-\left(  \frac{3}{p}%
+2\alpha\right)  }\cdot\left\Vert \left(  d\widetilde{\exp}\right)
^{-1}\right\Vert \leq C\varepsilon^{-\left(  \frac{3}{p}+2\alpha\right)  },
\]
and
\begin{align*}
F_{\varepsilon}^{\prime}\left(  0\right)  \widetilde{Q_{\varepsilon}}-id  &
=\left(  F_{\varepsilon}^{\prime}\left(  0\right)  -\widetilde{\mathcal{D}%
}+\widetilde{\mathcal{D}}\right)  \widetilde{Q_{\varepsilon}}-id\\
&  =\left(  F_{\varepsilon}^{\prime}\left(  0\right)  -\widetilde{\mathcal{D}%
}\right)  \widetilde{Q_{\varepsilon}}+\widetilde{\mathcal{D}}\widetilde
{Q_{\varepsilon}}-id\\
&  =\left(  F_{\varepsilon}^{\prime}\left(  0\right)  -\widetilde{\mathcal{D}%
}\right)  \widetilde{Q_{\varepsilon}},
\end{align*}
where the last identity is because $\widetilde{\mathcal{D}}\widetilde
{Q_{\varepsilon}}=d\widetilde{\exp}\circ\mathcal{D}Q_{\varepsilon}\circ\left(
d\widetilde{\exp}\right)  ^{-1}=id$. From the previous proposition $\left\Vert
\left(  F_{\varepsilon}^{\prime}\left(  0\right)  -\widetilde{\mathcal{D}%
}\right)  \right\Vert \leq C\varepsilon^{1-\alpha}.$ Therefore
\begin{align*}
\left\Vert F_{\varepsilon}^{\prime}\left(  0\right)  \widetilde{Q_{\varepsilon
}}-id\right\Vert  &  \leq\left\Vert \left(  F_{\varepsilon}^{\prime}\left(
0\right)  -\widetilde{\mathcal{D}}\right)  \right\Vert \left\Vert
\widetilde{Q_{\varepsilon}}\right\Vert \\
&  \leq C\varepsilon^{1-\alpha}\cdot C\varepsilon^{-\left(  \frac{3}%
{p}+2\alpha\right)  }<\frac{1}{2}%
\end{align*}
when $\varepsilon$ is sufficiently small, by our assumption that $1-\left(
\frac{3}{p}+3\alpha\right)  >0$. So $\widetilde{Q_{\varepsilon}}$ is an
approximate right inverse of $F_{\varepsilon}^{\prime}\left(  0\right)  $. The
true right inverse $\widetilde{Q}^{true}$ of $F_{\varepsilon}^{\prime}\left(
0\right)  $ is $\widetilde{Q}^{true}=\widetilde{Q_{\varepsilon}}\left(
F_{\varepsilon}^{\prime}\left(  0\right)  \widetilde{Q_{\varepsilon}}\right)
^{-1}$and
\[
\left\Vert \widetilde{Q}^{true}\right\Vert \leq\left\Vert \widetilde
{Q_{\varepsilon}}\right\Vert \left\Vert \left(  F_{\varepsilon}^{\prime
}\left(  0\right)  \widetilde{Q_{\varepsilon}}\right)  ^{-1}\right\Vert \leq
C\varepsilon^{-\left(  \frac{3}{p}+2\alpha\right)  }\cdot2\leq C\varepsilon
^{-\left(  \frac{3}{p}+2\alpha\right)  }.
\]
\end{proof}

\subsection{Quadratic Estimates\label{quadratic}}

\begin{proposition}
\label{quadratic-estimate}There exists $\delta_{0}>0$ such that for all
sections $V_{0},V$ of $N_{A_{\varepsilon}^{\prime}/M}$ that $\left\Vert
V_{0}\right\Vert _{C^{1,\alpha}\left(  A_{\varepsilon}^{\prime}%
,N_{A_{\varepsilon}^{\prime}/M}\right)  }<\delta_{0}$, we have%
\begin{align}
&  \left\Vert F_{\varepsilon}^{\prime}\left(  V_{0}\right)  V-F_{\varepsilon
}^{\prime}\left(  0\right)  V\right\Vert _{C^{\alpha}\left(  A_{\varepsilon
}^{\prime},N_{A_{\varepsilon}^{\prime}/M}\right)  }\nonumber\\
&  \leq C\left\Vert V_{0}\right\Vert _{C^{1,\alpha}\left(  A_{\varepsilon
}^{\prime},N_{A_{\varepsilon}^{\prime}/M}\right)  }\left\Vert V\right\Vert
_{C^{1,\alpha}\left(  A_{\varepsilon}^{\prime},N_{A_{\varepsilon}^{\prime}%
/M}\right)  ,} \label{quadratic-estmt}%
\end{align}
where the constant $C$ is independent on $\varepsilon$.
\end{proposition}

\begin{proof}
Because $F_{\varepsilon}\left(  V\right)  =P_{\varepsilon}\circ F\left(
V\right)  $ and $P_{\varepsilon}$ is a bounded linear operator independent on
$V$, it is enough to prove the quadratic estimate of $F\left(  V\right)  $,
namely%
\begin{align*}
&  \left\Vert F^{\prime}\left(  V_{0}\right)  V-F^{\prime}\left(  0\right)
V\right\Vert _{C^{\alpha}\left(  A_{\varepsilon}^{\prime},N_{A_{\varepsilon
}^{\prime}/M}\right)  }\\
&  \leq C\left\Vert V_{0}\right\Vert _{C^{1,\alpha}\left(  A_{\varepsilon
}^{\prime},N_{A_{\varepsilon}^{\prime}/M}\right)  }\left\Vert V\right\Vert
_{C^{1,\alpha}\left(  A_{\varepsilon}^{\prime},N_{A_{\varepsilon}^{\prime}%
/M}\right)  \text{.}}%
\end{align*}
For any $p=\varphi\left(  t,z\right)  $ on $A_{\varepsilon}^{\prime}$ , we
choose a normal frame field $\left\{  W_{\alpha}\right\}  _{\alpha=1,2,...,7}$
in its neighborhood $B_{\delta_{0}}\left(  p\right)  $ where $\tau
=\omega^{\alpha}\otimes W_{a}$, with $\delta_{0}=$ the injectivity radius of
$M$. For the point $q:=\exp_{p}V_{0}$ and if $\left\vert V_{0}\right\vert
<\delta_{0}$, then we can assume this neighborhood covers $q$. Let the
submanifold
\[
A_{\varepsilon}\left(  V\right)  =:\left(  \exp V\right)  \left(
A_{\varepsilon}^{\prime}\right)  ,
\]
which is similar to $\left(  \ref{A_e_V}\right)  $, except that the section
$V$ here is in $N_{A_{\varepsilon}^{\prime}/M}$ instead of $\mathbb{S}$, and
the $\exp:N_{A_{\varepsilon}^{\prime}/M}\rightarrow M$ is the exponential map
on $M$. For the family of embeddings of submanifolds%
\[
\psi_{t}:=\exp\left(  V_{0}+tV\right)  :\mathtt{A}_{\varepsilon}\rightarrow
A_{\varepsilon}\left(  V_{0}+tV\right)  \subset M,
\]
we have
\begin{align*}
\psi_{0}\left(  \mathtt{A}_{\varepsilon}\right)   &  =A_{\varepsilon}\left(
V_{0}\right)  ,\\
\left.  \frac{d}{dt}\right\vert _{t=0}\psi_{t}\left(  q\right)   &  =\left(
d\exp_{^{p}}V_{0}\right)  V\left(  p\right)  :=V_{1}\left(  q\right)
\end{align*}
for any $p\in A_{\varepsilon}^{\prime}=A_{\varepsilon}\left(  0\right)  $.
Arguing as in Theorem \ref{McLean-thm}, from $\nabla\tau=0$ we have $\nabla
W_{\alpha}\left(  q\right)  =d\omega^{\alpha}\left(  q\right)  =0$. We have%
\begin{align}
&  F^{\prime}\left(  V_{0}\right)  V\left(  p\right)  \nonumber\\
&  =\left.  \frac{d}{dt}\right\vert _{t=0}F\left(  V_{0}+tV\right)  \left(
p\right)  \nonumber\\
&  =\left.  \frac{d}{dt}\right\vert _{t=0}\left[  \exp\left(  V_{0}+tV\right)
^{\ast}\omega^{\alpha}\left(  p\right)  \otimes T_{V_{0}+tV}W_{\alpha}\left(
\exp_{p}\left(  V_{0}+tV\right)  \right)  \right]  \nonumber\\
&  =\left.  \frac{d}{dt}\right\vert _{t=0}\left(  \exp\left(  V_{0}+tV\right)
^{\ast}\omega^{\alpha}\right)  \otimes W_{\alpha}\left(  p\right)  \nonumber\\
&  +\left(  \exp V_{0}\right)  ^{\ast}\omega^{\alpha}\left(  p\right)
\otimes\left.  \frac{d}{dt}\right\vert _{t=0}T_{V_{0}+tV}W_{\alpha}\left(
\exp_{p}\left(  V_{0}+tV\right)  \right)  ,\label{DF-2-terms}%
\end{align}
where in the last equality we have used $T_{V_{0}+tV}W_{\alpha}\left(
\exp_{p}\left(  V_{0}+tV\right)  \right)  =W_{\alpha}\left(  p\right)  $ by
the parallel property of $W_{\alpha}$.
For the first term of $\left(  \ref{DF-2-terms}\right)  $, by the property of
the exponential map, namely $d\exp_{q}\left(  0\right)  =id_{T_{q}M\text{ }}%
,$we have
\[
\exp_{p}\left(  V_{0}+tV\right)  =\exp_{q}\left(  tV_{1}+t\left\vert
V\right\vert \beta\left(  V_{0,}V,t\right)  \right)  \circ\exp_{p}V_{0}%
\]
where $V_{1}=\left(  d\exp_{A_{\varepsilon}^{\prime}}V_{0}\right)  V$ is a
vector field on $A_{\varepsilon}\left(  V_{0}\right)  $, $\beta\left(
V_{0,}V,t\right)  $ and $\beta^{\prime}\left(  V_{0,}V,t\right)  $ are
uniformly bounded, and $\lim_{t\rightarrow0}\beta\left(  V_{0,}V,t\right)
=0$. Therefore
\begin{align*}
&  \left.  \frac{d}{dt}\right\vert _{t=0}\exp\left(  V_{0}+tV\right)  ^{\ast
}\omega^{\alpha}\left(  p\right)  \\
&  =\left.  \frac{d}{dt}\right\vert _{t=0}\left(  \exp_{A_{\varepsilon
}^{\prime}}V_{0}\right)  ^{\ast}\exp_{A_{\varepsilon}\left(  V_{0}\right)
}\left(  tV_{1}+t\left\vert V\right\vert \beta\left(  V_{0,}V,t\right)
\right)  ^{\ast}\omega^{\alpha}\\
&  =\left(  \exp_{p}V_{0}\right)  ^{\ast}L_{V_{1}}\omega^{\alpha}\\
&  =\left(  \exp_{p}V_{0}\right)  ^{\ast}\left(  d\left(  i_{V_{1}}%
\omega^{\alpha}\right)  +i_{V_{1}}d\omega^{\alpha}\right)  ,
\end{align*}
where in the third equality we have used that $\lim_{t\rightarrow0}%
\beta\left(  V_{0,}V,t\right)  =0$. For the second term of $\left(
\ref{DF-2-terms}\right)  $, by Gauss-Bonnet formula, which compares the
parallel transports along different paths by curvature integration, we have
\begin{align*}
&  T_{V_{0}+tV}W_{\alpha}\left(  \exp_{p}\left(  V_{0}+tV\right)  \right)  \\
&  =T_{V_{0}}\circ T_{tV_{1}+t\left\vert V\right\vert \beta\left(
V_{0,}V,t\right)  }W_{\alpha}\left(  \exp_{q}\left(  tV_{1}+t\left\vert
V\right\vert \beta\left(  V_{0,}V,t\right)  \right)  \right)  \\
&  +\int_{\Delta\left(  V_{0,}tV\right)  }R\left(  \sigma\right)  d\sigma
W_{\alpha}\left(  p\right)  ,\text{ \ }%
\end{align*}
where $R\left(  \sigma\right)  $ is the Ricci curvature tensor, $\Delta\left(
V_{0,}tV\right)  $ is the $2$-dimensional geodesic triangle with the vertices
$p,\exp_{p}V_{0},$ and $\exp_{p}\left(  V_{0}+tV\right)  $, and $\sigma$ is
the area element. Therefore%
\begin{align*}
&  \left.  \frac{d}{dt}\right\vert _{t=0}T_{V_{0}+tV}W_{\alpha}\left(
\exp_{p}\left(  V_{0}+tV\right)  \right)  \\
&  =T_{V_{0}}\circ\left.  \frac{d}{dt}\right\vert _{t=0}T_{tV_{1}}W_{\alpha
}\left(  \exp_{q}tV_{1}\right)  +\left.  \frac{d}{dt}\right\vert _{t=0}%
\int_{\Delta\left(  V_{0,}tV\right)  }R\left(  \sigma\right)  d\sigma
W_{\alpha}\left(  p\right)  ,
\end{align*}
where in the first term we have dropped the higher order term $t\left\vert
V\right\vert \beta\left(  V_{0,}V,t\right)  $ since it is irrelevant for
derivative at $0$. If we denote
\[
B\left(  V_{0},V\right)  :=\left.  \frac{d}{dt}\right\vert _{t=0}\int
_{\Delta\left(  V_{0,}tV\right)  }R\left(  \sigma\right)  d\sigma,
\]
by the area formula of $\Delta\left(  V_{0,}tV\right)$ it can be shown that%
\[
\left\vert B\left(  V_{0},V\right)  W_{\alpha}\left(  p\right)  \right\vert
\leq C_{5}\left\vert V_{0}\right\vert \left\vert V\right\vert \left\vert
W_{\alpha}\right\vert \left(  p\right)  ,
\]
where $C_{5}$ is a constant only depending on $\left(  M,g\right)  $. Plugging
these in $\left(  \ref{DF-2-terms}\right)  $, we have
\begin{align}
&  F^{\prime}\left(  V_{0}\right)  |_{A_{\varepsilon\left(  0\right)  }%
}V\left(  p\right)  \nonumber\\
&  =\left(  \exp V_{0}\right)  ^{\ast}L_{V_{1}}\omega^{\alpha}\otimes
W_{\alpha}\left(  p\right)  +\left(  \exp V_{0}\right)  ^{\ast}\omega^{\alpha
}\otimes T_{V_{0}}\nabla_{V_{1}}W_{\alpha}\left(  q\right)  \nonumber\\
&  \text{ \ \ }+B\left(  V_{0},V\right)  W_{\alpha}\left(  p\right)
\nonumber\\
&  =\left(  \left(  \exp V_{0}\right)  ^{\ast}\otimes T_{V_{0}}\right)
\circ\left[  L_{V_{1}}\omega^{\alpha}\otimes W_{\alpha}\left(  q\right)
+\omega^{\alpha}\left(  q\right)  \otimes\nabla_{V_{1}}W_{\alpha}\left(
q\right)  \right]  \nonumber\\
&  \text{ \ \ }+B\left(  V_{0},V\right)  W_{\alpha}\left(  p\right)
\nonumber\\
&  =\left(  \left(  \exp V_{0}\right)  ^{\ast}\otimes T_{V_{0}}\right)
\circ\left[  \left(  d\left(  i_{V_{1}}\omega^{\alpha}\right)  +i_{V_{1}%
}d\omega^{\alpha}\right)  \otimes W_{\alpha}\left(  q\right)  +\omega^{\alpha
}\left(  q\right)  \otimes\nabla_{V_{1}}W_{\alpha}\left(  q\right)  \right]
\nonumber\\
&  \text{ \ \ }+B\left(  V_{0},V\right)  W_{\alpha}\left(  p\right)
.\label{DFV0_triangle}%
\end{align}
Here we have used that $T_{V_{0}}W_{\alpha}\left(  q\right)  =W_{\alpha
}\left(  p\right)  $. The $\left(  \ref{DFV0_triangle}\right)  $ can be
rewritten as%
\begin{align*}
F^{\prime}\left(  V_{0}\right)  |_{A_{\varepsilon\left(  0\right)  }}V\left(
p\right)   &  =\left(  \left(  \exp V_{0}\right)  ^{\ast}\otimes T_{V_{0}%
}\right)  \circ\left[  F^{\prime}\left(  0\right)  |_{A_{\varepsilon}\left(
V_{0}\right)  }V_{1}\left(  q\right)  \right]  \\
&  +B\left(  V_{0},V\right)  W_{\alpha}\left(  p\right)  ,
\end{align*}
which means the derivative $F^{\prime}\left(  V_{0}\right)  $ on
$A_{\varepsilon}\left(  0\right)  $ can be expressed by the derivative
$F^{\prime}\left(  0\right)  $ on $A_{\varepsilon}\left(  V_{0}\right)  $ via
the transform $\left(  \exp V_{0}\right)  ^{\ast}\otimes T_{V_{0}}$, up to the
curvature term $B\left(  V_{0},V\right)  W_{\alpha}\left(  p\right)  $. If
$V_{0}=0$, then $q=p$ and $V_{1}=V$, together with
\[
i_{V}d\omega^{\alpha}\left(  p\right)  =\nabla_{V}W_{\alpha}\left(  p\right)
=B\left(  0,V\right)  =0,
\]
the formula $\left(  \ref{DFV0_triangle}\right)  $ is simplified as
\[
F^{\prime}\left(  0\right)  |_{A_{\varepsilon}^{\prime}}V\left(  p\right)
=d\left(  i_{V}\omega^{\alpha}\right)  \otimes W_{\alpha}\left(  p\right)  .
\]
Therefore
\begin{align}
&  F^{\prime}\left(  V_{0}\right)  V\left(  p\right)  -F^{\prime}\left(
0\right)  V\left(  p\right)  \nonumber\\
&  =\left[  \left(  \exp V_{0}\right)  ^{\ast}d\left(  i_{V_{1}}\omega
^{\alpha}\right)  -d\left(  i_{V}\omega^{\alpha}\right)  \right]
|_{A_{\varepsilon\left(  0\right)  }}\otimes W_{\alpha}\left(  p\right)
+\label{DF-difference-1}\\
&  \left(  \left(  \exp V_{0}\right)  ^{\ast}\otimes T_{V_{0}}\right)
\circ\left[  i_{V_{1}}d\omega^{\alpha}\otimes W_{\alpha}+\omega^{\alpha
}\otimes\nabla_{V_{1}}W_{\alpha}\right]  \left(  p\right)  +B\left(
V_{0},V\right)  W_{\alpha}\left(  p\right)  \label{DF-difference-2}%
\end{align}
For the second row $\left(  \ref{DF-difference-2}\right)  $, it is a $0$-th
order linear differential operator on $V\in C^{1,\alpha}\left(  \Gamma\left(
N_{A_{\varepsilon}^{\prime}/M}\right)  \right)  $, where
\[
V_{1}=\left(  d\exp V_{0}\right)  |_{A_{\varepsilon}\left(  0\right)
}V=E\left(  V_{0}\right)  V.
\]
We notice that for the linear operator%
\[
H:C^{1,\alpha}\left(  \Gamma\left(  N_{A_{\varepsilon}^{\prime}/M}\right)
\right)  \rightarrow C^{\alpha}\left(  \Gamma\left(  Hom\left(
TM|_{A_{\varepsilon}^{\prime}},TM|_{A_{\varepsilon}\left(  V_{0}\right)
}\right)  \right)  \right)
\]
that sends%
\begin{align*}
V_{0}  &  \rightarrow\left(  \left(  \exp V_{0}\right)  ^{\ast}\otimes
T_{V_{0}}\right)  \circ\left[  i_{E\left(  V_{0}\right)  \left(  \cdot\right)
}d\omega^{\alpha}\otimes W_{\alpha}+\omega^{\alpha}\otimes\nabla_{E\left(
V_{0}\right)  \left(  \cdot\right)  }W_{\alpha}\right] \\
&  +B\left(  V_{0},\cdot\right)  W_{\alpha},
\end{align*}
$H$ is a bounded operator since the terms $\exp V_{0},\left(  \exp
V_{0}\right)  ^{\ast},T_{V_{0}},E\left(  V_{0}\right)  $ and $B\left(
V_{0},\cdot\right)  $ as elements in $C^{\alpha}$ are differentiable for
$V_{0}\in$ $C^{1,\alpha}\left(  \Gamma\left(  N_{A_{\varepsilon}^{\prime}%
/M}\right)  \right)  $ with bounded Frechet derivatives. The bound of the
derivatives is uniform on $\varepsilon$ so $\left\Vert H\right\Vert $ is
uniformly bounded. Since $H\left(  V_{0}\right)  =0$ when $V_{0}=0$, we have
\[
\left\Vert H\left(  V_{0,}V\right)  \right\Vert _{C^{\alpha}}\leq
C_{8}\left\Vert V_{0}\right\Vert _{C^{1,\alpha}\left(  A_{\varepsilon}%
^{\prime}\right)  }\left\Vert V\right\Vert _{C^{1,\alpha}\left(
A_{\varepsilon}^{\prime}\right)  }.
\]
For the first row $\left(  \ref{DF-difference-1}\right)  $, $\left[  \left(
\exp V_{0}\right)  ^{\ast}d\left(  i_{V_{1}}\omega^{\alpha}\right)  -d\left(
i_{V}\omega^{\alpha}\right)  \right]  $ is a first order linear differential
operator on $V$ $\left(  V_{1}=\left(  d\exp V_{0}\right)  |_{A_{\varepsilon
}\left(  0\right)  }V\right)  $, and only involves the covariant derivatives
of $V$.\ In the next lemma we will show
\[
\left\Vert \left(  \exp V_{0}\right)  ^{\ast}d\left(  i_{V_{1}}\omega^{\alpha
}\right)  -d\left(  i_{V}\omega^{\alpha}\right)  \right\Vert _{C^{\alpha
}\left(  A_{\varepsilon}^{\prime}\right)  }\leq C_{7}\left\Vert V_{0}%
\right\Vert _{C^{1,\alpha}\left(  A_{\varepsilon}^{\prime}\right)  }\left\Vert
V\right\Vert _{C^{1,\alpha}\left(  A_{\varepsilon}^{\prime}\right)  }.
\]
Combining the two rows we get
\begin{align*}
&  \left\Vert F^{\prime}\left(  V_{0}\right)  V-F^{\prime}\left(  0\right)
V\right\Vert _{C^{\alpha}\left(  A_{\varepsilon}^{\prime},N_{A_{\varepsilon
}^{\prime}/M}\right)  }\\
&  \leq C\left\Vert V_{0}\right\Vert _{C^{1,\alpha}\left(  A_{\varepsilon
}^{\prime},N_{A_{\varepsilon}^{\prime}/M}\right)  }\left\Vert V\right\Vert
_{C^{1,\alpha}\left(  A_{\varepsilon}^{\prime},N_{A_{\varepsilon}^{\prime}%
/M}\right)  .}%
\end{align*}
\end{proof}

\begin{lemma}
For any $V_{0},V\in\Gamma\left(  N_{A_{\varepsilon}^{\prime}/M}\right)  $ with
$\left\Vert V_{0}\right\Vert _{C^{1,\alpha}\left(  A_{\varepsilon}^{\prime
}\right)  }\leq\delta_{0}$, we have%
\[
\left\Vert \left(  \exp V_{0}\right)  ^{\ast}d\left(  i_{V_{1}}\omega^{\alpha
}\right)  -d\left(  i_{V}\omega^{\alpha}\right)  \right\Vert _{C^{\alpha
}\left(  A_{\varepsilon}^{\prime}\right)  }\leq C_{7}\left\Vert V_{0}%
\right\Vert _{C^{1,\alpha}\left(  A_{\varepsilon}^{\prime}\right)  }\left\Vert
V\right\Vert _{C^{1,\alpha}\left(  A_{\varepsilon}^{\prime}\right)  },
\]
where the constants $\delta_{0}$ and $C_{7}$ are independent on $\varepsilon$.
\end{lemma}

\begin{proof}
Consider the $3$ linear maps from $T_{p}M$ to $T_{q}M$, where $q=\exp_{p}%
V_{0}:$%
\begin{align*}
&
\begin{array}
[c]{c}%
Pal_{V_{0}}:T_{p}M\rightarrow T_{q}M,\text{ parallel transport along the
geodesic }\\
\text{ \ \ \ \ \ \ \ \ \ \ \ }\exp_{p}tV_{0}\text{ for }0\leq t\leq1,
\end{array}
\\
&
\begin{array}
[c]{c}%
d\left(  \exp V_{0}\right)  :T_{p}M\rightarrow T_{q}M,\text{ tangent map of
the diffeomorphism \ \ }\\
\text{ \ \ \ \ \ \ \ \ \ \ \ \ \ }\exp V_{0}:A_{\varepsilon}^{\prime
}\rightarrow A_{\varepsilon}\left(  V_{0}\right)  ,
\end{array}
\\
&
\begin{array}
[c]{c}%
d\left(  \exp_{p}\right)  \left(  V_{0}\right)  :T_{V_{0}}\left(
T_{p}M\right)  \rightarrow T_{q}M,\text{ tangent map of the exponential map}\\
\text{ \ \ \ \ \ \ \ \ \ \ \ \ \ \ \ \ \ \ \ \ \ \ }\exp_{p}:T_{p}M\rightarrow
M,
\end{array}
\end{align*}
where in the last one we identify $T_{V_{0}}\left(  T_{p}M\right)  \simeq
T_{p}M$. Note that for diffeomorphism $d\left(  \exp V_{0}\right)  $ we need a
vector field $V_{0}$ on $A_{\varepsilon}^{\prime}$ but for $\ d\left(
\exp_{p}\right)  \left(  V_{0}\right)  $ we only need $V_{0}\in T_{p}M$, so
they are essentially different maps. But the $3$ maps are very close to each
other when $V_{0}$ is small. More precisely, for any $V_{0},V\in\Gamma\left(
N_{A_{\varepsilon}^{\prime}/M}\right)  $ with $\left\Vert V_{0}\right\Vert
_{C^{1,\alpha}\left(  A_{\varepsilon}^{\prime}\right)  }\leq\delta_{0}$, we
have comparision of the maps as the following:
\begin{align*}
d\left(  \exp V_{0}\right)  V &  =Pal_{V_{0}}V+h_{1}\left(  V_{0}\right)  V,\\
d\left(  \exp_{p}\right)  \left(  V_{0}\right)  V &  =d\left(  \exp
V_{0}\right)  V+h_{2}\left(  V_{0}\right)  V,
\end{align*}
where for each $i=1,2,$ the error term%
\[
h_{i}\left(  \cdot\right)  :C^{1,\alpha}\left(  \Gamma\left(
N_{A_{\varepsilon}^{\prime}/M}\right)  \right)  \rightarrow C^{\alpha}\left(
\Gamma\left(  Hom\left(  TM|_{A_{\varepsilon}^{\prime}},TM|_{A_{\varepsilon
}\left(  V_{0}\right)  }\right)  \right)  \right)  .
\]
is differentiable with respect to the variable $V_{0}$, and
\[
h_{i}\left(  V_{0}\right)  =0\text{ for }V_{0}=0.
\]
This is because for any fixed $p$, $Pal_{V_{0}\left(  p\right)  }$ and
$d\left(  \exp_{p}\right)  \left(  V_{0}\right)  $ smoothly depend on $V_{0}$,
and for the compact family $p\in A_{\varepsilon}^{\prime}$, $Pal_{V_{0}}$ and
$d\left(  \exp_{p}\right)  \left(  V_{0}\right)  $ inherits the smooth
dependence on $V_{0}\in C^{1,\alpha}\left(  \Gamma\left(  N_{A_{\varepsilon
}^{\prime}/M}\right)  \right)  $. For $d\left(  \exp V_{0}\right)  $, since%
\[
d\left(  \exp V_{0}\right)  \left(  x\right)  =F_{1}\left(  x,V_{0}\right)
dx+F_{2}\left(  x,V_{0}\right)  dV_{0}\left(  x\right)
\]
where $F_{1}\left(  x,y\right)  =\frac{\partial}{\partial x}\exp_{x}y$ and
$F_{2}\left(  x,y\right)  =\frac{\partial}{\partial y}\exp_{x}y$ are bounded
on $A_{\varepsilon}$, we see $d\left(  \exp V_{0}\right)  \in C^{\alpha}$
smoothly depends on $V_{0}\in C^{1,\alpha}\left(  \Gamma\left(
N_{A_{\varepsilon}^{\prime}/M}\right)  \right)  $ with bounded Frechet
derivative. This especially implies that
\begin{equation}
\left\Vert h_{i}\left(  V_{0}\right)  V\right\Vert _{C^{a}}\leq C_{8}%
\left\Vert V_{0}\right\Vert _{C^{1,\alpha}\left(  A_{\varepsilon}^{\prime
}\right)  }\left\Vert V\right\Vert _{C^{1,\alpha}\left(  A_{\varepsilon
}^{\prime}\right)  }.\label{hi}%
\end{equation}
for $i=1,2$. We have
\begin{align*}
V_{1} &  =d\left(  \exp V_{0}\right)  V+h_{2}\left(  V_{0}\right)  V,\\
d\left(  \exp V_{0}\right)  W_{\alpha}\left(  p\right)   &  =W_{a}\left(
q\right)  +h_{1}\left(  V_{0}\right)  W_{\alpha}\left(  p\right)  ,\\
\left(  \exp V_{0}\right)  ^{\ast}\omega^{\alpha}\left(  q\right)   &
=\omega^{\alpha}\left(  p\right)  +h_{3}\left(  V_{0}\right)  \omega^{\alpha
}\left(  q\right)
\end{align*}
using the parallel property of $W_{a}$ and $\tau$ (hence $\omega^{\alpha}$).
(Here $h_{3}\left(  V_{0}\right)  $ depends on $V_{0}\in C^{1,\alpha}\left(
\Gamma\left(  N_{A_{\varepsilon}^{\prime}/M}\right)  \right)  $ smoothly by
similar reason of $h_{1}$ and $h_{2}$. Since $\omega^{\alpha}$ is a $3$-form,
$h_{3}\left(  V_{0}\right)  $ can have terms of cubic power order of $V_{0}$).
Therefore%
\begin{align*}
&  \left(  \exp V_{0}\right)  ^{\ast}d\left(  i_{V_{1}}\omega^{\alpha}\right)
-d\left(  i_{V}\omega^{\alpha}\right)  \\
&  =d\left[  \left(  \exp V_{0}\right)  ^{\ast}\left(  i_{d\left(  \exp
V_{0}\right)  V+h_{2}\left(  V_{0}\right)  V}\omega^{\alpha}\right)
-i_{V}\omega^{\alpha}\right]  \\
&  =d\left[  \left(  \exp V_{0}\right)  ^{\ast}i_{d\left(  \exp V_{0}\right)
V}\omega^{\alpha}+\left(  \exp V_{0}\right)  ^{\ast}i_{h_{2}\left(
V_{0}\right)  V}\omega^{\alpha}-i_{V}\omega^{\alpha}\right]  \\
&  =d\left[  i_{V}\left(  \exp V_{0}\right)  ^{\ast}\omega^{\alpha}+\left(
\exp V_{0}\right)  ^{\ast}i_{h_{2}\left(  V_{0}\right)  V}\omega^{\alpha
}-i_{V}\omega^{\alpha}\right]  \\
&  =d\left[  i_{V}\left(  \omega^{\alpha}+h_{3}\left(  V_{0}\right)
\omega^{\alpha}\right)  +\left(  \exp V_{0}\right)  ^{\ast}i_{h_{2}\left(
V_{0}\right)  V}\omega^{\alpha}-i_{V}\omega^{\alpha}\right]  \\
&  =d\left[  i_{V}\left(  h_{3}\left(  V_{0}\right)  \omega^{\alpha}\left(
q\right)  \right)  +\left(  \exp V_{0}\right)  ^{\ast}i_{h_{2}\left(
V_{0}\right)  V}\omega^{\alpha}\left(  q\right)  \right]  .
\end{align*}
Using the property $\left(  \ref{hi}\right)  $ of $h_{i}\left(
i=1,2,3\right)  $, we observe that each term in the above last identity in
$C^{\alpha}$ norm smoothly depends on $V_{0}\in C^{1,\alpha}\left(
\Gamma\left(  N_{A_{\varepsilon}^{\prime}/M}\right)  \right)  $ and linearly
on $V\in C^{1,\alpha}\left(  \Gamma\left(  N_{A_{\varepsilon}^{\prime}%
/M}\right)  \right)  $. Also when $V_{0}=0$, we have $\left(  \exp
V_{0}\right)  ^{\ast}d\left(  i_{V_{1}}\omega^{\alpha}\right)  -d\left(
i_{V}\omega^{\alpha}\right)  =0$. Therefore
\[
\left\Vert \left(  \exp V_{0}\right)  ^{\ast}d\left(  i_{V_{1}}\omega^{\alpha
}\right)  -d\left(  i_{V}\omega^{\alpha}\right)  \right\Vert _{C^{\alpha
}\left(  A_{\varepsilon}^{\prime}\right)  }\leq C\left\Vert V_{0}\right\Vert
_{C^{1,\alpha}\left(  A_{\varepsilon}^{\prime}\right)  }\left\Vert
V\right\Vert _{C^{1,\alpha}\left(  A_{\varepsilon}^{\prime}\right)  }.
\]
\end{proof}

\begin{remark}
\label{Schauder-over-Lp}We also have some point estimates tied to the feature
that $\tau$ is a $3$-form. It is well known that for any smooth embeddings
$\varphi:A_{\varepsilon}\rightarrow M$ and smooth sections $V_{0}$ of
$TM|_{A_{\varepsilon}^{\prime}}$ with $\left\vert V_{0}\right\vert $ smaller
than the injectivity radius of $M$,
\[
\left\vert d\exp_{\varphi}V_{0}\right\vert \leq C_{6}\left(  \left\vert
d\varphi\right\vert +\left\vert \nabla V_{0}\right\vert \right)  ,
\]
where $C_{6}$ is a constant only depending on $\left(  M,g\right)  $. Hence
for the map $\exp V_{0}:TM|_{A_{\varepsilon}^{\prime}}\rightarrow M$ and any
section $V$ of $TM|_{A_{\varepsilon}^{\prime}}$, at $p\in A_{\varepsilon
}^{\prime}$ we have%
\[
\left\vert \left(  d\exp V_{0}\right)  \left(  p\right)  V\right\vert \leq
C_{6}\left(  \left\vert d\varphi\right\vert +\left\vert \nabla V_{0}%
\right\vert \right)  \left\vert V\left(  p\right)  \right\vert .
\]
So for the first term $\left(  \ref{DF-difference-1}\right)  $, we have
\begin{align*}
&  \left\vert \left[  \left(  \exp V_{0}\right)  ^{\ast}d\left(  i_{V_{1}%
}\omega^{\alpha}\right)  -d\left(  i_{V}\omega^{\alpha}\right)  \right]
\otimes W_{\alpha}\left(  p\right)  \right\vert \\
&  \leq C_{7}\left(  \left\vert d\varphi\right\vert +\left\vert \nabla
V_{0}\right\vert \right)  ^{3}\left(  \left\vert \nabla V\right\vert
\left\vert V_{0}\right\vert +\left\vert V\right\vert \left\vert \nabla
V_{0}\right\vert +\left\vert V\right\vert \left\vert V_{0}\right\vert \right)
\left(  p\right)  ,
\end{align*}
where the power $3$ is because $\omega^{\alpha}$ is a $3$-form. For the next
row $\left(  \ref{DF-difference-2}\right)  $, it is a $0$-th order linear
differential operator on $V$, where $V_{1}=\left(  d\exp V_{0}\right)
|_{A_{\varepsilon}\left(  0\right)  }V$. By the $C^{2}$ smoothness of
$\tau=\omega^{\alpha}\otimes W_{\alpha}$ and
\[
i_{W}d\omega^{\alpha}\left(  p\right)  =\nabla_{W}W_{\alpha}\left(  p\right)
=0
\]
for any $W\in\Gamma\left(  N_{A_{\varepsilon}^{\prime}/M}\right)  $, at
$q=\exp_{p}V_{0}$ we have
\[
\left\vert i_{V_{1}}d\omega^{\alpha}\left(  q\right)  \right\vert \leq
C_{4}\left\vert V_{0}\right\vert \left\vert V\right\vert ,\text{
\ \ }\left\vert \nabla_{V_{1}}W_{\alpha}\left(  q\right)  \right\vert \leq
C_{4}\left\vert V_{0}\right\vert \left\vert V\right\vert .\text{ \ }%
\]
Hence
\begin{align*}
&  \left\vert \left(  \left(  \exp V_{0}\right)  ^{\ast}\otimes T_{V_{0}%
}\right)  \circ\left[  i_{V_{1}}d\omega^{\alpha}\otimes W_{\alpha}%
+\omega^{\alpha}\otimes\nabla_{V_{1}}W_{\alpha}\right]  \left(  p\right)
+B\left(  V_{0},V\right)  W_{\alpha}\left(  p\right)  \right\vert \\
&  \leq C_{6}^{3}\left(  \left\vert d\varphi\right\vert +\left\vert \nabla
V_{0}\right\vert \right)  ^{3}\cdot2C_{4}\left\vert V_{0}\right\vert
\left\vert V\right\vert \left(  p\right)  +C_{5}\left\vert V_{0}\right\vert
\left\vert V\right\vert \left(  p\right)  .
\end{align*}
where the power $3$ is because $\omega^{\alpha}$ is a $3$-form. Combining all,
we have
\begin{align*}
&  \left\vert F^{\prime}\left(  V_{0}\right)  V-F^{\prime}\left(  0\right)
V\right\vert \left(  p\right) \\
&  \leq\left[  C_{7}\left(  \left\vert d\varphi\right\vert +\left\vert \nabla
V_{0}\right\vert \right)  ^{3}\left(  \left\vert V_{0}\right\vert \left\vert
V\right\vert +\left\vert \nabla V_{0}\right\vert \left\vert V\right\vert
+\left\vert V_{0}\right\vert \left\vert \nabla V\right\vert \right)
+C_{5}\left\vert V_{0}\right\vert \left\vert V\right\vert \right]  \left(
p\right)  .
\end{align*}
Because of the cubic power terms, it is difficult to get similar quadratic
estimate in the $L^{p}$ setting, namely
\begin{align*}
&  \left\Vert F^{\prime}\left(  V_{0}\right)  V-F^{\prime}\left(  0\right)
V\right\Vert _{L^{p}\left(  A_{\varepsilon}^{\prime},N_{A_{\varepsilon
}^{\prime}/M}\right)  }\\
&  \leq C\left\Vert V_{0}\right\Vert _{W^{1,p}\left(  A_{\varepsilon}^{\prime
},N_{A_{\varepsilon}^{\prime}/M}\right)  }\left\Vert V\right\Vert
_{W^{1,p}\left(  A_{\varepsilon}^{\prime},N_{A_{\varepsilon}^{\prime}%
/M}\right)  .}%
\end{align*}
This is one of the key reasons that we choose the Schauder setting for
implicit function theorem, where the right inverse bound of $F^{\prime}\left(
0\right)  $ is much harder to obtain than in the $L^{p}$ setting. In contrast,
the Cauchy-Riemann operator of $J$-holomorphic curves is more linear in the
$L^{p}$ setting: it has the quadratic estimate (see Proposition 3.5.3 in
\cite{MS})%
\[
\left\Vert F^{\prime}\left(  V_{0}\right)  V-F^{\prime}\left(  0\right)
V\right\Vert _{L^{p}\left(  \Sigma\right)  }\leq C\left\Vert V_{0}\right\Vert
_{W^{1,p}\left(  \Sigma\right)  }\left\Vert V\right\Vert _{W^{1,p}\left(
\Sigma\right)  .}%
\]

\end{remark}

\subsection{Perturbation Argument}

To find the zeros of $F_{\varepsilon}$, we are going to apply the following
quantitative version of the implicit function theorem (c.f. Theorem 15.6
\cite{DE} or Proposition A3.4 in \cite{MS}).

\begin{theorem}
\label{implicit}Let $\left(  X,\left\vert \cdot\right\vert _{X}\right)  $ and
$\left(  Y,\left\vert \cdot\right\vert _{Y}\right)  $be Banach spaces and
$F:B_{r}\left(  0\right)  \subset X\rightarrow Y$ a $C^{1}$-map, such that

\begin{enumerate}
\item $\left(  DF\left(  0\right)  \right)  ^{-1}$ is a bounded linear
operator with $\left\Vert \left(  DF\left(  0\right)  \right)  ^{-1}F\left(
0\right)  \right\Vert \leq A$ and $\left\Vert \left(  DF\left(  0\right)
\right)  ^{-1}\right\Vert \leq B;$

\item $\left\Vert DF\left(  x\right)  -DF\left(  0\right)  \right\Vert
\leq\kappa\left\vert x\right\vert _{X}$ \ for all $x\in B_{r}\left(  0\right)
;$

\item $2\kappa AB<1$ and $2A<r.$
\end{enumerate}

Then $F$ has a unique zero in $B_{2A}\left(  0\right)  .$
\end{theorem}

To apply the above theorem, we define the map $\mathcal{\ }$
\[
\tilde{F}_{\varepsilon}:=\varepsilon^{-\left(  \frac{3}{p}+2\alpha\right)
}F_{\varepsilon}:C_{-}^{1,\alpha}\left(  \mathtt{A}_{\varepsilon}^{\prime
},N_{A_{\varepsilon}^{\prime}/M}\right)  \longrightarrow C^{\alpha}\left(
\mathtt{A}_{\varepsilon}^{\prime},N_{A_{\varepsilon}^{\prime}/M}\right)  .
\]
Then Proposition \ref{DFe-inverse-bd} implies that
\[
\left\Vert \left(  D\tilde{F}_{\varepsilon}\left(  0\right)  \right)
^{-1}\right\Vert \leq\varepsilon^{\frac{3}{p}+2\alpha}\cdot C\varepsilon
^{-\left(  \frac{3}{p}+2\alpha\right)  }=C.
\]

By Proposition \ref{quadratic-estimate}, for $\left\Vert V_{0}\right\Vert
_{C_{-}^{1,\alpha}\left(  \mathtt{A}_{\varepsilon}^{\prime},N_{A_{\varepsilon
}^{\prime}/M}\right)  }\leq\delta_{0}$ and any $V\in C_{-}^{1,\alpha}\left(
\mathtt{A}_{\varepsilon}^{\prime},N_{A_{\varepsilon}^{\prime}/M}\right)  $,
\begin{align*}
&  \left\Vert \left(  D\tilde{F}_{\varepsilon}\left(  V_{0}\right)
-D\tilde{F}_{\varepsilon}\left(  0\right)  \right)  V\right\Vert _{C^{\alpha
}\left(  \mathtt{A}_{\varepsilon}^{\prime},N_{A_{\varepsilon}^{\prime}%
/M}\right)  }\\
&  \leq C\varepsilon^{-\left(  \frac{3}{p}+2\alpha\right)  }\left\Vert
V_{0}\right\Vert _{C_{-}^{1,\alpha}\left(  \mathtt{A}_{\varepsilon}^{\prime
},N_{A_{\varepsilon}^{\prime}/M}\right)  }\left\Vert V\right\Vert
_{C_{-}^{1,\alpha}\left(  \mathtt{A}_{\varepsilon}^{\prime},N_{A_{\varepsilon
}^{\prime}/M}\right)  }\text{.}%
\end{align*}

For $F\left(  V\right)  =T_{V}\circ\left(  \exp V\right)  ^{\ast}\tau$, fixing
a volume form on $\mathtt{A}_{\varepsilon}^{\prime}$ we can regard $F\left(
V\right)  $ as a section of $TM|_{\mathtt{A}_{\varepsilon}^{\prime}}$. We have%
\begin{align}
\left\Vert F\left(  0\right)  \right\Vert _{_{C^{\alpha}\left(  \mathtt{A}%
_{\varepsilon}^{\prime},TM\right)  }}  &  =\left\Vert \varphi^{\ast}%
\tau\right\Vert _{C^{\alpha}\left(  \mathtt{A}_{\varepsilon}^{\prime
},TM\right)  }\nonumber\\
&  =C\left\Vert \left(  \varphi^{\ast}\tau\right)  \left(  0,z\right)
+\int_{0}^{x_{1}}\frac{\partial}{\partial x_{1}}\left(  \varphi^{\ast}%
\tau\right)  \left(  s,z\right)  ds\right\Vert _{C^{\alpha}\left(
\mathtt{A}_{\varepsilon}^{\prime},TM\right)  }\nonumber\\
&  =C\left\Vert \int_{0}^{x_{1}}\frac{\partial}{\partial x_{1}}\varphi^{\ast
}\tau\left(  s,z\right)  ds\right\Vert _{C^{\alpha}\left(  \mathtt{A}%
_{\varepsilon}^{\prime},TM\right)  }\leq C\varepsilon^{1-\alpha},
\label{almost-instanton}%
\end{align}
where $\left(  \varphi^{\ast}\tau\right)  \left(  0,z\right)  =0$ is because
$\varphi\left(  0\times\Sigma\right)  $ is $J_{n}$-holomorphic and then
$\tau|_{T\mathtt{A}_{\varepsilon}^{\prime}|_{\varphi\left(  0\times
\Sigma\right)  }}=0$, and the last inequality is because for%
\[
H\left(  x_{1},z\right)  :=\int_{0}^{x_{1}}\frac{\partial}{\partial x_{1}%
}\varphi^{\ast}\tau\left(  s,z\right)  ds
\]
where $f\left(  s,z\right)  :=\frac{\partial}{\partial x_{1}}\varphi^{\ast
}\tau\left(  s,z\right)  $ is smooth on $\left[  0,\varepsilon\right]
\times\Sigma$, we have
\begin{align*}
\left\Vert H\right\Vert _{C^{0}\left(  \mathtt{A}_{\varepsilon}\right)  }  &
\leq\int_{0}^{\varepsilon}\left\Vert f\right\Vert _{C^{0}\left(
\mathtt{A}_{\varepsilon}\right)  }ds=\left\Vert f\right\Vert _{C^{0}\left(
\mathtt{A}_{\varepsilon}\right)  }\varepsilon,\\
\left[  H\right]  _{\alpha;\mathtt{A}_{\varepsilon}}^{z}  &  \leq\int
_{0}^{\varepsilon}\left[  f\right]  _{\alpha;\mathtt{A}_{\varepsilon}}%
^{z}ds=\left[  f\right]  _{\alpha;\mathtt{A}_{\varepsilon}}^{z}\varepsilon,\\
\left[  H\right]  _{\alpha;\mathtt{A}_{\varepsilon}}^{x_{1}}  &  \leq C\left[
H\right]  _{1;\mathtt{A}_{\varepsilon}}^{x_{1}}\varepsilon^{1-\alpha}%
\leq\left\Vert f\right\Vert _{C^{0}\left(  \mathtt{A}_{\varepsilon}\right)
}\varepsilon^{1-\alpha}\text{.}%
\end{align*}
(Here $\left[  f\right]  _{\alpha;\mathtt{A}_{\varepsilon}}^{z}$, $\left[
f\right]  _{\alpha;\mathtt{A}_{\varepsilon}}^{x_{1}}$ are the Schauder
components of $f$ in the $z$ and $x_{1}$ directions of $\mathtt{A}%
_{\varepsilon}$). Note that $\left(  \ref{almost-instanton}\right)  $ implies
that the tangent space of $\mathtt{A}_{\varepsilon}^{\prime}$ is
$\varepsilon^{1-\alpha}$-close to be associative.

By our construction
\[
F_{\varepsilon}\left(  V\right)  =\ast_{\mathtt{A}_{\varepsilon}^{\prime}%
}\circ\bot_{\mathtt{A}_{\varepsilon}^{\prime}}\circ F\left(  V\right)  ,
\]
where $\ast_{\mathtt{A}_{\varepsilon}^{\prime}}\circ\bot_{\mathtt{A}%
_{\varepsilon}^{\prime}}$ is a bounded linear operator and $T_{V}$ is a
parallel transport, we have
\begin{align*}
\left\Vert \tilde{F}_{\varepsilon}\left(  0\right)  \right\Vert _{C^{\alpha
}\left(  \mathtt{A}_{\varepsilon},N_{A_{\varepsilon}^{\prime}/M}\right)  }  &
\leq C\varepsilon^{-\left(  \frac{3}{p}+2\alpha\right)  }\left\Vert F\left(
0\right)  \right\Vert _{C^{\alpha}\left(  \mathtt{A}_{\varepsilon
},N_{A_{\varepsilon}^{\prime}/M}\right)  }\\
&  \leq C\varepsilon^{-\left(  \frac{3}{p}+2\alpha\right)  }\varepsilon
^{1-\alpha}=C\varepsilon^{1-\left(  \frac{3}{p}+3\alpha\right)  }.
\end{align*}
Then we have
\[
\left\Vert D\tilde{F}_{\varepsilon}\left(  0\right)  ^{-1}\right\Vert
\left\Vert \tilde{F}_{\varepsilon}\left(  0\right)  \right\Vert _{C^{\alpha
}\left(  \mathtt{A}_{\varepsilon},\mathbb{S}\right)  }\varepsilon^{-\left(
\frac{3}{p}+2\alpha\right)  }\leq C\varepsilon^{1-\frac{6}{p}-5\alpha}.
\]
Note we have chosen that $0<\frac{3}{p}+3\alpha<$ $\frac{1}{2}$ (say
$\alpha=1/12$ and $p>12$) in the beginning, so $1-\frac{6}{p}-5\alpha>0$ then
\begin{align*}
&  \left\Vert D\tilde{F}_{\varepsilon}\left(  0\right)  ^{-1}\right\Vert
\left\Vert \tilde{F}_{\varepsilon}\left(  0\right)  \right\Vert _{C^{\alpha
}\left(  \mathtt{A}_{\varepsilon},\mathbb{S}\right)  }\varepsilon^{-\left(
\frac{3}{p}+2\alpha\right)  }\\
&  \leq\,C\varepsilon^{1-\left(  \frac{3}{p}+3\alpha\right)  }\varepsilon
^{-\left(  \frac{3}{p}+2\alpha\right)  }=C\varepsilon^{1-\frac{6}{p}-5\alpha
}\rightarrow0
\end{align*}
as $\varepsilon\rightarrow0$. Theorem \ref{implicit} then implies that there
is a unique $V_{\varepsilon}\in\Gamma\left(  N_{A_{\varepsilon}^{\prime}%
/M}\right)  $ with
\begin{equation}
\left\Vert V_{\varepsilon}\right\Vert _{C^{1,\alpha}\left(  \mathtt{A}%
_{\varepsilon}^{\prime},N_{A_{\varepsilon}^{\prime}/M}\right)  }%
\leq2C\varepsilon^{1-\left(  \frac{3}{p}+3\alpha\right)  }
\label{perturb-order}%
\end{equation}
that solves $\tilde{F}_{\varepsilon}\left(  V_{\varepsilon}\right)  =0$, i.e.
$F_{\varepsilon}\left(  V_{\varepsilon}\right)  =0$. Note that $\left(
\ref{almost-instanton}\right)  $ and $\left(  \ref{perturb-order}\right)  $
together imply that the tangent space of $\mathtt{A}_{\varepsilon}\left(
V_{\varepsilon}\right)  $ is $\varepsilon^{1-\left(  \frac{3}{p}%
+3\alpha\right)  }$-close to be associative. By Proposition
\ref{1Prop alm instanton}, for almost associative submanifolds $\mathtt{A}%
_{\varepsilon}\left(  V_{\varepsilon}\right)  $,
\[
F_{\varepsilon}\left(  V_{\varepsilon}\right)  =0\Leftrightarrow F\left(
V_{\varepsilon}\right)  =0\text{,}%
\]
while the latter means that $\mathtt{A}_{\varepsilon}\left(  V_{\varepsilon
}\right)  $ is associative.

Finally, we obtain our main result:

\begin{theorem}
\label{main}Suppose that $M$ is a $G_{2}$-manifold and $C_{t}$ is a one
parameter family of coassociative submanifolds in $M$. Suppose that the
self-dual two form $\eta=dC_{t}/dt|_{t=0}\in\Omega_{+}^{2}\left(  C\right)  $
is nonvanishing, then it defines an almost complex structure $J$ on $C_{0}$.

For any regular $J$-holomorphic curve $\Sigma$ in $C_{0}$, there is an
instanton $A_{\varepsilon}$ in $M$ which is diffeomorphic to $\left[
0,\varepsilon\right]  \times\Sigma$ and $\partial A_{\varepsilon}\subset
C_{0}\cup C_{\varepsilon}$, for all sufficiently small positive $\varepsilon$.
\end{theorem}

Finally, we expect that any instanton $A$ in $M$ bounding $C_{0}\cup C_{t}$
and with small volume must arise in the above manner. Namely we need to prove
a $\varepsilon$-regularity result for instantons.

\bigskip A few remarks are in order: First, counting such small instantons is
basically a problem in four manifold theory because of Bryant's result
\cite{Bryant} which says that the zero section $C$ in $\Lambda_{+}^{2}\left(
C\right)  $ is always a coassociative submanifold for an incomplete $G_{2}%
$-metric on its neighborhood provided that the bundle $\Lambda_{+}^{2}\left(
C\right)  $ is topologically trivial. Second, when $\eta$ has zeros, the above
Theorem \ref{main} should still hold true. However using the present approach
to prove it would require a good understanding of the Seiberg-Witten theory on
any four manifold with a degenerated symplectic form as in Taubes program. But
at least for regular $J_{\eta}$-holomorphic curves $\Sigma$ in $C_{0}$ that is
\emph{away from the zero locus} of $\eta$, the instantons $\mathtt{A}%
_{\varepsilon}^{\prime}$ in Theorem \ref{main} still exist nearby
$\Sigma\subset M$, for our gluing analysis only involves the local geometry of
$\Sigma$ in $M$. Third, if we do not restrict to instantons of small volume,
then we have to take into account the compactification of the moduli of
instantons which has not been established yet, e.g. \emph{bubbling}%
{\footnotesize \ }phenomenon as for pseudo-holomorphic curves, and gluing of
instantons of big and small volumes similar to \cite{Oh Zhu} in Floer
trajectory case. Nevertheless, one would expect that if the volume of
$\mathtt{A}_{t}$'s are small, then bubbling cannot occur, thus they would
converge to a $J_{\eta}$-holomorphic curve in $C_{0}$.

\section{Appendix: The Exponential-like Map $\widetilde{\exp}$}

We construct the exponential-like map $\widetilde{\exp}:\mathbb{S\rightarrow
}M$ satisfying the properties 1$\sim$3 in Subsection 4.2.\

For any section $V=\left(  u,v\right)  $ of the spinor bundle
$\mathbb{S\rightarrow}A_{\varepsilon}$, since the bundle $\mathbb{S}%
\rightarrow\mathtt{A}_{\varepsilon}$ is a Cartesian product of the spinor
bundle $\mathbb{S}_{\Sigma}\rightarrow\Sigma$ with the interval $\left[
0,\varepsilon\right]  $, we may alternatively regard $V\left(  t,z\right)  $
as a one parameter family of sections of the bundle $\mathbb{S}_{\Sigma
}\rightarrow\Sigma$ for the parameter $t\in\left[  0,\varepsilon\right]  $.
Our goal is to obtain the deformation of $A_{\varepsilon}^{\prime}\subset M$
from $V$. The idea is to deform $\mathtt{A}_{\varepsilon}\subset
\mathtt{C:}=\left[  0,\varepsilon\right]  \times C$ using only the $u$
component, then map the deformed $\mathtt{A}_{\varepsilon}$ to $M$ by
$\varphi$, and at last deform it in $M$ by the $v$ component. The
coassociative boundary condition is preserved under the map $\widetilde{\exp}%
$, since we separate the $u$ and $v$ deformations before and after the map
$\varphi$ respectively. The precise description is in order.

For each fixed $t\in\left[  0,\varepsilon\right]  $, $u\left(  t,z\right)  $
and $v\left(  t,z\right)  $ are sections of the spinor bundle $\mathbb{S}%
_{\Sigma}\rightarrow\Sigma$. Using the real vector bundle isomorphism
\[
\left(  id,f\right)  :N_{\Sigma/C}\oplus\wedge_{\mathbb{C}}^{0,1}\left(
N_{\Sigma/C}\right)  \simeq N_{\Sigma/C}\oplus N_{\mathcal{C}/M}|_{\Sigma},
\]
$u\left(  t,z\right)  $ and $f\left(  v\left(  t,z\right)  \right)  $ are
sections of $N_{\Sigma/C}$ and $N_{\mathcal{C}/M}|_{\Sigma}$ respectively. For
any fixed $z\in\Sigma$, using the $u$ component, we have a one parameter
deformation $\exp_{z}^{C}u\left(  t,z\right)  $ in $C$ for $0\leq
t\leq\varepsilon$, where $\exp^{C}$ is the exponential map associated to the
induced metric of $C$ in $M$. Alternatively, we can view the deformation in
the product space $\mathtt{C:}=\left[  0,\varepsilon\right]  \times C$, such
that the line $\left[  0,\varepsilon\right]  \times\left\{  z\right\}
\subset\mathtt{C}$ is deformed to the curve $\left\{  \left(  t,\exp_{z}%
^{C}u\left(  t,z\right)  \right)  |0\leq t\leq\varepsilon\text{ }\right\}
\subset\mathtt{C}$. Let the curve $\gamma_{u}$ in $M$ be
\[
\gamma_{u}\left(  t,z\right)  :=\varphi\left(  t\text{,}\exp_{z}^{C}u\left(
t,z\right)  \right)  ,\text{ for }0\leq t\leq\varepsilon\text{ and }z\in
\Sigma\text{.}%
\]
It is clear that
\[
\gamma_{u}\left(  t,z\right)  \subset C_{t}:=\varphi\left(  t,C\right)
\]
since $\exp_{z}^{C}u\left(  t,z\right)  \subset C$.

For the $v$ component, for each fixed $t\in\left[  0,\varepsilon\right]  $, we
have $f\left(  v\left(  t,z\right)  \right)  \in N_{\mathcal{C}/M}|_{\Sigma
}\left(  z\right)  $. Let
\[
T_{\gamma_{u}\left(  t,z\right)  }:N_{\mathcal{C}/M}|_{\gamma_{u}\left(
0,z\right)  }\longrightarrow N_{\mathcal{C}/M}|_{\gamma_{u}\left(  t,z\right)
}%
\]
be the parallel transport \ along the curve $\gamma_{u}\left(  s,z\right)  $
for $0\leq s\leq t$ with respect to the connection of $N_{\mathcal{C}/M}$
induced from the metric $g$ on $M$. We define the exponential-like map
$\widetilde{\exp}$ as follows:%
\[%
\begin{array}
[c]{ccc}%
\widetilde{\exp}:\mathbb{S=S}^{+}\oplus\mathbb{S}^{-} & \longrightarrow & M\\
V\left(  t,z\right)  =\left(  u\left(  t,z\right)  ,v\left(  t,z\right)
\right)  &  & \exp_{\gamma_{u\left(  t,z\right)  }}^{M}\left(  T_{\gamma
_{u\left(  t,z\right)  }}f\left(  v\left(  t,z\right)  \right)  \right)
\end{array}
\]
where $\exp^{M}$ is the exponential map in $M$. It follows from our
construction that for any $V\in C_{-}^{m}\left(  \mathtt{A}_{\varepsilon
},\mathbb{S}\right)  $ and $x\in\partial\mathtt{A}_{\varepsilon}=\left\{
0,\varepsilon\right\}  \times\Sigma$, $\widetilde{\exp}V$ \emph{satisfies the}
\emph{boundary condition} $\widetilde{\exp}V|_{\partial\mathtt{A}%
_{\varepsilon}}$ $\subset C_{0}\cup C_{\varepsilon}$, because
\[
\left(  \widetilde{\exp}V\right)  \left(  x\right)  =\exp_{\gamma_{u}\left(
x\right)  }^{M}\left(  0\right)  =\gamma_{u}\left(  x\right)  \subset
\varphi\left(  \left\{  0,\varepsilon\right\}  \times C\right)  =C_{0}\cup
C_{\varepsilon}\text{.}%
\]
By the definition of $\widetilde{\exp}$, it is easy to see $\widetilde{\exp
}|_{\mathtt{A}_{\varepsilon}}=\varphi$, because $\mathtt{A}_{\varepsilon}$ is
the zero section of $\mathbb{S}$. Hence on base directions of \ $\mathbb{S}%
\rightarrow\mathtt{A}_{\varepsilon}$, $d\widetilde{\exp}|_{\mathtt{A}%
_{\varepsilon}}=d\varphi|_{\mathtt{A}_{\varepsilon}}$, especially
\[
d\widetilde{\exp}|_{\left\{  0\right\}  \times\Sigma}=id:T\Sigma\rightarrow
T\Sigma\text{ and }d\widetilde{\exp}|_{\left\{  0\right\}  \times\Sigma}%
:\frac{\partial}{\partial x_{1}}\rightarrow n\left(  z\right)  .
\]
On fiber directions, that $d\widetilde{\exp}|_{\left\{  0\right\}
\times\Sigma}=\left(  id,f\right)  :N_{\Sigma/C}\oplus\wedge_{\mathbb{C}%
}^{0,1}\left(  N_{\Sigma/C}\right)  \rightarrow N_{\Sigma/C}\oplus
N_{\mathcal{C}/M}|_{\Sigma}$ follows from $d\exp^{M}\left(  0\right)  =id$.

\bigskip{\small Addresses:}

{\small Naichung Conan Leung}

{\small The Institute of Mathematical Sciences, The Chinese University of Hong
Kong, Shatin, Hong Kong}

{\small Email: leung@math.cuhk.edu.hk}

\bigskip

{\small Xiaowei Wang}

{\small Department of Mathematics, Rutgers University-Newark, Newark, NJ
07102}

{\small Email: xiaowwan@rutgers.edu }

\bigskip

{\small Ke Zhu}

{\small Department of Mathematics, University of Minnesota, Minneapolis, MN
55455}

{\small Current: Department of Mathematics, Harvard University, Cambridge, MA
02138}

{\small Email: kzhu@math.harvard.edu}
\end{document}